\documentclass[a4paper, 11pt, twoside]{article}
\usepackage{appendix}
\usepackage{amsmath,amsthm,empheq}
\usepackage{amssymb,latexsym}
\usepackage{mathrsfs}
\usepackage{enumerate}
\usepackage{graphicx}
\usepackage{float}
\usepackage[colorlinks,citecolor=blue,dvipdfm]{hyperref}
\usepackage[numbers,sort&compress]{natbib}
\usepackage{indentfirst}

\usepackage{mathrsfs}

\DeclareMathAlphabet{\mathpzc}{OT1}{pzc}{m}{it}

\newtheorem{theorem}{Theorem}[section]
\newtheorem{corollary}{Corollary}[section]
\newtheorem{lemma}{Lemma}[section]
\newtheorem{proposition}{Proposition}[section]
\newtheorem{definition}{Definition}[section]
\newtheorem{remark}{Remark}[section]

\newtheorem{example}{Example}[section]

\theoremstyle{definition} \theoremstyle{remark}
\numberwithin{equation}{section}
\allowdisplaybreaks

\begin{document}
	\markboth{X. Huang, L. Peng, J. C. Pozo, Y. Zhou}{Stochastic time nonlocal telegram equations}
	
	\date{}
	\baselineskip 0.22in
	\title{{\bf Well-posedness of stochastic time-nonlocal telegraph equations with H\"{o}lder diffusion coefficient: hereditary phase-space lifting and novel generalized coupling method}}
	
	\author{Xi Huang$^{1}$, Li Peng$^{1}$\thanks{\footnotesize {Corresponding author.} }, Juan Carlos Pozo$^{2}$, Yong Zhou$^{1,3}$ \\[1.8mm]
		\footnotesize  {$^1$Faculty of Mathematics and Computational Science, Xiangtan University,}\\
		\footnotesize  {Hunan 411105, China}\\
		\footnotesize  {$^2$Acad\'{e}mico Instituto de Ciencias de la Ingenier\'{\i}a,
			Universidad de O'Higgins, Rancagua, Chile}\\
		\footnotesize  {$^3$Faculty of Information Technology, Macau University of Science and Technology,}\\
		\footnotesize  {Macau 999078, China}
	}
	
	\maketitle
	
	\begin{abstract}
		We consider the initial-boundary value problem for the stochastic time-nonlocal telegraph equation with $(\mathcal{PC}_\varepsilon)$-type kernel $a$:
		\begin{align*}
			\gamma \partial_t \left( a \ast \partial_t (a \ast v)\right) =\Delta v-\partial_t (a \ast v)+ \Psi(v)+ \Phi(v) \frac{\mathrm{d}W(t)}{\mathrm{d}t},
		\end{align*}
		where $W$ is a space-time Gaussian white noise, $\Psi$ satisfies a linear growth condition, and $\Phi$ is H\"{o}lder continuous and uniformly nondegenerate. This model characterizes high-frequency signal propagation in small-scale systems under stochastic fluctuations. We develop a new hereditary phase-space lifting framework for time-nonlocal telegraph equations. In addition, we propose a novel generalized coupling framework, which features a new construction of the damping control term for the velocity. Based on these analytic tools, we prove the first results on weak existence and uniqueness in law for mild solutions, valued in $L_{loc}^2(\mathbb R_+; H^{\delta})$, to the stochastic nonlocal telegraph equation. The regularity index $\delta$ can be arbitrarily close to $\min\{\frac{1}{2},\frac{\varepsilon}{2+\varepsilon} \}$ from below, and the admissible lower bound of the H\"{o}lder exponent $\kappa$ is quantitatively determined by the integrability exponent $\varepsilon$(see \eqref{2.1}). For the corresponding IBVP of the stochastic damped wave equation, obtained by replacing $a$ with the Dirac measure $\delta_0$, $\kappa$ can be improved to any value in $(\frac{3}{5}, 1]$. More significantly, the generalized coupling framework also handles low-regularity nonlinearities depending on both displacement and velocity. \\[2mm]
	\end{abstract}
	
	\tableofcontents
	
	\section{Introduction}
	
	The telegraph equation is a classical hyperbolic model for signal transmission in elastic media and pulse propagation in electrical circuits. It is also widely employed in microdevices and high-frequency transmission systems.
	While the classical telegraph equation relies solely on a local dissipative differential term, it fails to account for memory effects in media with delayed response. In recent years, the time-nonlocal telegraph equation, which incorporates the convolution operator
	$a \ast \cdot$ with a memory kernel to model hysteretic dynamics, has accordingly emerged as a central subject in the study of nonlocal partial differential equations. For example,
	Pozo and Vergara \cite{J. C. Pozo 20} introduced a family of balance equations governing transmission lines with generalized memory-type constitutive effects
	\begin{align*}
		&\partial_{x} v(t,x)+\frac{D_0 h_0}{W_0} \partial_{t}(\mathrm{d} \vartheta_1 \ast i(\cdot,x))(t)+\frac{h_0}{A_0}\left(\mathrm{d} \overline{\vartheta}_{R} \ast i(\cdot,x)\right)(t)=0,\\
		&\partial_{x} i(t,x)+\frac{A_0}{h_0} \partial_{t}(\mathrm{d} \vartheta_2 \ast v(\cdot,x))(t)+\frac{A_0}{h_0}\left(\mathrm{d} \vartheta_{G} \ast v(\cdot,x)\right)(t)=0.
	\end{align*}
	Here $v(t,x)$ and $i(t,x)$ are the voltage and current at position $x$ and time $t$, respectively. The constants $D_0, h_0, W_0, A_0$ are fixed physical parameters. The measures $\mathrm{d}\vartheta_1$, $\mathrm{d}\vartheta_2$, and  $\mathrm{d}\vartheta_G$ correspond to the magnetic permeability, dielectric constant, and parallel leakage conductance, respectively, while
	$\mathrm{d}\bar{\vartheta}_R$ is the measure of inverse conductivity for the series resistance. For $\vartheta \in BV_{loc}(\mathbb R_+)$, the convolution is defined by $(\mathrm{d}\vartheta\ast v)(t) := \int_0^t v(t-\zeta)\, \mathrm{d}\vartheta(\zeta)$. Combining the above relations yields the following generalized time-nonlocal telegraph equation
	\begin{align}\label{1.1}
		\begin{aligned}
			\frac{A_0D_0}{W_0} \partial_t \left( \mathrm{d}\vartheta_1 \ast \partial_t (\mathrm{d}\vartheta_2 \ast v) \right)&=\partial_{x}^2v - \frac{A_0D_0}{W_0} \partial_t (\mathrm{d}\vartheta_1 \ast \mathrm{d}\vartheta_G \ast v )
			\\
			&\quad -\left( \mathrm{d}\bar{\vartheta}_R \ast \partial_t(\mathrm{d}\vartheta_2\ast v )\right)- (\mathrm{d}\bar{\vartheta}_R \ast \mathrm{d}\vartheta_G \ast v ).
		\end{aligned}
	\end{align}
	
	In \eqref{1.1}, $\mathrm{d}\vartheta_1$ and $\mathrm{d}\vartheta_2$ independently characterize the inductive and capacitive storage effects in the medium. Since their relative strength determines the signal propagation behavior, it is natural to ask how one should formulate a telegraph equation that quantifies the relative intensity of $\mathrm{d}\vartheta_1$ against $\mathrm{d}\vartheta_2$?
	
	In the high-frequency transmission regime, the leakage resistance is negligible, so we set $\mathrm{d}\vartheta_G \equiv 0$.
	Since ordinary resistance exhibits only instantaneous, non-memory response, we normalize $\mathrm{d}\bar{\vartheta}_R =\delta_0(t)\,\mathrm{d}t$. To describe the relative strength between $\mathrm{d}\vartheta_1$ and  $\mathrm{d}\vartheta_2$, we assume a linear scaling  $\mathrm{d}\vartheta_1 = \alpha \mathrm{d} \vartheta_2$. To isolate the memory effect, we take the capacitive contribution as $\mathrm{d}\vartheta_2(t) = a(t)\,\mathrm{d}t$ with $a \in L_{loc}^1(\mathbb R_+)$. Under these physical simplifications, the constitutive relations in \eqref{1.1} take the form
	\begin{align*}
		\mathrm{d}\vartheta_1(t) = \alpha a(t)\mathrm{d}t, \ \mathrm{d}\vartheta_2(t) = a(t)\mathrm{d}t, \ \mathrm{d}\bar{\vartheta}_R = \delta_0, \ \mathrm{d}\vartheta_G \equiv 0.
	\end{align*}
	Set $\gamma = \frac{AD}{W_0} \alpha $, which quantifies the intensity of $\mathrm{d}\vartheta_1$ relative to
	$\mathrm{d} \vartheta_2$. With the above constitutive relations, we derive the homogeneous time-nonlocal telegraph equation
	\begin{align*}
		\gamma \partial_t \left(a \ast \partial_t (a \ast v(\cdot,x)) \right)
		(t)+ \partial_t(a\ast v(\cdot,x) )(t) -\partial_{x}^2 v(t,x)= 0.
	\end{align*}
	
	Prior work has largely focused on the homogeneous time-nonlocal telegraph equation with $a$ satisfying $ (\mathcal{PC})$, exploring the connection between its fundamental solution and stochastic processes. The condition $(\mathcal{PC})$ is stated as follows:
	\begin{enumerate}
		\item[$(\mathcal{PC})$] The function $a \in L^1_{loc}(\mathbb R_+)$ is nonnegative and nonincreasing, and there exists a function $b \in L_{\text{loc}}^{1}(\mathbb{R}_+)$ such that
		$(a \ast b)(t) = 1$ for $t>0$.
	\end{enumerate}
	For instance, Pozo and Vergara \cite{J. C. Pozo 20} were the first to prove that the fundamental solution of the initial value problem on the whole real line
	\begin{align*}
		\partial^2_t \left(a \ast a \ast v(\cdot,x) \right)
		(t)+ \partial_t(a\ast v(\cdot,x) )(t) -\partial_{x}^2 v(t,x)= 0, \quad t>0, \ x \in \mathbb R
	\end{align*}
	coincides with the probability density of a certain stochastic process $X(t)$. By invoking the Karamata-Feller Tauberian theorem (\cite[Chapter XIII]{W. Feller}), they further characterized the asymptotic behavior of the variance of this process over different time scales. Subsequently,
	Alegr\'{i}a and Pozo \cite{F. Alegria} devised a systematic procedure for constructing $(\mathcal{PC})$ kernels for which the associated variance $X(t)$ has variance growing sublinearly or logarithmically.  More recently, Alegr\'{i}a and Pozo \cite{F. Alegria 23} examined the non-Markovian telegraph process $X_a(t)$ on $\mathbb R^N$, whose density satisfies
	\begin{align*}
		\begin{cases}
			\partial^2_t \left(a \ast a \ast (v(\cdot,x)-\delta_0) \right)
			(t)+ \partial_t(a\ast (v(\cdot,x)-\delta_0) )(t) -\Delta v(t,x)= 0, \quad t>0, \ x \in \mathbb R^N,\\
			v(0,x)=\delta_0(x), \quad x \in \mathbb R^N.
		\end{cases}
	\end{align*}
	They proved that the moments of $X_a(t)$
	satisfy the Carleman condition, and further demonstrated that the probability distribution of this process can be obtained via a subordination transformation of the classical telegraph process.
	
	Even though the nonlocal telegraph equation admits a reformulation as a Volterra-type integral equation, the kernel of the resulting integral equation does not necessarily satisfy the complete positivity condition (cf. \cite{J. C. Pozo 20, X. Huang 1}). Hence the standard subordination framework is not directly available. To resolve this obstacle, we developed in our previous work \cite{X. Huang 1} a novel approach to constructing completely positive kernels. Based on this construction, we showed that, on $\mathbb R^3$, the solution operator for the time-nonlocal telegraph equation with measure-valued kernels (including $(\mathcal{PC})$-type kernel) is given by a subordination of the operator
	$\cos(\theta(-\Delta)^{\frac{1}{2} })$,  the solution operator of the wave equation. This advance enables us to extend the study of nonlocal telegraph equations from linear dynamics to complex nonlinear dynamics. Although the above subordination representation is highly efficient for deriving deterministic a priori estimates of the solution operator, the operator
	$\cos(\theta(-\Delta)^{\frac{1}{2} })$ itself lacks any compactifying effect with respect to the time variable. This poses a substantial obstacle to the compactness analysis of the subordinated solution operator. Since compactness is essential for well-posedness in stochastic PDEs, new analytic tools are urgently needed to resolve this difficulty.
	
	For the damped wave equation, which is a special case of \eqref{1.1},  it is a well-established and effective approach to construct the first-order evolutionary state variable $(v, \partial_t v)$ and rewrite the equation as a first-order evolution system via phase-space lifting. This strategy has proven effective in simultaneously deriving a priori estimates and establishing compactness, as evidenced by the results in \cite{S. Cerrai, S. Cerrai 1}. For the time-nonlocal telegraph equation, however, the same state variables $(v, \partial_t v)$ are inappropriate, since they are incapable of encoding the memory property inherent in the model, i.e., the dependence of the system evolution on the whole history. A natural question therefore arises: can we construct a phase-space lifting framework for the nonlocal telegraph equation that both simplifies the analysis and captures the memory structure? In this paper, we provide an affirmative answer to this question.
	
	For the treatment of low-regularity nonlinearities, particularly those that are only H\"{o}lder continuous, a generalized coupling framework was introduced in \cite{A. Kulik}. Its core idea is to construct a generalized coupling via stochastic control and then reduce it to a classical coupling, where the shift in marginal distributions induced by the control is compensated appropriately. This framework was later applied by Han \cite{Y. Han 1}  to prove weak uniqueness and exponential ergodicity for infinite-dimensional stochastic heat equations with H\"{o}lder coefficients, as well as for their Burgers-type generalizations.
	Since this generalized coupling approach is characterized by the introduction of displacement damping in the construction of the control term, we shall call it the generalized coupling framework with displacement damping, abbreviated as the GC-D method.
	
	To our knowledge, however, no coupling analysis has yet been developed for nonlocal telegraph equations that can accommodate both nonlinearities and low-regularity coefficients, such as  H\"{o}lder continuous ones. Moreover, we observe that the existing GC-D method suffers from an inherent limitation: it fails to extract damping decay factors in the moment estimates for the velocity component of the stochastic damped wave equation, thereby hindering a unified treatment of nonlinearities that depend on both displacement and velocity. Detailed arguments are provided in Section \ref{sec 7}.
	
	Motivated by the central gaps in the existing literature outlined above, we study the following IBVP for the stochastic time-nonlocal telegraph equation
	\begin{align}\label{1.2}
		\begin{cases}
			\gamma \partial_t \left( a \ast \partial_t(a \ast v(\cdot,x))\right)(t)=\Delta v(t,x)-\partial_t \left(a \ast v(\cdot,x) \right)(t) \\
			\quad\quad\quad\quad+ \Psi(v(t,x) ) + \Phi(v(t,x)) \frac{\mathrm{d}W(t)}{\mathrm{d}t} , \quad t>0, \ x \in  (0,1), \\
			v(t,0) =v(t,1) = 0, \quad t>0,\\
			(a \ast v(\cdot,x))(0) =u_0(x), \ \left(a \ast \partial_t(a \ast v(\cdot,x))\right)(0)=u_1(x), \quad x \in (0,1).
		\end{cases}
	\end{align}
	where $\gamma >0$, $\Delta$ is the Laplacian, $\Psi$ and $\Phi$ are low-regularity drift and diffusion coefficients, respectively, with assumptions detailed in $(\mathcal M_1) \sim (\mathcal M_4)$ of Section \ref{sec 2}, and $W$ is a $(\mathcal{F}_t)_{t \geq 0}$-cylindrical Wiener process over $H$. The kernel $a$ satisfies:
	\begin{itemize}
		\item[$(\mathcal{PC}_\varepsilon)$] The function $a \in L^1_{loc}(\mathbb R_+)$ is nonnegative and nonincreasing, and there exists a nonincreasing function $b \in L_{\text{loc}}^{2+\varepsilon}(\mathbb{R}_+)$ with some $\varepsilon > 0$ such that
		$(a \ast b)(t) = 1$ for $t>0$.
	\end{itemize}
	Clearly, $(\mathcal{PC}_\varepsilon)\subset(\mathcal{PC})$ and $(\mathcal{PC}_{\varepsilon_2}) \subset (\mathcal{PC}_{\varepsilon_1})$ for $\varepsilon_1 < \varepsilon_2$. The model captures memory effects, nonlinear displacement evolution, and random fluctuations, thereby providing a precise description of high-frequency signal transmission in small-scale systems under stochastic perturbations.
	
	We invoke the theory of completely positive kernel Volterra equations to establish a hereditary phase-space lifting framework adapted to  weakly regular kernels of type
	$(\mathcal{PC}_\varepsilon)$. In contrast to the classical instantaneous phase point construction $(v,\partial_t v)$ for integer-order memoryless systems, which relies only on instantaneous quantities and yields finite-dimensional Markovian dynamics, we introduce a new lifted phase $Y(t) = (v(t), \partial_t (a \ast v)(t))$. This lifted phase fully encodes the historical dependence of the entire trajectory from the initial time onward, and is fundamentally different from the classical one: the evolution equation of $Y(t)$ still contains hereditary convolution terms and cannot be reduced to Markovian dynamics determined solely by the current state.
	
	In addition, we develop a novel generalized coupling framework based on velocity damping, which we refer to as the GC-V method. The way in which velocity damping is incorporated in this framework enables the simultaneous extraction of high-order damping decay factors in the moment estimates of both the displacement and velocity components (see Lemmas \ref{lem 5.1} and \ref{lem 5.2}). This feature provides a crucial analytical foundation for treating low-regularity nonlinearities that depend on both the displacement and velocity fields.
	
	Building upon the hereditary phase-space lifting and the new GC-V method, we establish a complete well-posedness result for \eqref{1.2}: for each $T>0$, \eqref{1.2} admits a weak mild solution as $L^2([0,T]; H^\delta)$-valued process, and the solution is unique in the sense of law. The Sobolev regularity index $\delta$ of the solution can be chosen arbitrarily close to  $\min\{\frac{1}{2},\frac{\varepsilon}{2+\varepsilon} \}$ from below. Moreover, our analysis shows that the admissible lower bound of the H\"{o}lder exponent $\kappa$ for the diffusion coefficient in \eqref{1.2} depends on the integrability exponent $\varepsilon$ of the kernel $b$, and decreases as
	$\varepsilon$ increases. This dependence possibly reveals an underlying mechanism by which the regularity of the memory kernel affects the admissible regularity of the noise in the stochastic system.
	
	In particular, when the kernel $a(t)$ is replaced by the Dirac measure $\delta_0$, \eqref{1.2} degenerates into the IBVP of the stochastic damped wave equation with
	a H\"{o}lder continuous diffusion coefficient. In this case, the admissible lower bound of the H\"{o}lder exponent $\kappa$ for the diffusion coefficient can be improved to $\frac{3}{5}$ (excluding the endpoint), see Section \ref{sec 7}. More significantly, even when the nonlinear term couples the displacement $v$ and  the velocity $\partial_t v$, our GC-V method still yields a complete well-posedness theory (see Corollary \ref{cor 7.1}). These results show that the GC-V method is a broadly applicable coupling analysis framework with significant advantages.
	
	The paper is structured as follows. Section \ref{sec 2} presents the preliminaries: notation, function spaces, assumptions, stochastic integration, and completely positive kernels for Volterra equations. Section \ref{sec 3} develops a hereditary phase-space lifting framework for the stochastic nonlocal telegraph equation and a novel representation of mild solutions. Section \ref{sec 4} develops a new extraction procedure for velocity damping decay factors and establishes several deterministic a priori estimates. Section 5 demonstrates how to extract velocity damping decay factors in stochastic integral moment estimates and proves the compactness condition for solution trajectories. Section \ref{sec 6} introduces a new coupling analysis framework, which is then used to prove the existence and uniqueness in law of weak mild solutions. Section \ref{sec 7} provides explicit kernel constructions satisfying $(\mathcal{PC}_\varepsilon)$, highlights the advantages of the GC-V method, and outlines potential extensions.
	
	\section{Preliminaries}\label{sec 2}
	\subsection{Notations and Assumptions}
	The Laplace transform of a function $f_1:(0,\infty) \to \mathbb C$
	is denoted as $\widehat{f}_1(\lambda)$. Recall that whenever $(a,b) \in (\mathcal{PC})$, the Laplace transforms   $\widehat{a}(\lambda)$ and $\widehat{b}(\lambda)$ are well defined for each $\lambda \in \mathbb C$ satisfying $Re(\lambda)>0$.
	
	Let $\{e_n\}_{n=1}^\infty$ stand for an orthonormal basis of $L^2([0,1])$ formed by eigenfunctions associated with $-\Delta$
	subject to the homogeneous boundary condition, i.e.,
	\begin{align*}
		-\Delta e_j = \lambda_j e_j \text{ ~in~ } [0,1],\quad  e_j(0) =e_j(1)=0,
	\end{align*}
	where $\lambda_j =\pi^2 j^2$ with $j \in\mathbb N$. For convenience, we denote $L^2([0,1])=H$. For any $\delta \in \mathbb{R}$, denote $H^\delta([0,1])$ the fractional Sobolev space of order $\delta$ on $[0,1]$, which we shall simply write as
	$H^\delta$. This space is endowed with the norm
	\begin{align*}
		\|f\|_{H^\delta}^2 = \sum_{j=1}^\infty \lambda_j^\delta |\langle f, e_j \rangle|^2,
	\end{align*}
	where $\langle \cdot,\cdot \rangle$ denotes the inner product in $H$. We define the phase space $\mathcal{H}^\delta := H^\delta \times H^{\delta-1}$.  For each $N \in \mathbb N_+$, we denote by $P_N$ the orthogonal projection onto the span of $\{e_j\}_{j \leq N}$.
	
	To simplify the notation, we adopt the following convention throughout this paper. The letter $C$ denotes a generic positive constant. The notation $\lesssim$
	is used to absorb universal constants in
	estimates. Subscripted constants such as
	$C_{c_1,c_2}$ indicate dependence on the parameters $c_1$ and $c_2$. None of these symbols refer to a fixed numerical value. They merely indicate that the corresponding term can be bounded by some constant.
	
	We shall now state the assumptions required for our main results.
	\begin{enumerate}
		\item[$(\mathcal M_1)$] The map $\Psi: L_{loc}^2(\mathbb R_+; H) \to L_{loc}^2(\mathbb R_+; H)$ is casual in the sense that, for any $T>0$ and any $w_1, w_2 \in L_{loc}^2(\mathbb R_+;H)$, if $w_1=w_2$ a.e. on $[0,T]$, then $\Psi(w_1) = \Psi(w_2)$ a.e. on $[0,T]$ in $H$.
		Moreover, for each $T>0$, the induced map $\Psi_T: L^2([0,T]; H) \to L^2([0,T]; H)$ is Borel measurable.
		
		\item[($\mathcal M_2$)] For each $v \in H$, $\Phi(v)$ has a right inverse $\Phi^{-1}(v)$ on $H$ in the sense that $\Phi(v)\Phi^{-1}(v)u = u$ for any $u \in H$,
		with $\Phi^{-1}(v)H \subset H$ and satisfies
		\begin{align*}
			\sup_{v \in H} \|\Phi^{-1}(v)\|_{\mathcal L(H)} \leq C \text{ ~for some~ } C>0.
		\end{align*}
		In addition, $\Phi$ id self-adjoint, in the sense that for any $v \in H$, $\langle \Phi(v)e_i,e_j\rangle =\langle e_i,\Phi(v)e_j\rangle$ for any $i,j\in\mathbb N_+$.
		
		\item[$(\mathcal M_3)$] $\Phi$ admits an extension to a map $\Phi: L_{loc}^2(\mathbb R_+;H)  \to L_{loc}^2(\mathbb R_+; \mathcal L(H))$, and is assumed to be
		$\kappa$-H\"{o}lder continuous in the sense that, for each $T>0$, there exists a constant $C_T$ such that
		\begin{align*}
			\|\Phi(w_1) - \Phi(w_2)\|_{L^2([0,T];{\mathcal L}(H)) } \leq C_T\|w_1 - w_2\|_{L^2([0,T];H)}^\kappa
		\end{align*}
		for any $w_1, w_2 \in L_{loc}^2(\mathbb R_+;H)$, where
		\begin{align}\label{2.1}
			\kappa>\begin{cases}
				1-\frac{\varepsilon^2}{2\varepsilon^2+4\varepsilon+4}, \text{ ~if~ } \varepsilon \in (0,2), \\
				1-\frac{2\varepsilon}{5(2+\varepsilon)}, \text{ ~if~ } \varepsilon \in [2,\infty).
			\end{cases}
		\end{align}
		Additionally, there exists a positive constant
		$C_T'$ depending only on $T$ such that
		\begin{align}\label{2.2}
			\begin{aligned}
				&\quad\|P_N\Phi(w_1) - P_N\Phi(w_2)\|_{L^2([0,T];{\mathcal L}(H)) }\\
				&\leq C_T'  \|\Phi(P_Nw_1) - \Phi(P_Nw_2)\|_{L^2([0,T];{\mathcal L}(H))}, \quad w_1, w_2 \in L_{loc}^2(\mathbb R_+;H),
			\end{aligned}
		\end{align}
		holds uniformly for all sufficiently large $N\in \mathbb N_+$.
		
		\item[$(\mathcal M_4)$] For each $T>0$, there exist some constant $C_{\Psi,T}$ and $C_{\Phi,T}> 0$ such that
		\begin{align*}
			&\|\Psi(v)\|_{L^2([0,T],H)} \leq C_{\Psi,T}(1 + \|v\|_{L^2([0,T];H)}).\\
			&\|\Phi(v)\|_{L^2([0,T],\mathcal L(H) )} \leq C_{\Phi,T}(1 + \|v\|_{L^2([0,T];H)}).
		\end{align*}
	\end{enumerate}
	
	\begin{remark}\label{rem 2.1}
		We can also provide an example of spectral-type multiplicative diffusion coefficient that satisfies all the assumptions stated above. For each
		$v\in H$, define $\Phi(v) \in \mathcal L(H)$ by its spectral decomposition
		\begin{align*}
			\Phi(v)u = \sum_{j=1}^\infty f_j(\langle v,e_j \rangle)\langle u,e_j \rangle e_j, \quad u \in H,
		\end{align*}
		where the functions $f_j:\mathbb R \to \mathbb R\setminus (-C_0,C_0)$(with $C_0>0$ independent of $j$) are uniformly H\"{o}lder continuous in $j$ and satisfy a uniform linear growth condition: there exists constant $C>0$ independent of $j$ such that
		\begin{align*}
			|f_j(y_1)-f(y_2)| \leq C |y_1-y_2|^\kappa, \ |f_j(y)|\leq C(1+|y|),
		\end{align*}
		where $\kappa$ as in \eqref{2.1}.
	\end{remark}
	
	\subsection{Infinite I'to Integral and stochastic convolution}
	Suppose $(U, \langle \cdot, \cdot \rangle_U, \|\cdot\|_U)$ and $(V, \langle \cdot, \cdot \rangle_V, \|\cdot\|_V)$ are two separable Hilbert spaces, and $W$ is a $(\mathcal{F}_t)_{t \geq 0}$-cylindrical Wiener process over $U$.
	We write $\mathcal L_2(U, V)$ for the separable space of all Hilbert-Schmidt operators mapping $U$ into $V$, which is endowed with the canonical Hilbert-Schmidt norm $\|\varTheta\|_{{\mathcal L}_2(U,V)}^2 = \sum_{j=1}^{\infty} \|\varTheta e_j^U\|_V^2$.
	
	Consider a random mapping $\varphi : [0, T] \times \Omega \to \mathcal L_2(U, V)$ that is jointly measurable and $(\mathcal{F}_t)_{t \geq 0}$-adapted, and fulfills the integrability condition  $\int_0^T \mathbb{E} \left[ \|\varphi(\zeta)\|_{\mathcal L_2(U,V)}^2 \right]\,\mathrm{d}t < \infty$. Then the stochastic integral defined by
	\begin{align*}
		t \mapsto \int_0^t \varphi(\zeta)\, \mathrm{d}W(\zeta) := \sum_{j=1}^{\infty} \int_0^t \varphi(\zeta) e_j^U\, \mathrm{d}\beta_j(\zeta)
	\end{align*}
	yields a continuous martingale with values in $V$. Here $(\beta_j(t))_{t \geq 0}$ is a sequence of independent standard Brownian motions.
	
	We use $\mathcal L(U, V)$ to represent the space of bounded linear operators from $U$ to $V$, equipped with the strong operator topology. Given exponents $q \in [1, \infty)$, $p \in [1, \infty)$ and $T > 0$,
	we introduce the Banach space $\mathcal{X}^{q,p}([0, T]; V)$. Its elements are all measurable and  $(\mathcal{F}_t)_{t \in [0,T]}$-adapted  random processes $\varphi : [0, T] \times \Omega \to V$ with finite norm given by
	\begin{align*}
		\|\varphi\|_{\mathcal{X}^{q,p}([0,T]; V)} = \left( \int_0^T \left( \mathbb{E} \left[ \|\varphi(t)\|_V^q \right] \right)^{\frac{p}{q}} dt \right)^{\frac{1}{p}}.
	\end{align*}
	
	Throughout this paper, we frequently utilize the convolution between operator families and stochastic integrals. For a given operator family $(Q(t))_{t \in [0, T]} \subset \mathcal L(V)$, we define the stochastic convolution operation as
	\begin{align*}
		(Q \ast  \varphi \mathrm{d}W)(t) := \int_0^t Q(t - \zeta)\varphi(\zeta)\, \mathrm{d}W(\zeta), \quad t \in [0, T].
	\end{align*}
	We next provide a sufficient criterion guaranteeing the associativity property of stochastic convolutions combined with deterministic operator convolutions.
	\begin{lemma}\label{lem 2.1}
		(\cite[Lemma 2.1]{L. A. Bianchi}) Let $V_0, V_1$ be separable Hilbert spaces. Assume that $\varphi \in \mathcal{X}^{2,2}([0, T]; \mathcal L_2(U, V))$,  $G \in L^2([0, T]; \mathcal L(V, V_0))$, and $F \in L_{loc}^1(\mathbb{R}_+; \mathcal L(V_0, V_1))$. Then it holds that
		\begin{align*}
			F\ast(G\ast \varphi \mathrm{d}W) = (F\ast G)\ast \varphi \mathrm{d}W,
		\end{align*}
		where all deterministic and stochastic convolution terms appearing in the equality are well-defined.
	\end{lemma}
	
	\subsection{Scalar volterra equations with complete positive kernels}
	Let us revisit fundamental results of the theory of Volterra integral equations. Given $\nu \in \mathbb{C}$ and  $l \in L_{loc}^1(\mathbb R_+)$,
	the scalar resolvent function $s_\nu$ and the integrated scalar resolvent function $r_\nu$ are uniquely defined as the solutions to the following Volterra equations
	\begin{align}
		&s_\nu(t) + \nu(l\ast s_\nu)(t)= 1, \quad t \geq 0. \label{2.3} \\
		&r_\nu(t) + \nu(l\ast r_\nu)(t)= l(t), \quad t>0, \label{2.4}
	\end{align}
	respectively, see \cite[Chapter 2, Theorem 3.1]{G. Gripenberg}.
	Consider a forcing term $f$ belonging to $L_{loc}^p(\mathbb{R}_+)$ for some exponent $p\geq 1$. We focus on the general Volterra integral equation
	\begin{align}\label{2.5}
		v_\nu(t) + \nu(l\ast v_\nu)(t) = f(t), \quad t>0.
	\end{align}
	The unique solution to the above can be explicitly formulated as
	\begin{align}\label{2.6}
		v_\nu(t) = f(t) - \nu (r_\nu \ast f)(t), \quad t>0,
	\end{align}
	which satisfies $v_\nu \in L_{loc}^p(\mathbb R_+)$, see \cite[Chapter 2, Theorem 3.5]{G. Gripenberg}.
	
	We next introduce the notion of complete positivity for integral kernels. A kernel $l$ is said to be completely positive whenever the associated resolvent families $s_\nu$ and $r_\nu$ are nonnegative for $\nu \geq 0$, see \cite[Definition 1.1]{P. Clement 81}.
	An equivalent characterization of complete positivity has been developed in \cite[Theorem 2.2]{P. Clement 81}. Specifically, the kernel $l$ fulfills the complete positivity condition if and only if there exist a constant $k_0\geq 0$ and a nonincreasing nonnegative function $k_1 \in L_{loc}^1(\mathbb{R}_+)$ such that
	\begin{align}\label{2.7}
		k_0l(t) + (k_1\ast l)(t) = 1, \quad t>0.
	\end{align}
	In particular, the structural assumption $(a,b) \in (\mathcal{PC}_\varepsilon)$ imposed throughout this work guarantees the complete positivity of the kernel $b$. It is worth emphasizing that any completely positive kernel is necessarily nonnegative, as verified in \cite[Proposition 2.1(i)]{P. Clement 81}. Further equivalent criteria characterizing complete positivity can be found in \cite[Proposition 4.5]{J. Pruss}.
	
	We next introduce a key decomposition of $k_1$ for subsequent resolvent analysis. Since $k_1$ is nonnegative and nonincreasing, it can be decomposed as $k_1(t)=k_2(t)+k_\infty$. Here $k_\infty\ge0$ denotes a constant offset, and $k_2$ denotes a nonnegative, nonincreasing function decaying to zero as $t\to\infty$.
	\begin{proposition}\label{prop 2.1}
		Assume that $l$ is completely positive. Then the following hold:
		\begin{itemize}
			\item[(i)] Let $\nu \in \mathbb C$ satisfies $Re(\nu) \geq -k_\infty$, then $s_\nu(t)$ and $r_\nu(y)$ can be represented by
			\begin{align*}
				s_\nu(t) = -\int_0^\infty e^{-\nu\theta}  \varpi(t,\mathrm{d}\theta), \ r_\nu(t)=\int_0^\infty e^{-\nu\theta} \eta(t,\mathrm{d}\theta)
			\end{align*}
			where $-\varpi(t,\theta)$ and $\eta(t,\mathrm{d}\theta)$ are postive finite measures satisfies
			\begin{align*}
				-\widehat{\varpi}(\lambda,\mathrm{d}\theta)= \frac{1}{\lambda\widehat{l}(\lambda)} e^{ -\frac{\theta} {\widehat{l}(\lambda)} }\mathrm{d}\theta, \ \widehat{\eta}(\lambda,\mathrm{d}\theta) = e^{ -\frac{\theta} {\widehat{l}(\lambda)} }\mathrm{d}\theta.
			\end{align*}
			Typically, it holds that
			\begin{align*}
				-\int_0^\infty \varpi(t,\mathrm{d}\theta) =1, \ \int_0^\infty \eta(t,\mathrm{d}\theta) =l(t).
			\end{align*}
			\item[(ii)] Let $\nu \in \mathbb C$, the function $s_\nu(t)$ has the following representation
			\begin{align*}
				s_\nu(t) = 1 - \nu \int_0^t r_\nu(\zeta)\, \mathrm{d}\zeta = k_0 r_\nu(t)+ (k_1 \ast r_\nu)(t)=(\mathrm{d}k \ast r_\nu)(t), \quad t \geq 0,
			\end{align*}
			which implies that $s_\nu$ is continuous on $\mathbb R_+$ and differentiable on $(0,\infty)$.
			\item[(iii)] For $\nu\geq0$ and $t\geq 0$, it holds
			\begin{align*}
				0\leq s_\nu(t) \leq \frac{1}{1+\nu (1\ast l)(t)}.
			\end{align*}
			\item[(iv)] For $\nu\geq0$ and $t>0$, if $l$ is nonincreasing, then
			\begin{align*}
				0\leq r_\nu(t) \leq \frac{ l(t)}{1+\nu (1\ast l)(t)}.
			\end{align*}
		\end{itemize}
		\begin{proof}
			The proof of (i) can be found in \cite[Lemma A.1]{X. Huang 1}. For part (ii), the case
			$\nu \geq 0$ follows from
			\cite[Proposition 2.1]{X. Huang};
			the result can be extended to
			$\nu \in \mathbb C$. As for part (iii), since $s_\nu(\cdot)$ is nonincreasing, the desired estimate follows from applying the bound $(l \ast s_\nu)(t)\leq s_\nu(t)(1\ast l)(t)$ to \eqref{2.3}.
			The proof of conclusion (iv) adopts an argument similar to the one presented in \cite[Lemma 5.4]{J.C. Pozo}, hence we omit the detailed derivation for brevity.
		\end{proof}
	\end{proposition}
	
	In the sequel, we shall frequently make use of Proposition \ref{prop 2.1}. For the sake of brevity, we will not always explicitly refer to it when applying its assertions.
	
	\section{Novel solution representation under hereditary phase-space lifting}\label{sec 3}
	In this section, we construct a hereditary phase-lifting framework for the stochastic time-nonlocal telegraph equation and derive a novel solution representation. Consider the following problem:
	\begin{align}\label{3.1}
		\begin{cases}
			\partial_t (a \ast v) = w,  \\
			\partial_t (a \ast w) = \frac{1}{\gamma} \left[\Delta v - (\xi+1)w
			+\Psi(v)
			+  \Phi(v) \frac{\mathrm{d}w}{\mathrm{d}t} \right],
		\end{cases}
	\end{align}
	subject to the initial conditions
	\begin{align*}
		(a \ast v(\cdot,x))(0) = u_0(x) , \ (a \ast w(\cdot,x))(0) = u_1(x).
	\end{align*}
	We introduce the linear operator $\mathcal {A}_{\gamma,\xi}:D(\mathcal{A}_{\gamma,\xi}) = \mathcal H^{\delta+1} \to \mathcal H^{\delta}$ by
	\begin{align*}
		\mathcal A_{\gamma,\xi}(v,w) = ( w, \gamma^{-1} \Delta v - \gamma^{-1} (1+\xi)w )
	\end{align*}
	with $\xi \geq 0$. We also define the projection operators $\mathcal I_1:\mathcal H^\delta\to H^\delta$ and $\mathcal I_2:\mathcal H^\delta\to H^{\delta-1}$ via
	\begin{align*}
		\mathcal{I}_1(v,w) = v, \ \mathcal{I}_2(v,w) = w.
	\end{align*}
	Finally, we define the map $J_\gamma: H^\delta \to \mathcal H^{\delta}$ by
	\begin{align*}
		J_\gamma v = (0, \gamma^{-1}v ).
	\end{align*}
	
	Formally, we define a ``phase point'' by
	$Y(t) = (v(t), \partial_t (a \ast  v)(t) )$, where the first component describes the voltage at position $x$ at time $t$, while the second component accounts for the history-dependent time-varying velocity of the voltage at $(t,x)$.
	Then \eqref{3.1} can be rewritten as
	\begin{align}\label{3.2}
		\mathrm{d}(a \ast Y)(t) = \mathcal A_{\gamma,\xi} Y(t)\mathrm{d}t + J_\gamma \Psi(\mathcal I_1 Y(t))\mathrm{d}t + J_\gamma \Phi(\mathcal I_1 Y(t)) \mathrm{d}W(t)
	\end{align}
	with the initial condition $(a \ast Y)(0)= (u_0,u_1):=X(0)$. We further assume that for each $T>0$, $\Phi(v)\in \mathcal X^{2,2}([0,T];\mathcal L_2(H, H))$, so that $J_\gamma \Phi(v) \in \mathcal{X}^{2,2}([0,T]; \mathcal L_2(\mathcal H,\mathcal H ))$ for each $T>0$ as well. Convolving the first-order system \eqref{3.2} with
	$b$ and using the identity $a \ast b \equiv 1$, we obtain
	\begin{align}\label{3.3}
		\begin{aligned}
			Y(t) &= b(t)X(0) + (b \ast \mathcal A_{\gamma,\xi} Y)(t) + (b \ast  J_\gamma \Psi(\mathcal I_1 Y))(t) \\
			&\quad + (b \ast  J_\gamma \Phi(\mathcal I_1 Y) \mathrm{d}W )(t).
		\end{aligned}
	\end{align}
	Note that, since $b \in L_{loc}^2(\mathbb R_+)$  is deterministic, the stochastic convolution appearing at the end of the above identity is well-defined.
	
	Clearly, $\mathcal A_{\gamma,\xi}$ generates a $C_0$-semigroup on $\mathcal H^\delta$. Consider the Volterra equation
	\begin{align}\label{3.4}
		(v(t),w(t)) = (f_1(t),f_2(t)) + \int_0^t b(t - \zeta)\mathcal A_{\gamma,\xi} (v(\zeta),w(\zeta))\, \mathrm{d}\zeta, \quad t \geq 0,
	\end{align}
	where $b$ is a completely positive function. By \cite[Theorem 4.2]{J. Pruss}, \eqref{3.4} admits a resolvent
	$\{\mathcal S_{\gamma,\xi}(t)\}_{t\geq 0}$ on $\mathcal H^\delta$. According to
	\cite[Definition 1.3]{J. Pruss}, a family $\{S_{\gamma,\xi}(t)\}_{t\geq0} \subset \mathcal{L}(\mathcal H^{\delta})$ is called a resolvent for \eqref{3.4} if:
	\begin{itemize}
		\item[(i)]  $S_{\gamma,\xi}(t)$ is strongly continuous on $\mathbb{R}_+$ and $S_{\gamma,\xi}(0)=I$, where $I:\mathcal H^{\delta} \to \mathcal H^{\delta}$ is given by $I(v,w) =(v,w)$;
		\item[(ii)]$S_{\gamma,\xi}(t)$ commutes with $\mathcal A_{\gamma,\xi}$, i.e., $\mathcal S_{\gamma,\xi}(t)\mathcal{D}(\mathcal A_{\gamma,\xi}) \subset \mathcal{D}(\mathcal A_{\gamma,\xi})$ and $\mathcal A_{\gamma,\xi} S_{\gamma,\xi}(t)(v,w) = S_{\gamma,\xi}(t)\mathcal A_{\gamma,\xi}(v,w)$ for all $(v,w) \in \mathcal{D}(\mathcal A_{\gamma,\xi})$ and $t \geq 0$;
		\item[(iii)] the resolvent equation
		\begin{align*}
			S_{\gamma,\xi}(t)(v,w) = (v,w) + \int_0^t b(t - \zeta)\mathcal A_{\gamma,\xi} S_{\gamma,\xi}(\zeta)(v,w)\, \mathrm{d}\zeta
		\end{align*}
		holds for for all $(v,w) \in \mathcal{D}(\mathcal A_{\gamma,\xi})$ and $t \geq 0$,
	\end{itemize}
	Then, applying the variation of parameters formula to \eqref{3.4}, we obtain
	\begin{align}\label{3.5}
		\begin{aligned}
			Y(t) &= \frac{\mathrm{d}}{\mathrm{d}t}(\mathcal S_{\gamma,\xi} \ast(b X(0) + b \ast J_\gamma \Psi(\cdot, \mathcal I_1 Y)  +b \ast J_\gamma \Phi(\mathcal I_1 Y) \mathrm{d}W )(t).
		\end{aligned}
	\end{align}
	So far, we have derived a representation of solutions to \eqref{1.2}. However, this differential form representation is rather difficult to handle in problems such as moment estimation. Therefore, it is necessary to further transform the above solution representation so as to eliminate its differential structure.
	An interesting idea is as follows. suppose there exists a family
	there exists a family of $\mathcal R_{\gamma,\xi} \in L_{loc}^2(\mathbb R_+, \mathcal L(\mathcal H^\delta))$ such that for a.a $t \geq 0$, $\mathcal R_{\gamma,\xi}(t)$ leaves $\mathcal D(\mathcal A_{\gamma,\xi})$ invariant,
	commutes with $\mathcal A_{\gamma,\xi}$ on $\mathcal D(\mathcal A_{\gamma,\xi})$, and satisfies for each $(v,w)$ on $\mathcal D(\mathcal A_{\gamma,\xi})$,
	\begin{align}\label{3.6}
		\mathcal R_{\gamma,\xi}(t)(v,w)= b(t)(v,w)+ \int_0^t b(t-\zeta) \mathcal A_{\gamma,\xi} \mathcal R_{\gamma,\xi}(\zeta)(v,w)\, \mathrm{d}\zeta.
	\end{align}
	Then convolving this with $a$ yields for all $(v,w) \in \mathcal{D}(\mathcal A_{\gamma,\xi})$ and $t \geq 0$,
	\begin{align*}
		(a \ast \mathcal R_{\gamma,\xi})(t)(v,w)= (v,w)+ \int_0^t b(t-\zeta) \mathcal A_{\gamma,\xi} (a \ast \mathcal R_{\gamma,\xi})(\zeta)(v,w)\, \mathrm{d}\zeta.
	\end{align*}
	Hence $(a \ast \mathcal R_{\gamma,\xi})(t) = \mathcal S_{\gamma,\xi}(t)$. Substituting this into \eqref{3.5} yields
	\begin{align}\label{3.7}
		\begin{aligned}
			Y(t) &= \frac{d}{\mathrm{d}t}( 1  \ast(\mathcal R_{\gamma,\xi} X(0) + \mathcal R_{\gamma,\xi} \ast J_\gamma \Psi(\mathcal I_1 Y)+\mathcal R_{\gamma,\xi} \ast J_\gamma \Phi(\mathcal I_1 Y) \mathrm{d}W)(t) \\
			&= \mathcal R_{\gamma,\xi}(t)X(0) + \int_0^t \mathcal R_{\gamma,\xi}(t-\zeta) J_\gamma \Psi(\mathcal I_1 Y(\zeta))\, \mathrm{d}\zeta \\
			&\quad+ \int_0^t \mathcal R_{\gamma,\xi}(t-\zeta) J_\gamma \Phi(\mathcal I_1 Y(\zeta))\, \mathrm{d}W(\zeta), \quad \text{ ~a.a.~ } t \in \mathbb R_+.
		\end{aligned}
	\end{align}
	Here, in view of $J_\gamma \Phi(\mathcal I_1) \in \mathcal{X}^{2,2}([0, T]; \mathcal L_2(\mathcal H,\mathcal H ))$ for each $T>0$, together with Lemma \ref{lem 2.1} and the identity $a \ast b \equiv 1$, we have
	\begin{align*}
		&\quad \frac{\mathrm{d}}{\mathrm{d}t}(\mathcal S_{\gamma,\xi} \ast (b \ast J_\gamma \Phi(\mathcal I_1 Y ) \mathrm{d}W )  )(t)\\
		&=\frac{\mathrm{d}}{\mathrm{d}t} (\mathcal  (S_{\gamma,\xi} \ast b)  \ast J_\gamma \Phi(\mathcal I_1 Y) \mathrm{d}W )(t)\\
		&=\frac{\mathrm{d}}{\mathrm{d}t} ( (a \ast b \ast \mathcal R_{\gamma,\xi}) \ast J_\gamma \Phi(\mathcal I_1 Y ) \mathrm{d}W )(t)\\
		&=\frac{\mathrm{d}}{\mathrm{d}t} ( (1 \ast \mathcal R_{\gamma,\xi}) \ast J_\gamma \Phi(\mathcal I_1 Y) \mathrm{d}W )(t)\\
		&= \frac{\mathrm{d}}{\mathrm{d}t} ( 1 \ast (\mathcal R_{\gamma,\xi} \ast J_\gamma \Phi(\mathcal I_1 Y ) \mathrm{d}W) )(t)\\
		&=(\mathcal R_{\gamma,\xi} \ast J_\gamma \Phi(\mathcal I_1 Y ) \mathrm{d}W)(t) \text{ ~for} t>0.
	\end{align*}
	In this way, we have derived a mild solution representation for the time-nonlocal telegraph equation. Furthermore, from $\mathcal R_{\gamma,\xi}(t) = (b \ast \mathcal S_{\gamma,\xi})'(t)$ and the uniqueness of $S_{\gamma,\xi}(t)$, it follows that if a solution to \eqref{3.6} exists, it must be unique.
	
	We now prove the existence of $\mathcal R_{\gamma,\xi}(t)$ via the solution of the following ODEs. For notational convenience, for each $j \in \mathbb N_+$, $\gamma>0$,  and $\xi \geq 0$, denote
	\begin{align*}
		&Z_1=Z_{1,\gamma,j,\xi} = \frac{1+\xi-\sqrt{(1+\xi)^2-4\gamma\lambda_j}}{2\gamma},\\ &Z_2=Z_{2,\gamma,j,\xi}=\frac{1+\xi+\sqrt{(1+\xi)^2-4\gamma\lambda_j}}{2\gamma}.
	\end{align*}
	
	\begin{lemma}\label{lem 3.1}
		Assume that $(a,b)\in (\mathcal{PC}_\varepsilon)$. Let $j\in\mathbb N_+$, $\gamma>0$ and $\xi \geq 0$. Consider the following nonlocal ODEs
		\begin{align}\label{3.8}
			\begin{cases}
				\gamma \partial_t \left(a \ast (\partial_t a \ast h)\right)(t) + (1+\xi)\partial_t(a \ast h)(t) + \lambda_j h(t) = m(t), \quad t>0,\\
				(a \ast h)(0)= h_0, \ \left(a \ast (\partial_t a \ast h)\right)(0)=h_1,
			\end{cases}
		\end{align}
		where $m\in L_{loc}^1(\mathbb R_+)$ is a given function and $h_0, h_1\in \mathbb R$. The following assertions hold.
		\begin{itemize}
			\item[(i)] If $1-4\lambda_j\gamma\neq 0$, then its solution is given by
			\begin{align}\label{3.9}
				\begin{aligned}
					h(t) &= (\frac{1}{2} +\frac{1+\xi}{2\sqrt{(1+\xi)^2-4\gamma\lambda_j} } )r_{Z_1}(t) h_0+ (\frac{1}{2} -\frac{1+\xi}{2\sqrt{(1+\xi)^2-4\gamma\lambda_j} } ) r_{Z_2}(t) h_0\\
					&\quad+ \frac{\gamma h_1}{\sqrt{(1+\xi)^2-4\gamma\lambda_j}}(r_{Z_1}(t) -r_{Z_2}(t) )\\
					&\quad + \frac{1}{\sqrt{(1+\xi)^2-4\gamma\lambda_j}}\int_{0}^t ( r_{Z_1}(t-\zeta) - r_{Z_2}(t-\zeta))m(\zeta) \, \mathrm{d}\zeta.
				\end{aligned}
			\end{align}
			\item[(ii)] If $1-4\lambda_j\gamma=0$, then its solution is given by
			\begin{align}\label{3.10}
				\begin{aligned}
					h(t)&=\left[r_{\frac{1+\xi}{2\gamma}}(t) +\frac{1+\xi}{2\gamma}(r_{\frac{1+\xi}{2\gamma}}\ast r_{\frac{1+\xi}{2\gamma}})(t) \right]h_0 +(r_{\frac{1+\xi}{2\gamma}}\ast r_{\frac{1+\xi}{2\gamma}})(t)h_1\\
					&\quad+ \gamma^{-1} \int_0^t (r_{\frac{1+\xi}{2\gamma}}\ast r_{\frac{1+\xi}{2\gamma}})(t-\zeta)m(\zeta) \, \mathrm{d}\zeta.
				\end{aligned}
			\end{align}
		\end{itemize}
		\begin{proof}
			Using identities $a \ast b \equiv 1$ and $(a +\frac{1+\xi}{\gamma})\ast r_{\frac{1+\xi}{\gamma}}\equiv1$ we get
			\begin{align*}
				&\quad \frac{1}{\gamma}b \ast r_{\frac{1+\xi}{\gamma}} \ast(\gamma \partial_t (a \ast \partial_t(a \ast h)) +(1+\xi)\partial_t (a \ast h) )\\
				&= \partial_t (b \ast r_{\frac{1+\xi}{\gamma}} \ast a \ast \partial_t (a \ast h) )  +\frac{1+\xi}{\gamma}\partial_t (1 \ast r_{\frac{1+\xi}{\gamma} } \ast h)\\
				&\quad- \frac{1+\xi}{\gamma}(b \ast r_{\frac{1+\xi}{\gamma}})(t)h_0 -(b \ast r_{\frac{1+\xi}{\gamma}})(t)h_1 \\
				&= \partial_t (r_{\frac{1+\xi}{\gamma}} \ast a \ast h )  +\frac{1+\xi}{\gamma}\partial_t (1 \ast r_{\frac{1+\xi}{\gamma} } \ast h)\\
				&\quad- r_{\frac{1+\xi}{\gamma}}(t)h_0- \frac{1+\xi}{\gamma}(b \ast r_{\frac{1+\xi}{\gamma}})(t)h_0 -(b \ast r_{\frac{1+\xi}{\gamma}})(t)h_1 \\
				&=\partial_t r_{\frac{1+\xi}{\gamma}} \ast (a + \frac{1+\xi}{\gamma})\ast h -b(t)h_0
				-(b \ast r_{\frac{1}{\gamma}})(t)h_1 \\
				&=h(t) -b(t)h_0 - (b \ast r_{\frac{1+\xi}{\gamma}})(t)h_1.
			\end{align*}
			Hence, \eqref{3.8} can be transformed into
			\begin{align}\label{3.11}
				\begin{cases}
					h(t) + \frac{\lambda_j}{\gamma} (b \ast r_{\frac{1+\xi}{\gamma} }\ast h)(t)  =b(t) h_0 + (b \ast r_{\frac{1+\xi}{\gamma} })(t)h_1+ \frac{1}{\gamma}(b \ast r_{\frac{1+\xi}{\gamma} }\ast m)(t),\\
					(a \ast h)(0)= h_0, \ \left(a \ast (\partial_t a \ast h)\right) (0)=h_1.
				\end{cases}
			\end{align}
			We next apply the Laplace transform to obtain
			\begin{align*}
				\widehat{h}(\lambda) = \frac{\widehat{b}(\lambda)h_0 + \widehat{b}(\lambda) \widehat{r}_{\frac{1+\xi}{\gamma} }(\lambda)h_1 + \gamma^{-1}\widehat{b}(\lambda) \widehat{r}_{\frac{1+\xi}{\gamma} }(\lambda)\widehat{m}(\lambda) }{1+\gamma^{-1}\lambda_j \widehat{b}(\lambda)\widehat{r}_{\frac{1+\xi}{\gamma}}(\lambda)}.
			\end{align*}
			Substituting $\widehat{b}(\lambda)= [\lambda\widehat{a}(\lambda)]^{-1}$, $\widehat{r}_{\frac{1+\xi}{\gamma} }(\lambda) = \widehat{b}(\lambda)(1+ \gamma^{-1}(1+\xi) \widehat{b}(\lambda) ) $
			into the above and simplifying yields
			\begin{align}\label{3.12}
				\begin{aligned}
					\widehat{h}(\lambda) &= \frac{\widehat{b}(\lambda)h_0 + \widehat{b}(\lambda)^2(1+ \gamma^{-1}(1+\xi)\widehat{b}(\lambda) )^{-1}h_1 + \gamma^{-1}\widehat{b}(\lambda)^2(1+ \gamma^{-1}(1+\xi)\widehat{b}(\lambda) )^{-1} \widehat{m}(\lambda) }{1+ \gamma^{-1}\lambda_j \widehat{b}(\lambda)^2(1+ \gamma^{-1}(1+\xi)\widehat{b}(\lambda) )^{-1}}\\
					&=\frac{(\lambda\widehat{a}(\lambda)+\gamma^{-1}(1+\xi))h_0 + h_1 +\gamma^{-1}\widehat{m}(\lambda) }{\lambda^2\widehat{a}(\lambda)^2 + \gamma^{-1}(1+\xi) \lambda\widehat{a}(\lambda) +\gamma^{-1}\lambda_j }.
				\end{aligned}
			\end{align}
			Note that
			\begin{align*}
				\lambda^2\widehat{a}(\lambda)^2 + \gamma^{-1}(1+\xi) \lambda\widehat{a}(\lambda) +\gamma^{-1} \lambda_j = (\lambda \widehat{a}(\lambda) +Z_1 )(\lambda \widehat{a}(\lambda)+  Z_2).
			\end{align*}
			If $(1+\xi)^2-4\gamma\lambda_j \neq 0$, then simplifying \eqref{3.12} yields
			\begin{align*}
				\widehat{h}(\lambda)&=(\frac{1}{2} +\frac{1+\xi}{2\sqrt{(1+\xi)^2-4\gamma\lambda_j} } )\frac{h_0}{\lambda\widehat{a}(\lambda)+ Z_1} + (\frac{1}{2} -\frac{1+\xi}{2\sqrt{(1+\xi)^2-4\gamma\lambda_j} } )\frac{h_0}{\lambda\widehat{a}(\lambda)+ Z_2}\\
				&\quad+ \frac{\gamma}{\sqrt{(1+\xi)^2-4\gamma\lambda_j}}( \frac{h_1+\gamma^{-1}\widehat{m}(\lambda)}{\lambda\widehat{a}(\lambda)+Z_1 } - \frac{h_1+\gamma^{-1}\widehat{m}(\lambda)}{\lambda\widehat{a}(\lambda)+Z_2 }).
			\end{align*}
			Employing $\widehat{s}_{Z_i}(\lambda) =\widehat{a}(\lambda)(Z_i+ \lambda\widehat{a}(\lambda))^{-1}$ and $\widehat{r}_{Z_i}(\lambda) =(Z_i+ \lambda\widehat{a}(\lambda))^{-1}$ with $i=1,2$ gives
			\begin{align*}
				\widehat{h}(\lambda) &= (\frac{1}{2} +\frac{1+\xi}{2\sqrt{(1+\xi)^2-4\gamma\lambda_j} } )\widehat{r}_{Z_1}(\lambda) h_0
				+ (\frac{1}{2} -\frac{1+\xi}{2\sqrt{(1+\xi)^2-4\gamma\lambda_j} } ) \widehat{r}_{Z_2}(\lambda) h_0\\
				&\quad+ \frac{\gamma h_1}{\sqrt{(1+\xi)^2-4\gamma\lambda_j}}(\widehat{r}_{Z_1}(\lambda) -\widehat{r}_{Z_2}(\lambda) ) + \frac{\widehat{m}(\lambda)}{\sqrt{(1+\xi)^2-4\gamma\lambda_j}}( \widehat{r}_{Z_1}(\lambda) -\widehat{r}_{Z_2}(\lambda)).
			\end{align*}
			By the uniqueness of the Laplace transform, \eqref{3.9} follows.
			
			If $(1+\xi)^2-4\gamma\lambda_j =0$, then simplifying \eqref{3.12} yields
			\begin{align*}
				\widehat{h}(\lambda) = \frac{h_0}{\lambda \widehat{a}(\lambda) + \frac{1+\xi}{2\gamma} } + \frac{\frac{1+\xi}{2\gamma}h_0}{\left[ \lambda \widehat{a}(\lambda) + \frac{1+\xi}{2\gamma} \right]^2 } + \frac{h_1}{\left[ \lambda \widehat{a}(\lambda) + \frac{1+\xi}{2\gamma} \right]^2 }+ \frac{\gamma^{-1} \widehat{m}(\lambda)} {\left[ \lambda \widehat{a}(\lambda) + \frac{1+\xi}{2\gamma} \right]^2 }.
			\end{align*}
			Employing $\widehat{r}_{\frac{1+\xi}{2\gamma}}(\lambda) = (\lambda \widehat{a}(\lambda) + \frac{1+\xi}{2\gamma})^{-1}$ and the uniqueness of Laplace transform leads to  \eqref{3.10}.
		\end{proof}
	\end{lemma}
	
	For convenience, we denote by $H_{\gamma,\lambda_j,\xi}(t;h_0,h_1,m)$ the solution $h$ of \eqref{3.8}.
	
	\begin{remark}\label{rem 3.1}
		If $(1+\xi)^2-4\gamma\lambda_j\neq 0$, by Proposition \ref{prop 2.1}(ii) we have
		\begin{align}\label{3.13}
			(a \ast r_{Z_1})'(t) =-Z_1r_{Z_1}(t),\  (a \ast r_{Z_2})'(t)=-Z_2r_{Z_2}(t), \quad t> 0.
		\end{align}
		Substituting \eqref{3.13} into \eqref{3.9}, it is straightforward to verify that
		\begin{align}\label{3.14}
			\begin{aligned}
				&\quad (a \ast H_{\gamma,\lambda_j,\xi}(\cdot;h_0,h_1,m))'(t)\\
				&= -\frac{\lambda_j h_0}{\sqrt{(1+\xi)^2-4\gamma\lambda_j} } (r_{Z_1}(t)-r_{Z_2}(t) )
				- \frac{\gamma h_1}{\sqrt{(1+\xi)^2-4\gamma\lambda_j}}(Z_1r_{Z_1}(t) -Z_2r_{Z_2}(t) )\\
				&\quad - \frac{1}{\sqrt{(1+\xi)^2-4\gamma\lambda_j}}\int_{0}^t ( Z_1r_{Z_1}(t-\zeta) -Z_2 r_{Z_2}(t-\zeta))m(\zeta) \, \mathrm{d}\zeta.
			\end{aligned}
		\end{align}
		If $(1+\xi)^2-4\gamma\lambda_j= 0$, by Proposition \ref{prop 2.1}(ii) we have $(a\ast r_{\frac{1+\xi}{2\gamma} })'(t) =-\frac{1+\xi}{2\gamma} r_{\frac{1+\xi}{2\gamma}}(t)$. A direct computation gives
		\begin{align}\label{3.15}
			\begin{aligned}
				&\quad (a \ast  H_{\gamma,\lambda_j,
					\xi}(\cdot;h_0,h_1,m))'(t)\\
				&=-\frac{(1+\xi)^2}{4\gamma^2}(r_{\frac{1+\xi}{2\gamma}}\ast r_{\frac{1+\xi}{2\gamma}})(t)h_0
				+ \left[ r_{\frac{1+\xi}{2\gamma} }(t)-\frac{1+\xi}{2\gamma}(r_{\frac{1+\xi}{2\gamma}}\ast r_{\frac{1+\xi}{2\gamma}})(t) \right]h_1\\
				&\quad - \frac{1}{\gamma} \int_0^t \left[ r_{\frac{1+\xi}{2\gamma} }(t-\zeta)-\frac{1+\xi}{2\gamma}(r_{\frac{1+\xi}{2\gamma}}\ast r_{\frac{1+\xi}{2\gamma}})(t-\zeta) \right]m(\zeta) \, \mathrm{d}\zeta.
			\end{aligned}
		\end{align}
	\end{remark}
	
	We now prove the existence of $\mathcal R_{\gamma,\xi}(t)$ and derive its explicit formula. Assume that
	\begin{align*}
		&v(t,\cdot) = \sum_{j=1}^\infty v_j(t) e_j, \
		u_0 = \sum_{j=1}^\infty u_{0,j} e_j, \  u_1 = \sum_{j=1}^\infty u_{1,j} e_j,\ \Psi(v) = \sum_{j=1}^\infty \psi_j(t) e_j.
	\end{align*}
	Note that
	\begin{align*}
		\int_0^t \Phi(v(\zeta)) \, \mathrm{d}W(\zeta) &= \sum_{i=1}^{\infty} \int_0^t \Phi(v(\zeta)) e_i \mathrm{d}\beta_i(\zeta) = \sum_{i=1}^{\infty} \int_0^t \sum_{j=1}^{\infty} \langle \Phi(v(\zeta)) e_i, e_j \rangle e_j\,  \mathrm{d}\beta_i(\zeta)\\
		&= \sum_{j=1}^{\infty} e_j  \int_0^t  \sum_{i=1}^{\infty} \langle \Phi(v(\zeta)) e_i, e_j  \rangle \,  \mathrm{d}\beta_i(\zeta)\\
		&=\sum_{j=1}^{\infty} e_j  \int_0^t  \sum_{i=1}^{\infty} \langle \Phi(v(\zeta)) e_j, e_i  \rangle \,  \mathrm{d}\beta_i(\zeta) \\
		&:= \sum_{j=1}^{\infty} e_j \int_0^t  \varLambda_j(\zeta) \,  \mathrm{d}W(\zeta).
	\end{align*}
	Here the interchange of the integral and the summation is justified by Fubini's theorem, together with the assumption that $\int_0^T \mathbb{E} \left[ \|\Phi(v(\zeta))\|_{\mathcal L_2(H,H)}^2 \right] dt < \infty$. Then \eqref{3.1} can be equivalently written as
	\begin{align*}
		\begin{cases}
			\gamma \partial_t \left(a \ast (\partial_t a \ast v_j)\right)(t) + \partial_t(a \ast v_j)(t) + (\lambda_j+\xi) v_{j}(t)
			=\psi_j(t) + \varLambda_j(t) \mathrm{d}W(t) ,\\
			(a \ast v_j)(0)= u_{0,j}, \  \left(a \ast (\partial_t a \ast v_j) \right)(0)=u_{1,j},
		\end{cases}
	\end{align*}
	where $j \in \mathbb N_+$. By Lemma \ref{lem 3.1} and Remark \ref{rem 3.1} we have
	\begin{align*}
		v_j(t) = H_{\gamma, \lambda_j,\xi}(t; v_{0,j}, v_{1,j},\psi_j)
		+\gamma^{-1}\int_0^t  H_{\gamma,\lambda_j,\xi}(t-\zeta;0,1,0)
		\varLambda_j(\zeta) \, \mathrm{d}W(\zeta)
	\end{align*}
	and
	\begin{align*}
		w_j(t)=(a\ast v_j)'(t) &= (a \ast H_{\gamma,\lambda_j,\xi}(\cdot;v_{0,j}, v_{1,j}, \psi_j))'(t)\\
		&\quad + \gamma^{-1}\int_0^t H_{\gamma,\lambda_j,\xi}(t-\zeta;0,1,0)\varLambda_j(\zeta) \, \mathrm{d}W(\zeta).
	\end{align*}
	
	Set $v_j = \langle v, e_j \rangle$, $w_j = \langle w, e_j \rangle$. Define $\mathcal R_{\gamma,\xi}(t): \mathcal H^\delta \to \mathcal H^\delta$ by
	\begin{align}\label{3.16}
		\mathcal R_{\gamma,\xi}(t)(v,w): =\sum_{j=1}^\infty   ( H_{\gamma,\lambda_j,\xi}(t; v_j, w_j, 0) e_j , (a \ast H_{\gamma,\lambda_j,\xi}(\cdot;v_j,w_j,0))'(t) e_j).
	\end{align}
	In particular, we denote $\mathcal R_{\gamma,0}(t) = \mathcal R_\gamma(t)$.
	Using \eqref{3.9}, \eqref{3.10}, \eqref{3.14}, and \eqref{3.15}, one checks that the operator defined by \eqref{3.15} satisfies: (i)$\mathcal R_{\gamma,\xi} \in L_{loc}^2(\mathbb R_+, \mathcal L(\mathcal H^\delta))$ such that for a.a $t \geq 0$;(ii) $\mathcal R_{\gamma,\xi}(t)$ leaves $\mathcal D(\mathcal A_{\gamma,\xi})$ invariant,
	commutes with $\mathcal A_{\gamma,\xi}$ on $\mathcal D(\mathcal A_{\gamma,\xi})$;(iii) for each $(v,w)$ on $\mathcal D(\mathcal A_{\gamma,\xi})$, \eqref{3.6} holds. Thus the mild solution representation \eqref{3.7} is valid.
	
	We now introduce the notion of weak mild solutions for problem \eqref{1.2}.
	\begin{definition}\label{def 3.1}
		Assume that $(a,b)\in (\mathcal{PC}_\varepsilon)$ and $\gamma>0$. A weak mild solution to \eqref{1.2}
		with initial pair $(u_0, u_1) \in \mathcal{H}^1$ as a collectionn $ (\Omega, \mathcal{F}, (\mathcal{F}_t), \mathbb{P}, W, Y)$, subject to the following conditions:
		\begin{itemize}
			\item[(i)] $(\Omega, \mathcal{F}, (\mathcal{F}_t), \mathbb{P})$ is some filtered probability space;
			\item[(ii)] $W$ is a cylindrical Wiener process with respect to this filtration;
			\item[(iii)] $Y$ is a progressively measurable process taking values in $L_{loc}^2(\mathbb R_+;\mathcal H^\delta)$ for some $\delta \in [0,\frac{1}{2})$;
			\item[(iv)] $\mathbb{P}$-a.s, the mild solution $Y(t)$ satisfies
			\begin{align}\label{3.17}
				\begin{aligned}
					Y(t) &= \mathcal R_{\gamma}(t)(u_0,u_1) + \int_0^t \mathcal R_{\gamma}(t-\zeta) J_\gamma \Psi(\mathcal I_1 Y(\zeta))\, \mathrm{d}\zeta \\
					&\quad+ \int_0^t \mathcal R_{\gamma}(t-\zeta) J_\gamma \Phi(\mathcal I_1 Y(\zeta)) \, \mathrm{d}W(\zeta), \text{ ~a.a.~ } t \in \mathbb R_+.
				\end{aligned}
			\end{align}
		\end{itemize}
	\end{definition}
	
	\section{Estimates on the nonlocal telegraph equations resolvent $\mathcal R_{\gamma,\xi}(t) $}\label{sec 4}
	In this section, we establish operator norm estimates for $\mathcal R_{\gamma,\xi}(t)$. For each $\gamma>0$ and $\xi \geq 0$, we define $g_{\gamma,\lambda_j,\xi}: \mathbb R_+ \to \mathbb R$ by
	\begin{align*}
		g_{\gamma,j,\xi}(\theta) = \begin{cases}
			\frac{\gamma}{\sqrt{(1+\xi)^2-4\gamma\lambda_j}}(e^{-Z_1\theta }- e^{-Z_2\theta }), \text{ if } (1+\xi)^2-4\gamma\lambda_j>0,\\
			\theta e^{-\frac{1+\xi}{2\gamma}\theta }, \text{ if } (1+\xi)^2-4\gamma\lambda_j=0,\\
			\frac{2\gamma}{\sqrt{4\gamma\lambda_j-(1+\xi)^2} }e^{-\frac{1+\xi}{2\gamma}\theta } \sin(\frac{\sqrt{4\gamma\lambda_j-(1+\xi)^2 } }{2\gamma}\theta ), \text{ if } (1+\xi)^2-4\gamma\lambda_j<0.
		\end{cases}
	\end{align*}
	Differentiating $g_{\gamma,\lambda_j,\xi}(\theta)$ with respect to $\theta$ yields
	\begin{align*}
		g'_{\gamma,\lambda_j,\xi}(\theta) = \begin{cases}
			\frac{\gamma}{\sqrt{(1+\xi)^2-4\gamma\lambda_j}}(-Z_1e^{-Z_1\theta }+ Z_2e^{-Z_2\theta }), \text{ if } (1+\xi)^2-4\gamma\lambda_j>0,\\
			e^{-\frac{1+\xi}{2\gamma}\theta }(1-\frac{1+\xi}{2\gamma}\theta), \text{ if } (1+\xi)^2-4\gamma\lambda_j=0,\\
			e^{-\frac{1+\xi}{2\gamma}\theta } \left[\cos(\frac{\sqrt{4\gamma\lambda_j-(1+\xi)^2 } }{2\gamma}\theta ) \right. \\
			\quad \quad \quad \quad \ \left.- \frac{1+\xi}{\sqrt{4\gamma\lambda_j -(1+\xi)^2}}\sin(\frac{\sqrt{4\gamma\lambda_j-(1+\xi)^2 } }{2\gamma}\theta )  \right], \text{ if } (1+\xi)^2-4\gamma\lambda_j<0,
		\end{cases}
	\end{align*}
	In what follows, we establish new power-type estimates for $g_{\gamma,\lambda_j,\xi}(\theta)$ and $g'_{\gamma,\lambda_j,\xi}(\theta)$, see Appendix \ref{app A} for the proof.
	\begin{proposition}\label{prop 4.1}
		For any $\theta \geq 0$, we have
		\begin{align*}
			|g_{\gamma,\lambda_j,\xi}(\theta)|\leq \begin{cases}
				6 (1+\xi)^{-\alpha}\gamma^{\frac{1+\alpha}{2} } \lambda_j^{-\frac{1-\alpha}{2}}e^{-\frac{\lambda_j}{2(1+\xi)}\theta}, \text{ ~if~ } (1+\xi)^2-4\gamma\lambda_j > 0, \ \alpha \in [-1,1].\\
				3(1+\xi)^{-\alpha}\gamma^{\frac{1+\alpha}{2} } \lambda_j^{-\frac{1-\alpha}{2}}e^{-\frac{1+\xi}{4\gamma}\theta}, \text{ ~if~ } (1+\xi)^2-4\gamma\lambda_j \leq 0, \ \alpha \in [0,1].
			\end{cases}
		\end{align*}
	\end{proposition}
	
	\begin{proposition}\label{prop 4.2}
		For each $\alpha \in [0, 1]$ and any $\theta >0$, we have
		\begin{align*}
			|g'_{\gamma,\lambda_j,\xi}(\theta)|\leq \begin{cases}
				2\sqrt{2}(1+\xi)^{-\alpha} \gamma^{ \frac{\alpha}{2} } \lambda_j^{ \frac{\alpha}{2} }e^{ -\frac{\lambda_j}{2(1+\xi)}\theta }  + \sqrt{2}e^{-\frac{1+\xi}{4\gamma} \theta}, \text{ ~if~ } (1+\xi)^2-4\gamma\lambda_j>0,\\
				2\sqrt{2} (1+\xi)^{ -\alpha } \gamma^{ \frac{\alpha}{2} } \lambda_j^{ \frac{\alpha}{2} } e^{-\frac{1+\xi}{4\gamma}\theta}, \text{ ~if~ } (1+\xi)^2-4\gamma\lambda_j\leq0.
			\end{cases}
		\end{align*}
	\end{proposition}
	
	For certain special values, $H_{\gamma,\lambda_j,\xi}$ and $(a \ast H_{\gamma,\lambda_j,\xi})'$ admit subordination-type representations, as specified below.
	
	\begin{lemma}\label{lem 4.1}
		For each $\gamma>0$ and $\xi \geq 0$ and all $t>0$, we have
		\begin{align}
			&H_{\gamma,\lambda_j,\xi}(t;0,\gamma^{-1},0) = \gamma^{-1} \int_0^\infty g_{\gamma,\lambda_j,\xi}(\theta) \eta(t,\mathrm{d}\theta), \label{4.1} \\
			&(a\ast H_{\gamma,\lambda_j,\xi}(\cdot;0,\gamma^{-1},0))'(t) = \gamma^{-1} \int_0^\infty g_{\gamma,\lambda_j,\xi}'(\theta) \eta(t,\mathrm{d}\theta). \label{4.2}
		\end{align}
		\begin{proof}
			If $(1+\xi)^2-4\gamma\lambda_j\neq0$, one verifies directly from \eqref{3.9} and Proposition \ref{prop 2.1}(i) that \eqref{4.1} holds. Likewise, \eqref{4.2} follows from \eqref{3.14} and Proposition \ref{prop 2.1}(i).
			
			From Proposition \ref{prop 2.1}(i) we obtain
			\begin{align*}
				\int_0^\infty \theta e^{-\nu \theta} \, \widehat{\eta}(\lambda,\mathrm{d}\theta)=\int_0^\infty \theta e^{-\nu \theta}  e^{-\frac{\theta}{\widehat{b}(\lambda) } }\,\mathrm{d}\theta = \frac{\widehat{b}(\lambda)^2}{(1+\nu \widehat{b}(\lambda))^2} = \widehat{r}_\nu(\lambda)^2, \quad Re(\nu)>0.
			\end{align*}
			The uniqueness of the Laplace transform then gives
			\begin{align*}
				\int_{0}^\infty \theta e^{-\nu \theta} \eta(t,\mathrm{d}\theta) = (r_\nu \ast r_\nu)(t).
			\end{align*}
			Thus, if $(1+\xi)^2-4\gamma\lambda_j =0 $ , it follows from \eqref{3.10} and \eqref{3.15} that
			\begin{align*}
				H_{\gamma,\lambda_j,\xi}(t;0,\gamma^{-1},0) = \gamma^{-1}\int_0^\infty \theta e^{ -\frac{1+\xi}{2\gamma}\theta }\eta(t,\mathrm{d}\theta) =\gamma^{-1} \int_0^\infty g_{\gamma,\lambda_j,\xi}(\theta) \eta(t,\mathrm{d}\theta)
			\end{align*}
			and
			\begin{align*}
				(a \ast H_{\gamma,\lambda_j,\xi}(\cdot;0,\gamma^{-1},0))'(t)&=\gamma^{-1} \int_0^\infty e^{-\frac{1+\xi}{2\gamma}\theta }(1-\frac{1+\xi}{2\gamma}\theta) \eta(t,\mathrm{d}\theta)\\
				&=\gamma^{-1} \int_0^\infty g'_{\gamma,\lambda_j,\xi}(\theta) \eta(t,\mathrm{d}\theta),
			\end{align*}
			respectively.
			To summarize, both \eqref{4.1} and \eqref{4.2} are valid.
		\end{proof}
	\end{lemma}
	
	The following two lemmas are crucial for extracting decay factors in the moment estimates of stochastic integrals. Detailed proofs are provided in Appendix \ref{app A}. Our analysis shows that the new power-type estimates and subordination-type representations allow us to extract not only decay factors but also higher-power decay factors from estimates involving $H_{\gamma,\lambda_j,\xi}$ and $(a \ast H_{\gamma,\lambda_j,\xi})'$. We also introduce parameters(e.g., $\alpha_1,\beta_1$, see below) to encode the power-law exponents in $\gamma$, $\xi$, and $\lambda_j$. By adjusting these parameters, one may enhance the estimates with respect to a chosen parameter-$\gamma$, $\xi$, or $\lambda_j$-thereby accommodating various proof strategies. This flexibility is particularly important..
	\begin{lemma}\label{lem 4.2}
		Assume that $(a,b) \in (\mathcal{PC}_\varepsilon)$. Let  $\gamma>0$ and $\xi \geq 0$. Fix arbitrary $T>0$. Then we have the following assertions.
		\begin{itemize}
			\item[(i)] Suppose that the parameters $\alpha_1$ and $\beta_1$ satisfy
			\begin{align}\label{4.3}
				\begin{cases}
					0<\alpha_1<1,\\
					0\leq \beta_1<\min\{\alpha_1, \frac{\varepsilon}{2(2+\varepsilon)} \},\\
					\frac{1}{2}-\alpha_1+2\beta_1 >0.
				\end{cases}
			\end{align}
			Then for $(1+\xi)^2-4\gamma\lambda_j >0$, we have
			\begin{align}\label{4.4}
				\int_0^T |H_{\gamma,\lambda_j,\xi}(t;0,\gamma^{-1},0)|^2\, \mathrm{d}t
				\leq C_{\varepsilon,\beta_1,T}  (1+\xi)^{-2(\alpha_1-\beta_1)}\gamma^{\alpha_1-1}  \lambda_j^{-(1-\alpha_1+2\beta_1)}
			\end{align}
			and
			\begin{align}\label{4.5}
				\begin{aligned}
					&\quad \int_0^T |(a\ast H_{\gamma,\lambda_j,\xi}(\cdot;0,\gamma^{-1},0))'(t)|^2\, \mathrm{d}t\\
					&\leq C_{\varepsilon,\beta_1,T}  \lambda_j^{-(1-\alpha_1+2\beta_1)}  (1+\xi)^{-2(\alpha_1-\beta_1)}\gamma^{\alpha_1-2}
					+(1+\xi)^{-2\beta_1}\gamma^{2\beta_1-2},
				\end{aligned}
			\end{align}
			where $C_{\varepsilon,\beta_1,T} \to 0$ as $T\to 0$.
			\item [(ii)] Suppose that the parameters $\alpha_2$ and $\beta_2$ satisfy
			\begin{align}\label{4.6}
				\begin{cases}
					0\leq\alpha_2<\frac{1}{2},\\
					0\leq \beta_2< \frac{\varepsilon}{2(2+\varepsilon)}.
				\end{cases}
			\end{align}
			Then for $(1+\xi)^2-4\gamma\lambda_j \leq 0$, we have
			\begin{align}\label{4.7}
				\int_0^T |H_{\gamma,\lambda_j,\xi}(t;0,\gamma^{-1},0)|^2\, \mathrm{d}t
				\leq C_{\varepsilon,\beta_2,T}  (1+\xi)^{-2(\alpha_2+\beta_2)}\gamma^{2\beta_2+\alpha_2-1}  \lambda_j^{-(1-\alpha_2)}
			\end{align}
			and
			\begin{align}\label{4.8}
				\begin{aligned}
					\int_0^T  |(a \ast H_{\gamma,\lambda_j,\xi}(\cdot;0,\gamma^{-1},0))'(t)|^2\, \mathrm{d}t
					\leq C_{\varepsilon,\beta_2,T} \lambda_j^{\alpha_2} (1+\xi)^{-2(\alpha_2+\beta_2)}\gamma^{2\beta_2+\alpha_2-2},
				\end{aligned}
			\end{align}
			where $C_{\varepsilon,\beta_2,T} \to 0$ as $T\to 0$.
		\end{itemize}
	\end{lemma}
	
	\begin{remark}\label{rem 4.1}
		If $(a,b) \in (\mathcal{PC}_{\varepsilon})$ with $\varepsilon>2$, then the constraint $\alpha_1 \in (0,1)$ in \eqref{4.3}
		can be relaxed to $\alpha_1 \in (0,1]$.
	\end{remark}
	
	\begin{lemma}\label{lem 4.3}
		Assume that $(a,b) \in (\mathcal{PC}_\varepsilon)$. Let $\gamma>0$ and $\xi \geq 0$. Fix arbitrary $T>0$ and let $\delta\in [0, \frac{\varepsilon}{2+\varepsilon})$. Then for $(1+\xi)^2-4\gamma\lambda_j>0$, we have
		\begin{align}\label{4.9}
			\int_0^T |H_ {\gamma,\lambda_j,\xi}(t; 0, \gamma^{-1},0)|\,\mathrm{d}t
			\leq C_{\varepsilon,\delta,T} (1+\xi)^{-\frac{1-\delta}{2}}  \lambda_j^{-\frac{1+\delta}{2}}
		\end{align}
		and
		\begin{align}\label{4.10}
			\int_0^T |(a \ast H_ {\gamma,\lambda_j,\xi}(\cdot; 0, \gamma^{-1},0))'(t)|\,\mathrm{d}t
			\leq C_{\varepsilon,\delta,T}\gamma^{-\frac{1}{2} } (1+\xi)^{-\frac{1-\delta}{2}}  \lambda_j^{-\frac{\delta}{2}};
		\end{align}
		for $(1+\xi)^2-4\gamma\lambda_j\leq 0$, we have
		\begin{align}\label{4.11}
			\int_0^T |H_ {\gamma,\lambda_j,\xi}(t; 0, \gamma^{-1},0)|\,\mathrm{d}t
			\leq C_{\varepsilon,\delta,T}\gamma^{\frac{\delta}{2} } (1+\xi)^{-\frac{1+\delta}{2}}\lambda_j^{-1}
		\end{align}
		and
		\begin{align}\label{4.12}
			\int_0^T |(a \ast H_ {\gamma,\lambda_j,\xi}(\cdot; 0, \gamma^{-1},0))'(t)|\,\mathrm{d}t
			\leq C_{\varepsilon,\delta,T}\gamma^{-\frac{1-\delta}{2} } (1+\xi)^{-\frac{1+\delta}{2}}.
		\end{align}
	\end{lemma}
	
	The next lemma collects some estimates
	related to  $H_{\gamma,\lambda_j,\xi}(\cdot;h_0,h_1,0)$ with respect to $h_0$ and $h_1$.
	These estimates are useful in deriving deterministic estimates for $\mathcal R_{\gamma,\xi}(t)$. The proof is given in Appendix \ref{app A}.
	\begin{lemma}\label{lem 4.4}
		The solution $H_{\gamma,\xi,\lambda_j}(t;h_0,h_1,m)$ of \eqref{3.8} satisfies
		\begin{align}\label{4.13}
			|H_{\gamma,\lambda_j,\xi}(t;h_0,h_1,0)|\leq \frac{5}{2}b(t)|h_0| +\frac{3\gamma}{1+\xi} b(t)|h_1|,
		\end{align}
		and
		\begin{align}\label{4.14}
			\begin{aligned}
				&\quad |(a\ast H_{\gamma,\lambda_j,\xi}(\cdot;h_0,h_1,0))'(t)|\\
				&\leq \begin{cases}
					\frac{3\lambda_j}{1+\xi}b(t)|h_0| +2(\sqrt{2}+1)b(t)|h_1|,\text{ if } (1+\xi)^2-4\gamma\lambda_j \geq 0,\\
					3 \sqrt{\frac{\lambda_j}{\gamma} } b(t)|h_0| +2(\sqrt{2}+1)b(t)|h_1|,\text{ if } (1+\xi)^2-4\gamma\lambda_j < 0.
				\end{cases}
			\end{aligned}
		\end{align}
	\end{lemma}
	
	We now present estimates for the operator
	$\mathcal R_{\gamma,\xi}(t)$.
	
	\begin{lemma}\label{lem 4.5}
		Let $\xi \geq 0$ and $t>0$. For any $(v,w) \in \mathcal H^\delta$ with $\delta \in \mathbb R$, we have
		\begin{align*}
			&\|\mathcal R_{\gamma,\xi}(t)(v,w)\|_{ \mathcal H^\delta} \leq \frac{\sqrt{53}}{\sqrt{\gamma}}b(t)\|(v,w)\|_{\mathcal H^\delta} \text{ ~if~ } \gamma \in (0,1),\\
			&\| \mathcal R_{\gamma,\xi}(t)(v,w)\|_{\mathcal  H^\delta} \leq \sqrt{53\gamma}b(t)\|(v,w)\|_{\mathcal H^\delta} \text{ ~if~ } \gamma \geq 1.
		\end{align*}
		\begin{proof}
			It follows from \eqref{3.16} that
			\begin{align*}
				\left\| \mathcal R_{\gamma,\xi}(t)(v,w)\right\|^2_{\mathcal H^\delta} = \sum_{j=1}^\infty \left[ \lambda_j^\delta H_{\gamma, \lambda_j,\xi }(t,;v_j,w_j,0)^2 + \lambda_j^{\delta-1} \left[(a \ast H_{\gamma,\lambda_j,\xi }(\cdot,;v_j,w_j,0))'(t)\right]^2 \right].
			\end{align*}
			Further applying Lemma \ref{lem 4.4}, we have for all $\gamma>0$,
			\begin{align*}
				\lambda_j|H_{\gamma,\lambda_j,\xi}(t;h_0,h_1,0)|^2 + \gamma|(a \ast H_{\gamma,\lambda_j,\xi}(\cdot;h_0,h_1,0))'(t)|^2\leq 53 b(t)^2(\lambda_j|h_0|^2 + \gamma|h_1|^2).
			\end{align*}
			Hence, if $\gamma \in (0,1)$, then
			\begin{align*}
				\gamma \|\mathcal R_{\gamma,\xi}(t) (v,w)\|^2_{\mathcal H^\delta} \leq 53 b(t)^2 \sum_{j=1}^{\infty} (\lambda_j^\delta |v_j|^2 + \gamma\lambda_j^{\delta-1} |w_j|^2)\leq 53 b(t)^2 \|(v,w)\|_{\mathcal H^\delta}.
			\end{align*}
			If $\gamma \geq 1$, we get
			\begin{align*}
				\|\mathcal R_{\gamma,\xi}(t) (v,w)\|^2_{\mathcal H^\delta} \leq 53 \gamma b(t)^2 \sum_{j=1}^{\infty} (\lambda_j^\delta|v_j|^2 + \lambda_j^{\delta-1} |w_j|^2)\leq 53 \gamma b(t)^2 \|(v,w)\|_{\mathcal H^\delta}.
			\end{align*}
			The proof is completed.
		\end{proof}
	\end{lemma}
	
	Next, let $N_{\gamma,\xi} = \max \{j \in \mathbb{N}_+: (1+\xi)^2 - 4\gamma(\lambda_j+\xi) > 0\}$ and let $P_{N_{\gamma,\xi}}$ be the projection onto the span of $\{e_j\}_{j \leq N_{\gamma,\xi}}$.
	\begin{lemma}\label{lem 4.6}
		Assume that $(a,b) \in (\mathcal{PC}_\varepsilon)$. Let $\gamma>0$ and $\xi \geq 0$. Fix arbitrary $T>0$. For $\delta \in [0,\frac{1}{2})$, any $N\in \mathbb N_+$ with $N \leq N_{\gamma,\xi}$ and any $w \in L^2([0,T];H^{-1})$, there exists a positive constant $C_{\varepsilon,\delta,T}$ depending only on $\varepsilon$, $\delta$ and $T$, such that
		\begin{align*}
			\left\|\int_0^\cdot \mathcal I_1 \mathcal R_{\gamma,\xi}(\cdot-\zeta) J_\gamma P_Nw(\zeta) \right\|_{L^2([0,T];H^\delta)} \leq C_{\varepsilon,\delta,T} (1+\xi)^{-\frac{1-\delta}{2} } \|P_Nw\|_{L^2([0,T];H^{-1})}.
		\end{align*}
		\begin{proof}
			Write $w(\zeta)=\sum_{j=1}^\infty w_j(\zeta)e_j$. Then from \eqref{3.16} it follows that
			\begin{align*}
				\mathcal I_1\mathcal R_{\gamma,\xi}(t-\zeta) J_\gamma \Xi(\zeta) =\sum_{j=1}^\infty H_{\gamma,\lambda_j,\xi}(t-\zeta; 0, \gamma^{-1}, 0)w_j(\zeta) e_j.
			\end{align*}
			Set $\Xi_{1,\gamma,\xi}(t)= \int_0^t \mathcal I_1\mathcal R_{\gamma,\xi}(t-\zeta) J_\gamma P_Nw(\zeta)\, \mathrm{d}\zeta$. Then
			\begin{align*}
				\begin{aligned}
					\|\Xi_{1,\gamma,\xi}\|^2_{L^2([0,T];H^\delta)}&=\sum_{j=1}^N \lambda_j^\delta \int_0^T  \left|\int_0^t H_ {\gamma,\lambda_j,\xi}(t-\zeta; 0, \gamma^{-1}, 0)w_j(\zeta)\, \mathrm{d}\zeta \right|^2\, \mathrm{d}t\\
					&\leq \sum_{j=1}^N  \lambda_j^\delta\left(\int_0^T |H_ {\gamma,\lambda_j,\xi}(t; 0, \gamma^{-1},0)|\,\mathrm{d}t \right)^2 \int_0^T |w_j(t)|^2\,\mathrm{d}t.
				\end{aligned}
			\end{align*}
			Here the last estimate is obtained via Young's inequality.
			Combining the above with \eqref{4.9} proves the claim.
		\end{proof}
	\end{lemma}
	
	\begin{lemma}\label{lem 4.7}
		Assume that $(a,b) \in (\mathcal{PC}_\varepsilon)$. Let $\gamma>0$ and $\xi \geq 0$. Fix arbitrary $T>0$. For $\delta \in [0,\frac{\varepsilon}{2+\varepsilon})$, any $N\in \mathbb N_+$ with $N \leq N_{\gamma,\xi}$ and any $w \in L^2([0,T];H^{-1})$, there exists a positive constant $C_{\varepsilon,\delta,T}$, depending only on $\varepsilon$, $\delta$ and $T$, such that
		\begin{align*}
			\left\|\int_0^\cdot \mathcal I_2\mathcal R_{\gamma,\xi}(\cdot-\zeta) J_\gamma P_N w(\zeta)\, \mathrm{d}\zeta \right\|_{L^2([0,T];H^{\delta-1} )}\leq C_{\varepsilon,\delta,T}\gamma^{-\frac{1}{2} } (1+\xi)^{-\frac{1-\delta}{2} }\|P_N w\|_{L^2([0,T];H^{-1})}.
		\end{align*}
		\begin{proof}
			Denote $\Xi_{2,\gamma,\xi}(t)= \int_0^t \mathcal I_2\mathcal R_{\gamma,\xi}(t-\zeta) J_\gamma P_N w(\zeta)\, \mathrm{d}\zeta$. Similarly, we have
			\begin{align*}
				\begin{aligned}
					\|\Xi_{2,\gamma,\xi}\|^2_{L^2([0,T];H^{\delta-1})}
					&=\sum_{j=1}^N \lambda_j^{\delta-1} \int_0^T  \left|\int_0^t (a\ast H_ {\gamma,\lambda_j,\xi}(\cdot; 0, \gamma^{-1}, 0))'(t-\zeta)w_j(\zeta)\, \mathrm{d}\zeta \right|^2\, \mathrm{d}t\\
					&\leq \sum_{j=1}^N \lambda_j^{\delta-1} \left(\int_0^T |(a\ast H_ {\gamma,\lambda_j,\xi}(\cdot; 0, \gamma^{-1},0))'(t)|\,\mathrm{d}t \right)^2 \int_0^T |w_j(t)|^2\,\mathrm{d}t.
				\end{aligned}
			\end{align*}
			This together with \eqref{4.10} gives the result.
		\end{proof}
	\end{lemma}
	
	\begin{remark}\label{rem 4.2}
		Assume that  $(a,b) \in (\mathcal{PC}_\varepsilon)$. Let $\gamma>0$ and $\xi \geq 0$. Fix arbitrary $T>0$. For $\delta \in [0,\frac{\varepsilon}{2+\varepsilon})$, any $N \in \mathbb N_+$ with $N>N_{\gamma,\xi}$ and any $w \in L^2([0,T];H^{-1})$, we apply \eqref{4.11} and \eqref{4.12}.
		Following the same argument as in Lemma \ref{lem 4.6}, we can only obtain
		\begin{align*}
			&\quad \left\|\int_0^\cdot \mathcal I_1\mathcal R_{\gamma,\xi}(\cdot-\zeta) J_\gamma (I-P_N) w(\zeta)\, \mathrm{d}\zeta \right\|_{L^2([0,T];H )}\\
			&\leq C_{\varepsilon,\delta,T}\gamma^{\frac{\delta}{2}} (1+\xi)^{-\frac{1+\delta}{2} }\|(I-P_N) w\|_{L^2([0,T];H^{-1})}
		\end{align*}
		and
		\begin{align*}
			&\quad \left\|\int_0^\cdot \mathcal I_2\mathcal R_{\gamma,\xi}(\cdot-\zeta) J_\gamma (I-P_N) w(\zeta)\, \mathrm{d}\zeta \right\|_{L^2([0,T];H^{\delta-1} )}\\
			&\leq C_{\varepsilon,\delta,T}\gamma^{-\frac{1-\delta}{2} } (1+\xi)^{-\frac{1+\delta}{2} }\|(I-P_N) w\|_{L^2([0,T];H^{\delta-1})},
		\end{align*}
		respectively.
	\end{remark}
	
	\section{Moment estimation of stochastic integrals}\label{sec 5}
	For classical stochastic equations without memory effects, such as the stochastic heat and wave equations, solutions are typically studied as $C([0,T];H)$-valued processes. The nonlocal telegraph equation studied in this paper possesses an inherent initial singularity, which may cause the solution to be unbounded near the initial time. We therefore need to develop moment estimates for the solution in the more general space
	$L^2([0,T];H)$. Moreover, owing to the introduction of the new velocity damping mechanism, we are able to extract damping decay factors of higher power in the moment estimates for both the displacement and velocity fields.
	
	Given any progressively measurable self-adjoint process $\Upsilon \in L^2(\Omega, L^2([0,T],\mathcal L(H)))$, we denote
	\begin{align*}
		\Theta_{\gamma,\xi}(t)=\int_0^t \mathcal R_{\gamma,\xi}(t-\zeta)  J_\gamma \Upsilon(\zeta)\, \mathrm{d}W(\zeta).
	\end{align*}
	In certain scenarios, we may take
	$\Upsilon(\zeta)$ to be $\Phi(v(\zeta))$ $\Phi_n(v(\zeta))$, or $\Upsilon(\zeta)= \Phi(v(\zeta))-\Phi_n(v(\zeta))$ for some approximation $\Phi_n$ of $\Phi$.
	
	\begin{lemma}\label{lem 5.1}
		Assume that $(a,b) \in (\mathcal{PC}_\varepsilon)$. Let $\gamma>0$ and $\xi \geq 0$. Fix arbitrary $T>0$.
		\begin{itemize}
			\item[(i)] Suppose that the parameters $\alpha_1$ and $\beta_1$ satisfy \eqref{4.3}.  Then, for each $\delta\in [0, \frac{1}{2}-\alpha_1+2\beta_1)$,
			there exists a positive constant $C_{\alpha_1, \varepsilon, \beta_1, \delta, T}$ depending only on $\alpha_1$, $\varepsilon$, $\beta_1$, $\delta$ and $T$, such that
			\begin{align}\label{5.1}
				\begin{aligned}
					\mathbb E \left[ \int_0^T\|P_{N_{\gamma,\xi} } \mathcal I_1 \Theta_{\gamma,\xi}(t)\|^2_{H^\delta } \, \mathrm{d}t \right] &\leq C_{\alpha_1, \varepsilon, \beta_1, \delta, T} (1+\xi)^{-2(\alpha_1-\beta_1)}\gamma^{\alpha_1-1}\\
					&\quad \times\mathbb E \left[ \int_0^T \|\Upsilon(t)\|^2_{\mathcal L(H)}\, \mathrm{d}t \right].
				\end{aligned}
			\end{align}
			\item[(ii)] Suppose that the parameters $\alpha_2$ and $\beta_2$ satisfy \eqref{4.6}. Then, for each $\delta\in [0, \frac{1}{2}-\alpha_2)$, there exists a positive constant $C_{\alpha_2, \varepsilon, \beta_2, \delta, T}$ depending only on
			$\alpha_2$, $\varepsilon$, $\beta_2$, $\delta$ and $T$, such that
			\begin{align}\label{5.2}
				\begin{aligned}
					\mathbb E \left[ \int_0^T\|(I-P_{N_{\gamma,\xi} })\mathcal I_1 \Theta_{\gamma,\xi}(t)\|^2_{H^\delta } \, \mathrm{d}t \right] &\leq C_{\alpha_2, \varepsilon, \beta_2, \delta, T} (1+\xi)^{-2(\alpha_2+\beta_2)}\gamma^{2\beta_2+\alpha_2-1}\\
					&\quad \times \mathbb E \left[ \int_0^T \|\Upsilon(t)\|^2_{\mathcal L(H)}\, \mathrm{d}t \right].
				\end{aligned}
			\end{align}
		\end{itemize}
		\begin{proof}
			We first consider case (i). By Fubini's theorem and It\^o  isometry, we have
			\begin{align}\label{5.3}
				\begin{aligned}
					&\quad \mathbb E \left[ \int_0^T\|P_{N_{\gamma,\xi} }\mathcal I_1 \Theta_{\gamma,\xi}(t)\|^2_{H^\delta }\, \mathrm{d}t \right] \\
					&=\int_0^T \mathbb E \left[  \left \| \int_0^t  P_{N_{\gamma,\xi} }\mathcal I_1 \mathcal R_{\gamma,\xi}(t-\zeta)  J_\gamma \Upsilon(\zeta)\, \mathrm{d}W(\zeta)\right \|^2_{H^\delta} \right] \, \mathrm{d}t \\
					&=\int_0^T \mathbb E \left[  \int_0^t \left \| P_{N_{\gamma,\xi} } \mathcal I_1 \mathcal R_{\gamma,\xi}(t-\zeta)  J_\gamma \Upsilon(\zeta) \right\|^2_{\mathcal L_2(H,H^\delta)}  \, \mathrm{d}\zeta \right ] \, \mathrm{d}t\\
					&=\mathbb E \left[\int_0^T   \sum_{j=1}^{\infty } \int_0^t \left \|P_{N_{\gamma,\xi} }(-\Delta)^{\frac{\delta}{2}}  \mathcal I_1 \mathcal R_{\gamma,\xi}(t-\zeta)  J_\gamma \Upsilon(\zeta)e_j \right \|^2_{H}  \, \mathrm{d}\zeta \, \mathrm{d}t \right]\\
					&:=\mathbb E\left[  \int_0^T  \mathcal G_{1,\gamma,\xi}(t)  \, \mathrm{d}t\right] .
				\end{aligned}
			\end{align}
			Note that
			\begin{align*}
				\mathcal G_{1,\gamma,\xi}(t)
				&=\sum_{i=1 }^{N_{\gamma,\xi}} \sum_{j=1}^\infty \int_0^t  \left \langle (-\Delta)^{\frac{\delta}{2}} \mathcal I_1 \mathcal R_{\gamma,\xi}(t-\zeta)  J_\gamma \Upsilon(\zeta)e_j, e_i \right \rangle^2  \, \mathrm{d}\zeta\\
				&=\sum_{i=1}^{ N_{\gamma,\xi} } \sum_{j=1}^\infty \int_0^t  \left \langle \Upsilon(\zeta)e_j, J^{\ast}_\gamma \mathcal R^{\ast}_{\gamma,\xi}(t-\zeta)  \mathcal I_1^{\ast} (-\Delta)^{\frac{\delta}{2}} e_i \right \rangle^2  \, \mathrm{d}\zeta.
			\end{align*}
			Since
			\begin{align*}
				\left \langle   J^{\ast}_\gamma \mathcal R^{\ast}_{\gamma,\xi}(t-\zeta)  \mathcal I_1^{\ast} (-\Delta)^{\frac{\delta}{2}} e_i, e_j \right \rangle
				&= \left \langle (-\Delta)^{\frac{\delta}{2}}\mathcal I_1 \mathcal R_{\gamma,\xi}(t-\zeta) J_\gamma  e_i, e_j \right \rangle\\
				&=\begin{cases}
					\lambda_i^{\frac{\delta}{2} } H_{\gamma,\lambda_i,\xi}(t-\zeta;0,\gamma^{-1},0), \quad j=i,\\
					0, \quad j \neq i,
				\end{cases}
			\end{align*}
			we derive
			\begin{align}\label{5.4}
				\mathcal G_{1,\gamma,\xi}(t) = \sum_{i=1}^{ N_{\gamma,\xi} } \sum_{j=1}^\infty \int_0^t  \lambda_i^\delta |H_{\gamma,\lambda_i,\xi}(t-\zeta;0,\gamma^{-1},0)|^2 \left \langle \Upsilon(\zeta)e_j, e_i \right \rangle^2  \, \mathrm{d}\zeta.
			\end{align}
			From the self-adjointness of $\Upsilon$, we obtain
			\begin{align}\label{5.5}
				\begin{aligned}
					\sum_{j=1}^\infty \left\langle \Upsilon(\zeta)e_j,e_i \right\rangle^2 &= \sum_{j=1}^\infty \left\langle \Upsilon(\zeta)e_i,e_j \right\rangle^2
					=\|\Upsilon(\zeta)e_i\|^2_{H} \leq \|\Upsilon(\zeta)\|^2_{\mathcal L(H)} \|e_i\|^2_{ H}\\
					&\leq  \|\Upsilon(\zeta)\|^2_{\mathcal L(H)}.
				\end{aligned}
			\end{align}
			Therefore, we have
			\begin{align}\label{5.6}
				\begin{aligned}
					&\quad \int_0^T \mathcal G_{1,\gamma,\xi}(t)\, \mathrm{d}t\\
					&\leq \int_0^T \int_0^t \|\Upsilon(\zeta)\|^2_{\mathcal L(H)}  \sum_{i=1}^{ N_{\gamma,\xi} } \lambda_i^{\delta} |H_{\gamma,\lambda_i,\xi}(t-\zeta;0,\gamma^{-1},0)|^2\, \mathrm{d}\zeta \mathrm{d}t\\
					&=\int_0^T \|\Upsilon(\zeta)\|^2_{\mathcal L(H)}  \int_{\zeta}^T   \sum_{i=1}^{ N_{\gamma,\xi} } \lambda_i^{\delta} |H_{\gamma,\lambda_i,\xi}(t-\zeta;0,\gamma^{-1},0)|^2\, \mathrm{d}t \mathrm{d}\zeta  \\
					&=\int_0^T \|\Upsilon(\zeta)\|^2_{\mathcal L(H)}  \int_{0}^{T-\zeta}\sum_{i=1}^{ N_{\gamma,\xi} } \lambda_i^{\delta} |H_{\gamma,\lambda_i,\xi}(t;0,\gamma^{-1},0)|^2\, \mathrm{d}t \mathrm{d}\zeta\\
					&\leq \int_0^T \|\Upsilon(t)\|^2_{\mathcal L(H)}\, \mathrm{d}t \int_{0}^{T}   \sum_{i=1}^{ N_{\gamma,\xi} } \lambda_i^{\delta} |H_{\gamma,\lambda_i,\xi}(t;0,\gamma^{-1},0)|^2\, \mathrm{d}t.
				\end{aligned}
			\end{align}
			For  $(1+\xi)^2-4\gamma\lambda_j>0$, it follows from \eqref{4.4} that
			\begin{align*}
				\begin{aligned}
					\int_0^T \sum_{i=1}^{N_{\gamma,\xi}} \lambda_i^{\delta} |H_{\gamma,\lambda_i,\xi}(t;0,\gamma^{-1},0)|^2\, \mathrm{d}t
					\leq C_{\varepsilon,\beta_1,T} (1+\xi)^{-2(\alpha_1-\beta_1)}\gamma^{\alpha_1-1} \sum_{i=1}^{\infty} \lambda_i^{-(1-\alpha_1+2\beta_1-\delta)}.
				\end{aligned}
			\end{align*}
			Combining the above with \eqref{5.3} and \eqref{5.6} yields \eqref{5.1}, since the condition \eqref{4.3} and the choice of $\delta$ imply $\sum_{i=1}^{\infty} \lambda_i^{-(1-\alpha_1+2\beta_1-\delta)}<\infty$.
			
			Now consider case (ii). Similarly,
			\begin{align}\label{5.7}
				\begin{aligned}
					&\quad \mathbb E \left[ \int_0^T\|(I-P_{\gamma,\xi} ) \mathcal I_1 \Theta_{\gamma,\xi}(t)\|^2_{H^\delta }\, \mathrm{d}t \right] \\
					&=\mathbb E \left[\int_0^T   \sum_{j=1}^{\infty } \int_0^t \left \|(I-P_{N_{\gamma,\xi} })(-\Delta)^{\frac{\delta}{2}}  \mathcal I_1 \mathcal R_{\gamma,\xi}(t-\zeta)  J_\gamma \Upsilon(\zeta)e_j \right \|^2_{H}  \, \mathrm{d}\zeta \, \mathrm{d}t\right] \\
					&=:\mathbb E \left[ \int_0^T \mathcal G_{2,\gamma,\xi}(t)  \, \mathrm{d}t\right] .
				\end{aligned}
			\end{align}
			and
			\begin{align}\label{5.8}
				\begin{aligned}
					\int_0^T \mathcal G_{2,\gamma,\xi}(t)\,\mathrm{d}t
					\leq \int_0^T \|\Upsilon(t)\|^2_{\mathcal L(H)}\, \mathrm{d}t \int_{0}^{T}   \sum_{i=N_{\gamma,\xi}+1}^{ \infty } \lambda_i^{\delta} |H_{\gamma,\lambda_i,\xi}(t;0,\gamma^{-1},0)|^2\, \mathrm{d}t.
				\end{aligned}
			\end{align}
			For $(1+\xi)^2 \leq 4\gamma\lambda_j$, it follows from \eqref{4.7} that
			\begin{align*}
				\begin{aligned}
					\int_0^T \sum_{i=N_{\gamma,\xi}+1}^\infty \lambda_i^{\delta} |H_{\gamma,\lambda_i,\xi}(t;0,\gamma^{-1},0)|^2\, \mathrm{d}t
					\leq C_{\varepsilon,\beta_2,T} (1+\xi)^{-2(\alpha_2+\beta_2)}\gamma^{2\beta_2+\alpha_2-1} \sum_{i=1}^{\infty} \lambda_i^{-(1-\alpha_2-\delta)}.
				\end{aligned}
			\end{align*}
			Combining the above with \eqref{5.7} and \eqref{5.8} gives \eqref{5.2} since the condition \eqref{4.6} and the choice of $\delta$ imply $\sum_{i=1}^{\infty} \lambda_i^{-(1-\alpha_2-\delta)}<\infty$.
		\end{proof}
	\end{lemma}
	
	\begin{lemma}\label{lem 5.2}
		Assume that $(a,b) \in (\mathcal{PC}_\varepsilon)$. Let $\gamma>0$ and $\xi \geq 0$. Fix arbitrary $T>0$.
		\begin{itemize}
			\item[(i)] Suppose that the parameters $\alpha_1$ and $\beta_1$ satisfy \eqref{4.3}. Then for each $\delta\in [0, \min\{\frac{1}{2} ,\frac{1}{2}-\alpha_1+2\beta_1\})$, there exists a positive constant $C_{\alpha_1, \varepsilon, \beta_1, \delta, T}$ depending only on
			$\alpha_1$, $\varepsilon$, $\beta_1$, $\delta$ and $T$, such that
			\begin{align}\label{5.9}
				\begin{aligned}
					&\quad \mathbb E \left[ \int_0^T\|P_{N_{\gamma,\xi} } \mathcal I_2 \Theta_{\gamma,\xi}(t)\|^2_{H^{\delta-1} } \, \mathrm{d}t \right]\\ &\leq C_{\alpha_1, \varepsilon, \beta_1, \delta, T} \left[(1+\xi)^{-2(\alpha_1-\beta_1)}\gamma^{\alpha_1-2}+(1+\xi)^{-2\beta_1}\gamma^{2\beta_1-2} \right]\\
					&\quad \times\mathbb E \left[ \int_0^T \|\Upsilon(t)\|^2_{\mathcal L(H)}\, \mathrm{d}t \right].
				\end{aligned}
			\end{align}
			\item[(ii)] Suppose that parameters $\alpha_2$ and $\beta_2$ satisfy \eqref{4.6}. Then for each $\delta\in [0, \frac{1}{2}-\alpha_2)$, there exists a positive constant $C_{\alpha_2, \varepsilon, \beta_2, \delta, T}$ depending only on $\alpha_2$, $\varepsilon$, $\beta_2$, $\delta$ and $T$, such that
			\begin{align}\label{5.10}
				\begin{aligned}
					&\quad \mathbb E \left[ \int_0^T\|(I-P_{N_{\gamma,\xi} })\mathcal I_2 \Theta_{\gamma,\xi}(t)\|^2_{H^{\delta-1} } \, \mathrm{d}t \right]\\ &\leq C_{\alpha_2, \varepsilon, \beta_2, \delta, T} (1+\xi)^{-2(\alpha_2+\beta_2)}\gamma^{2\beta_2+\alpha_2-2}\\
					&\quad \times \mathbb E \left[ \int_0^T \|\Upsilon(t)\|^2_{\mathcal L(H)}\, \mathrm{d}t \right].
				\end{aligned}
			\end{align}
		\end{itemize}
		\begin{proof}
			First prove case (i). As in the derivation of \eqref{5.3} and \eqref{5.6}, it suffices to show
			\begin{align}\label{5.11}
				\begin{aligned}
					&\quad \mathbb E \left[ \int_0^T\|P_{N_{\gamma,\xi} }\mathcal I_2 \Theta_{\gamma,\xi}(t)\|^2_{H^{\delta-1} }\, \mathrm{d}t \right] \\
					&=\int_0^T \mathbb E \left[  \sum_{j=1}^{\infty } \int_0^t \left \|P_{N_{\gamma,\xi} }(-\Delta)^{-\frac{1-\delta}{2}}  \mathcal I_2 \mathcal R_{\gamma,\xi}(t-\zeta)  J_\gamma \Upsilon(\zeta)e_j \right \|^2_{H}  \, \mathrm{d}\zeta \right] \, \mathrm{d}t\\
					&=\mathbb E\left[ \int_0^T \mathcal G_{3,\gamma,\xi}(t)\, \mathrm{d}t \right].
				\end{aligned}
			\end{align}
			and
			\begin{align}\label{5.12}
				\begin{aligned}
					\int_0^T \mathcal G_{3,\gamma,\xi}(t)\,\mathrm{d}t \leq \int_0^T \|\Upsilon(t)\|^2_{\mathcal L(H)}\, \mathrm{d}t \int_{0}^{T}   \sum_{i=1}^{N_{\gamma,\xi}} \lambda_i^{-(1-\delta)} |(a \ast H_{\gamma,\lambda_i,\xi}(\cdot;0,\gamma^{-1},0))'(t)|^2\, \mathrm{d}t.
				\end{aligned}
			\end{align}	
			For $(1+\xi)^2-4\gamma\lambda_j>0$, from \eqref{4.5} we derive
			\begin{align*}
				\begin{aligned}
					&\quad \int_0^T \sum_{i=1}^{N_{\gamma,\xi}} \lambda_i^{-(1-\delta)} |(a\ast H_{\gamma,\lambda_i,\xi}(\cdot;0,\gamma^{-1},0))'(t)|^2\, \mathrm{d}t\\
					&\leq C_{\varepsilon, \beta_1, T} \sum_{i=1}^{\infty} \lambda_i^{-(1-\alpha_1+2\beta_1-\delta)}  (1+\xi)^{-2(\alpha_1-\beta_1)}\gamma^{\alpha_1-2}\\
					&\quad + C_{\varepsilon, \beta_1, T} \sum_{i=1}^{\infty} \lambda_i^{-(1-\delta)}(1+\xi)^{-2\beta_1}\gamma^{2\beta_1-2}.
				\end{aligned}
			\end{align*}
			Combining the above with \eqref{5.11} and \eqref{5.12} gives \eqref{5.9}.
			
			Now consider case (ii). Similarly,
			\begin{align}\label{5.13}
				\begin{aligned}
					&\quad \mathbb E \left[ \int_0^T\|(I-P_{\gamma,\xi} ) \mathcal I_2 \Theta_{\gamma,\xi}(t)\|^2_{H^{\delta-1} }\, \mathrm{d}t \right] \\
					&=\mathbb E \left[\int_0^T   \sum_{j=1}^{\infty } \int_0^t \left \|(I-P_{N_{\gamma,\xi} })(-\Delta)^{-\frac{(1-\delta)}{2}}  \mathcal I_2 \mathcal R_{\gamma,\xi}(t-\zeta)  J_\gamma \Upsilon(\zeta)e_j \right \|^2_{H}  \, \mathrm{d}\zeta \, \mathrm{d}t\right] \\
					&=:\mathbb E \left[ \int_0^T \mathcal G_{4,\gamma,\xi}(t)  \, \mathrm{d}t\right] .
				\end{aligned}
			\end{align}
			and
			\begin{align}\label{5.14}
				\begin{aligned}
					\int_0^T \mathcal G_{4,\gamma,\xi}(t)\,\mathrm{d}t
					\leq \int_0^T \|\Upsilon(t)\|^2_{\mathcal L(H)}\, \mathrm{d}t \int_{0}^{T}   \sum_{i=N_{\gamma,\xi}+1}^{ \infty } \lambda_i^{-(1-\delta)} |(a\ast H_{\gamma,\lambda_i,\xi}(\cdot;0,\gamma^{-1},0))'(t)|^2\, \mathrm{d}t.
				\end{aligned}
			\end{align}
			For $(1+\xi)^2 \leq 4\gamma\lambda_j$, from \eqref{4.8} we obtain
			\begin{align*}
				\begin{aligned}
					&\quad \int_0^T \sum_{i=N_{\gamma,\xi}+1}^\infty \lambda_i^{-(1-\delta)} |(a \ast H_{\gamma,\lambda_i,\xi}(\cdot;0,\gamma^{-1},0))'(t)|^2\, \mathrm{d}t\\
					&\leq C_{\varepsilon,\beta_2,T}\sum_{i=1}^{\infty} \lambda_i^{-(1-\alpha_2-\delta)} (1+\xi)^{-2(\alpha_2+\beta_2)}\gamma^{2\beta_2+\alpha_2-2}.
				\end{aligned}
			\end{align*}
			Combining the above with \eqref{5.13} and \eqref{5.14} gives \eqref{5.10}.
		\end{proof}
	\end{lemma}
	
	\begin{remark}\label{rem 5.1}
		If $(a,b) \in (\mathcal{PC}_{\varepsilon})$ with $\varepsilon>2$, then the constraint $\alpha_1 \in (0,1)$ in Lemma \ref{lem 5.1}(i) and Lemma \ref{lem 5.2}(i)
		can be relaxed to $\alpha_1 \in (0,1]$.
	\end{remark}
	
	For the compactness argument, we need the spatial compactness and temporal regularity conditions below. See Appendix \ref{app B} for proofs.
	\begin{proposition}\label{prop 5.1}
		Assume that $(a,b) \in (\mathcal{PC}_\varepsilon)$, $(\mathcal{M}_2)$ and $(\mathcal M_4)$ hold. Let  $\gamma>0$ and $\xi\geq 0$. Fix arbitrary $T>0$. Then for each $\delta\in [0, \frac{1}{2})$, there exists a positive constant $C_{\gamma,\varepsilon,\delta,T}$ depending only on $\gamma$, $\varepsilon$, $\delta$ and $T$, such that
		\begin{align}\label{5.15}
			\begin{aligned}
				\sup_{\xi \geq 0}\mathbb E \left[ \int_0^T\| \Theta_{\gamma,\xi}(t)\|^2_{\mathcal{H}^\delta } \, \mathrm{d}t \right] \leq C_{\gamma,\varepsilon,\delta,T}
				\left( 1 + \mathbb E \left[ \int_0^T \|\mathcal I_1 Y(t)\|^2_{H}\, \mathrm{d}t \right] \right),
			\end{aligned}
		\end{align}
		provided $Y \in L^2(\Omega,L^2([0,T];\mathcal H))$.
	\end{proposition}
	
	\begin{proposition}\label{prop 5.2}
		Assume that $(a,b) \in (\mathcal{PC}_\varepsilon)$, $(\mathcal{M}_2)$ and $(\mathcal M_4)$ hold. Let $\gamma >0$ and $\xi \geq 0$. Fix arbitrary $T>0$. Then there exists a positive constant $C_{\gamma,\varepsilon,T}$  depending only on $\gamma$, $\varepsilon$ and $T$, such that
		\begin{align}\label{5.16}
			\mathbb E\left[ \int_0^{T-\varrho} \| \Theta_{\gamma,\xi}(t+\varrho) -\Theta_{\gamma,\xi}(t)\|_{\mathcal H} \, \mathrm{d}t \right] \leq C_{\gamma,\varepsilon,T} o(1) (1+\mathbb E\left[ \|\mathcal I_1Y\|_{L^2([0,T]; H)}^2 \right]).
		\end{align}
		provided $Y \in L^2(\Omega,L^2([0,T];\mathcal H))$,
		where $o(1)$ depends solely on
		$\gamma$, $\xi$ and $\varrho$, and  $o(1) \to 0$ as $\varrho \to 0^+$.
	\end{proposition}
	
	\section{Novel generalized coupling method: weak existence and uniqueness in law}\label{sec 6}
	This section establishes the weak existence and uniqueness in law of mild solution to \eqref{1.2}.
	\begin{theorem}\label{the 6.1}
		Assume that $(a,b) \in (\mathcal{PC}_\varepsilon)$ and $\gamma>0$. $(\mathcal M_1), (\mathcal M_2), (\mathcal M_3), (\mathcal M_4)$ hold. Then, for any initial datum $(u_0, u_1) \in \mathcal H^1$, \eqref{1.2} admits a weak mild solution in the sense of Definition \ref{def 3.1}, where $\delta$ can be chosen arbitrarily close to $\min\{\frac{1}{2},\frac{\varepsilon}{2+\varepsilon}\}$ from below, and the solution is unique in law.
	\end{theorem}
	
	The following Lipschitz approximation result can be found in Ravsky \cite{A. Ravsky}, whose work is discussed in Han \cite[Proposition 3.3]{Y. Han}. We draw on this result to construct our approximating Lipschitz maps.
	
	\begin{proposition}\label{prop 6.1}
		Suppose $V$ is a separable Hilbert space and $V'$ a normed space. For any continuous map $f\colon V\to V'$, one may construct a sequence $\{f_n\}$ of bounded Lipschitz maps $V\to V'$ such that $\{f_n\}$ converge uniformly to $f$ on each compact subset of $V$ as $n \to \infty$.
	\end{proposition}
	
	The above property shows that, for each fixed $T>0$ and any map $\Phi:L^2([0,T];H) \to L^2([0,T];\mathcal L(H))$ satisfying $(\mathcal M_2)$, there exists a sequence of Lipschitz approximations
	$\Phi_n: L^2([0,T];H) \to L^2([0,T];\mathcal L(H))$  that converges uniformly to $\Phi$ on every compact set in $L^2([0,T];H)$. As stated in \cite[Section 3.2]{Y. Han}, for $n$ sufficiently large, $\Phi_n$ remains uniformly nondegenerate, and there exists a constant $C>0$, independent of
	$n$, such that $\Phi_n$ is  $\kappa$-H\"{o}lder continuous whenever
	$\|v_1-v_2\|_{L^2([0,T];H)} >d_n$, where $d_n>0$ satisfies $\lim_{n \to \infty}d_n =0$. Moreover, by employing the affine interpolation, one checks that if $\|v_1-v_2\|_{L^2([0,T];H)} <d_n$, then for some $C>0$ independent of $n$, $\|\Phi_n(v_1)-\Phi_n(v_2)\|_{L^2([0,T];\mathcal L(H)) } \leq C(d_n)^\kappa$. Also, $\Phi_n$ retains the linear growth condition with constant $C'_{\Phi,T}$ independent of $n$:
	\begin{align}\label{6.1}
		\|\Phi_n(v)\|_{L^2([0,T];\mathcal L(H)) } \leq C'_{\Phi,T} (1+\|v\|_{L^2([0,T];H)}) \text{ ~for each~ } n \in \mathbb N_+.
	\end{align}
	
	\subsection{Proof of weak existence}
	For any fixed $T>0$, consider the following SPDE with the initial value
	$((a \ast v)(0), (a \ast w)(0))= (u_0,u_1) \in \mathcal H^{1}$:
	\begin{align}\label{6.2}
		\begin{cases}
			\partial_t (a \ast  v^n) = w^n,  \\
			\partial_t (a \ast  w^n) = \frac{1}{\gamma} \left[\Delta  v^n - (1+\xi)  w^n
			+  \Phi_n(v^n) \frac{\mathrm{d}W}{\mathrm{d}t} \right].
		\end{cases}
	\end{align}
	Since $\Phi_n$ is Lipschitz, the classical Picard iteration suffices to prove the strong existence and uniqueness for \eqref{6.2}. Precisely, for any $\delta \in [0,\frac{1}{2})$, there is a unique progressively measurable $L^2([0,T];\mathcal H^\delta)$-valued process
	$Y^{\xi,n}(t) = (v^{\xi,n}(t),w^{\xi,n}(t)))$ such that $\mathbb{P}$-a.s.,
	\begin{align*}
		\begin{aligned}
			Y^{\xi,n}(t)= \mathcal R_{\gamma,\xi}(t)(v_0,v_1)
			+ \int_0^t \mathcal R_{\gamma,\xi}(t-\zeta) J_\gamma \Phi_n(\mathcal I_1Y^{\xi,n}(\zeta)) \mathrm{d}W(t), \text{ ~a.a.~ } t \in [0,T].
		\end{aligned}
	\end{align*}
	
	Indeed, the sequence of solutions $Y^{\xi,n}(t)$ further satisfies the spatial compactness and temporal regularity conditions stated below.
	
	\begin{proposition}\label{prop 6.2}
		For any arbitrary $\delta \in [0,\frac{1}{2})$, $\gamma>0$ and $T>0$, the following assertions hold.
		\begin{itemize}
			\item[(i)] $\sup_{\xi \geq 0}\sup_{n \in \mathbb N_+} \mathbb E\left[\int_0^T\|Y^{\xi,n}(t)\|^2_{\mathcal H^\delta}\, \mathrm{d}t \right] <\infty$.
			\item[(ii)] For each $\xi \geq 0$, $\lim_{\varrho\to \infty} \sup_{n \in \mathbb N_+} \mathbb E\left[ \int_0^{T-\varrho} \| Y^{\xi,n}(t+\varrho) -Y^{\xi}(t)\|^2_{\mathcal H} \, \mathrm{d}t \right]=0$.
		\end{itemize}
		\begin{proof}
			We first prove (i). We choose the parameters $\beta_1=\beta_2 =\frac{\varepsilon}{4(2+\varepsilon)}$, $\alpha_1=2\beta_1$ and $\alpha_2=0$. Fix arbitrary $T_1 >0$. From \eqref{4.4} and \eqref{4.7}, it follows that for any $\delta \in [0,\frac{1}{2})$, the function $q_1(t):= \sum_{i=1}^{ \infty } \lambda_i^{\delta} |H_{\gamma,\lambda_i,\xi}(t;0,\gamma^{-1},0)|^2$ belongs to $L^1([0,T_1])$ and satisfies
			\begin{align}\label{6.3}
				\int_0^T q_1(t)\, \mathrm{d}t \leq C_{\varepsilon,\delta,T}(1+\xi)^{-2\beta_1}\gamma^{2\beta_1-1}\leq  C_{\gamma,\varepsilon,\delta,T_1} \text{ ~for each~ } T \in (0,T_1].
			\end{align}
			Similarly, \eqref{4.5} and \eqref{4.8} imply that for any
			$\delta \in [0,\frac{1}{2})$, the function
			\begin{align*}
				q_2(t):= \sum_{i=1}^{ \infty } \lambda_i^{\delta-1} |(a\ast H_{\gamma,\lambda_i,\xi}(\cdot;0,\gamma^{-1},0))'(t)|^2
			\end{align*}
			belongs to $L^1([0,T_1])$ and satisfies
			\begin{align}\label{6.4}
				\int_0^T q_2(t)\, \mathrm{d}t \leq C_{\varepsilon,\delta,T}(1+\xi)^{-2\beta_1}\gamma^{2\beta_1-2}\leq  C_{\gamma,\varepsilon,\delta,T_1} \text{ ~for each~ }  T \in (0,T_1].
			\end{align}
			
			Recall from the proof of Lemma \ref{lem 5.1} and from \eqref{6.1} that
			\begin{align}\label{6.5}
				\begin{aligned}
					&\quad  \int_0^T \|\mathcal I_1 \Theta_{\gamma,\xi}(t)\|_{H^\delta}\, \mathrm{d}t  \\
					&\leq  \int_0^T \int_0^t \mathbb \|\Phi_n(\mathcal I_1Y^{\xi,n}(\zeta))\|^2_{\mathcal L(H)}  \sum_{i=1}^{ \infty } \lambda_i^{\delta} |H_{\gamma,\lambda_i,\xi}(t-\zeta;0,\gamma^{-1},0)|^2\, \mathrm{d}\zeta \mathrm{d}t\\
					&=\int_0^T \int_0^t \mathbb \|\Phi_n(\mathcal I_1Y^{\xi,n}(t-\zeta))\|^2_{\mathcal L(H)}  \sum_{i=1}^{ \infty } \lambda_i^{\delta} |H_{\gamma,\lambda_i,\xi}(\zeta;0,\gamma^{-1},0)|^2\, \mathrm{d}\zeta \mathrm{d}t\\
					&=\int_0^T \sum_{i=1}^{ \infty } \lambda_i^{\delta} |H_{\gamma,\lambda_i,\xi}(\zeta;0,\gamma^{-1},0)|^2  \int_0^{T-\zeta} \mathbb \|\Phi_n(\mathcal I_1Y^{\xi,n}(t))\|^2_{\mathcal L(H)} \, \mathrm{d}t \mathrm{d}\zeta\\
					&\leq C'_{\Phi,T_1} \int_0^T \sum_{i=1}^{ \infty } \lambda_i^{\delta} |H_{\gamma,\lambda_i,\xi}(\zeta;0,\gamma^{-1},0)|^2  (1+\|\mathcal I_1Y^{\xi,n}\|_{L^2([0,T-\zeta];H)}) \, \mathrm{d}\zeta.
				\end{aligned}
			\end{align}
			To justify the last step in the above, we note that for any
			$\zeta_1 \in [0,T]$ and $T \in (0,T_1]$,
			\begin{align*}
				\|\Phi_n(\mathcal I_1 Y^{\xi,n}(t) )\|_{L^2([0,\zeta_1];\mathcal L(H)) } &=\left\|\Phi_n \left(\mathcal I_1 Y^{\xi,n}(t)  1_{t \in [0,\zeta_1]} \right) \right\|_{L^2([0,\zeta_1];\mathcal L(H)) } \\
				&\leq \left\|\Phi_n \left(\mathcal I_1 Y^{\xi,n}(t)  1_{t \in [0,\zeta_1]} \right) \right\|_{L^2([0,T_1];\mathcal L(H))}\\
				&\leq C'_{\Phi,T_1} (1+\|\mathcal I_1Y^{\xi,n}(t)\|_{L^2([0,\zeta_1];H)}).
			\end{align*}
			This proves the desired inequality between the fourth and fifth terms. Next, applying Lemma \ref{lem 4.5} together with \eqref{6.3} and \eqref{6.5}, we obtain
			\begin{align}\label{6.6}
				\begin{aligned}
					&\quad \mathbb E\left[\|\mathcal I_1 Y^{\xi,n} \|_{L^2([0,T];H^{\delta})}  \right] \\
					&\leq C_\gamma \|b\|^2_{L^2([0,T_1])} \|(v_0,v_1)\|_{\mathcal H^1} + C'_{\Phi,T_1}C_{\gamma,\varepsilon,\delta,T_1}\\
					&\quad + C'_{\Phi,T_1} \int_0^T q_1(T-\zeta)  \mathbb E\left[\|\mathcal I_1Y^{\xi,n}\|_{L^2([0,\zeta];H)} \right]\, \mathrm{d}\zeta
				\end{aligned}
			\end{align}
			uniformly for all $n \in \mathbb N_+$. Similarly, from the proof of Lemma \ref{lem 5.2} and \eqref{6.1} we get an estimate like \eqref{6.5}. Combining it with Lemma \ref{lem 4.5} and \eqref{6.4}, then taking expectations, we obtain
			\begin{align}\label{6.7}
				\begin{aligned}
					&\quad \mathbb E\left[\|\mathcal I_2 \Theta_{\gamma,\xi}(t)\|_{L^2([0,T];lH^{\delta-1})} \right] \\
					&\leq C'_{\Phi,T_1} \int_0^T \sum_{i=1}^{ \infty } \lambda_i^{\delta} |(a\ast H_{\gamma,\lambda_i,\xi}(\cdot;0,\gamma^{-1},0))'(T-\zeta)|^2  (1+\mathbb E\left[\|\mathcal I_1Y^{\xi,n}\|_{L^2([0,\zeta];H)} \right] \, \mathrm{d}\zeta.
				\end{aligned}
			\end{align}
			Next, \eqref{6.4}, and \eqref{6.7} give
			\begin{align}\label{6.8}
				\begin{aligned}
					&\quad \mathbb E\left[\|\mathcal I_2 Y^{\xi,n} \|_{L^2([0,T];H^{\delta-1})}  \right] \\
					&\leq C_\gamma \|b\|^2_{L^2([0,T_1])} \|(v_0,v_1)\|_{\mathcal H^1} +C'_{\Phi,T_1} C_{\gamma,\varepsilon,\delta,T_1}\\
					&\quad+ C'_{\Phi,T_1} \int_0^T q_2(T-\zeta)  \mathbb E\left[\|\mathcal I_1Y^{\xi,n}\|_{L^2([0,\zeta];H)} \right]\, \mathrm{d}\zeta
				\end{aligned}
			\end{align}
			uniformly for all $n \in \mathbb N_+$.
			
			Let $p(T):= \mathbb E\left[\|Y^{\xi,n} \|_{L^2([0,T];\mathcal H^\delta)}  \right]$  be a nonnegative continuous function, and define $q(t):=q_1(t) +q_2(t)$. Let $C:=\max\{C'_{\Phi,T}, S_1+S_2\}$, where $S_1$ and $S_2$ denote the sums of the first two terms on the right-hand sides of \eqref{6.6} and \eqref{6.8}, respectively.
			In view of the compact embedding
			$H^\delta \subset\subset H$, from \eqref{6.6} and \eqref{6.8} we obtain
			\begin{align*}
				p(T)\leq C + \int_0^T Cq(T-\zeta) p(\zeta) \, \mathrm{d}\zeta,  \quad T\geq 0.
			\end{align*}
			Denote by $Q\in L^2([0,T_1])$ the resolvent kernel of
			$Cq$(see \cite[Chapter 3]{G. Gripenberg}), that is, $Q(T)$ is nonnegative solves
			\begin{align*}
				Q(T) = C(q \ast Q)(T) +Cq(T), \quad \text{ ~a.a.~} T \in [0,T_1].
			\end{align*}
			Applying the comparison principle yields
			\begin{align*}
				p(T) \leq C(1+ \|Q\|_{L^1([0,T])}), \text{ ~for~ } T \in [0,T_1]
			\end{align*}
			uniformly in $\xi$ and  $n$. This proves assertion (i).
			
			We now turn to (ii). The strong continuity of the translation operators implies that, for each
			$\xi \geq 0$,
			\begin{align*}
				\lim_{\varrho\to 0^+}\int_0^{T-\varrho} \left\| (\mathcal R_{\gamma,\xi}(t+\varrho)-\mathcal R_{\gamma,\xi}(t))(v_0,v_1) \right\|^2_{\mathcal H^\delta } \mathrm{d}t =0.
			\end{align*}
			Combining the above with Proposition \ref{prop 5.2} and assertion (i) yields (ii).
		\end{proof}
	\end{proposition}
	
	Let $Y^n(t) = Y^{0,n}(t)$ be the solution sequence. For each $\delta \in [0,\frac{1}{2})$, there exists $\delta' \in (\delta,\frac{1}{2})$ such that $H^{\delta'}\subset \subset H^{\delta} \subset H$. From Proposition \ref{prop 6.2} and Markov's inequality we obtain that for each $\sigma>0$, there exists a compact set $K_\sigma\subset L^2([0,T];\mathcal H^{\delta})$ defined by
	\begin{align*}
		K_\sigma:=\left\{Y: \|Y\|_{L^2([0,T];\mathcal H^{\delta'})}\leq 	\frac{C}{\sigma}, \int_0^{T-\varrho} \| Y(t+\varrho) -Y(t)\|_{\mathcal H} \, \mathrm{d}t\leq o(1) \right\},
	\end{align*}
	where $o(1) \to 0$ as $\varrho \to 0^+$ and $C$ is some constant, such that $\mathbb P(Y^n \in K_\sigma) \geq 1-\sigma$ for each $n$. An application of Prokhorov's theorem shows that the sequence of laws $\mathrm{Law}(Y^n)$ on $L^2([0,T];\mathcal H^{\delta})$  is relatively compact with respect to the topology of weak convergence.
	Consequently, there exists a subsequence, again denoted by $Y^n$, converging weakly to some probability measure $\mu$ on $L^2([0,T],\mathcal H^{\delta})$.  Finally, invoking the Skorokhod representation theorem together with the continuity of $\Phi$, we prove that for any $\delta \in[0,\frac{1}{2})$, there exists a progressively measurable process $Y \in L^{2}([0,T];\mathcal H^{\delta})$ which is a weak mild solution in the sense of Definition \ref{def 3.1}.
	
	\subsection{Proof of uniqueness in law}
	The proof of uniqueness in law is inspired by the approaches in \cite{Y. Han} and \cite{A. Kulik}. In our treatment, however, we introduce a completely different framework, the GC-V method, for which the stochastic control term, the stopping time construction, and the tail probability estimates all have to be devised anew. In addition, the simple velocity-damping coupling construction $\xi(\mathcal I_2Y(t)-\mathcal I_2\tilde Y^n(t))$ gives rise to the following obstruction:
	since the velocity field does not lie in the underlying space $H$ of the cylindrical Wiener process $W(t)$, the Girsanov transform cannot be applied to eliminate this coupling term, and consequently the distribution distance between the approximating process and the original process cannot be obtained.
	However, we overcome this obstacle by constructing a regularized velocity-damping control term (see \eqref{6.14}). But a new issue then arises: in estimating the extra error term
	\begin{align*}
		\xi(I-\epsilon \Delta)^{-1}  \mathcal I_2 \bar{Y}^n(\zeta)-\mathcal I_2 \bar{Y}^n(\zeta),
	\end{align*}
	how can one suppress the blow-up effect of $\xi$ as $\xi\to \infty$?(see \eqref{6.17}). We resolve this by means of an interpolation iteration strategy (see \eqref{6.19}), which achieves an improvement in the power-law exponent with respect to the factor $1+\xi$ and thereby successfully suppresses the blow-up, enabling the key tail probability estimate (see \eqref{6.20}).
	It is worth pointing out that the regularizing effect of the operator $\mathcal R_{\gamma,\xi}(t) J_\gamma (I-P_{N_{\gamma,\xi}})$ is relatively insufficient to support the implementation of the interpolation iteration strategy (see Remark \ref{rem 4.2}). To avoid this issue, we first apply the interpolation iteration strategy to the projection term $P_N(Y(t)-\tilde{Y}^n(t))$. After proving uniqueness in law for all finite-dimensional projections of the weak mild solution, we let $N \to \infty$ to obtain uniqueness in law for the weak mild solution itself.
	
	We first make several important observations.
	\begin{itemize}
		\item[(i)] In view of the nondegeneracy condition $(\mathcal M_2)$ on the diffusion coefficient
		$\Phi$ and the linear growth assumption on the drift $\Psi$,  it follows that if uniqueness in law for \eqref{1.2} can be established in the case $\Psi\equiv0$, then the same property holds for the general case by virtue of Girsanov's theorem. Hence, in what follows, it suffices to consider the case $\Psi\equiv0$. The Girsanov transform for infinite-dimensional SPDEs is discussed in detail in \cite[Remark 2.1]{F. Masiero}; see also the references therein.
		\item[(ii)] We set the parameters as
		$\beta_1=\beta_2 \in (0,\frac{\varepsilon}{2(2+\varepsilon)})$, $\alpha_1=2\beta_1$ and $\alpha_2=0$. Then, by Lemma \ref{lem 5.1} and Lemma \ref{lem 5.2}, for any $\gamma>0$, $T>0$ and $\delta \in [0,\frac{1}{2})$, there exists a positive constant $C_{\gamma, \varepsilon, \beta_1, \delta, T}$ depending only on $\gamma$, $\varepsilon$, $\beta_1$, $\delta$ and $T$, such that
		\begin{align}\label{6.9}
			\begin{aligned}
				\mathbb E \left[\int_0^T\|  \Theta_{\gamma,\xi}(t)\|^2_{\mathcal{H}^\delta} \, \mathrm{d}t \right] \leq C_{\gamma, \varepsilon, \beta_1,\delta,T} (1+\xi)^{-2\beta_1}
				\mathbb E \left[ \int_0^T \|\Upsilon(t)\|^2_{\mathcal L(H)}\, \mathrm{d}t \right].
			\end{aligned}
		\end{align}
		\item[(iii)] For each $\epsilon>0$ and each $\delta \in (0,1)$, we have
		\begin{align}
			&\|(I-\epsilon \Delta)^{-1} f\|_{H} \leq \epsilon^{ -\frac{1-\delta}{2}  } \|f\|_{H^{\delta-1}},  \label{6.10}\\
			& \|(I-\epsilon \Delta)^{-1} f-f\|_{H^{-1}} \leq \epsilon^{ \frac{\delta}{2} }    \|f\|_{H^{\delta-1}}. \label{6.11}
		\end{align}
		\item[(iv)] Let $\delta \in (0, \min\{\frac{1}{2},\frac{\varepsilon}{2+\varepsilon} \})$ and $\rho \in (\frac{(1-\delta)^2}{1+\delta^2},1)$.
		Since
		\begin{align*}
			\kappa>\begin{cases}
				1-\frac{\varepsilon^2}{2\varepsilon^2+4\varepsilon+4}, \text{ ~if~ } \varepsilon \in (0,2), \\
				1-\frac{2\varepsilon}{5(2+\varepsilon)}, \text{ ~if~ } \varepsilon \in [2,\infty).
			\end{cases}
		\end{align*}
		if $\varepsilon \in (0,2)$ , choose
		$\delta$ close to $\frac{\varepsilon}{2+\varepsilon}$, and $\rho$ close to $\frac{2}{\varepsilon^2+2\varepsilon+2}$; if $\varepsilon \in [2,\infty)$, choose $\delta$ close to $\frac{1}{2}$ and $\rho$ close to $\frac{1}{5}$. Then choose $\beta_1$ close to $\frac{\varepsilon}{2(2+\varepsilon)}$ such that $\beta_1(1-\rho) +\kappa = 1+2\chi_1>1 $ for some $\chi_1>0$. Moreover, $\rho \in (\frac{(1-\delta)^2}{1+\delta^2},1)$ implies that
		\begin{align*}
			\frac{\rho}{1-\delta}>\frac{(1-\delta)(1-\rho)}{2\delta}.
		\end{align*}
		We can then choose $c\in(\frac{(1-\delta)(1-\rho)}{\delta} ,\frac{2\rho}{1-\delta} ) $ such that
		\begin{align*}
			&1-\frac{(1+\delta)(1-\rho)}{2}+ \frac{c \delta}{2}=1+2\chi_2>1 \text{ ~for some~ } \chi_2>0, \\
			&\rho-\frac{c(1-\delta)}{2}=\chi_3>0.
		\end{align*}
	\end{itemize}
	
	Throughout this section, the parameters are chosen as in (iv), and we will not state this again.
	
	We first prove that, for any weak mild solution $Y(t)$ of \eqref{1.2}, the projection $P_N Y(t)$ is unique in law for each $N \in \mathbb N_+$. Consider the following SPDEs with the same initial value $((a \ast v)(0), (a \ast w)(0))= (u_0,u_1) \in \mathcal H^{1}$:
	\begin{align}
		&\begin{cases}
			\partial_t (a \ast v) =w,  \\
			\partial_t (a \ast w) = \frac{1}{\gamma} \left[\Delta  v -w
			+ \Phi(v) \frac{\mathrm{d}W}{\mathrm{d}t} \right],
		\end{cases}\label{6.12}\\
		&\begin{cases}
			\partial_t (a \ast  v^n) = w^n,  \\
			\partial_t (a \ast  w^n) = \frac{1}{\gamma} \left[\Delta  v^n -  w^n
			+  \Phi_n(v^n) \frac{\mathrm{d}W}{\mathrm{d}t} \right],
		\end{cases}\label{6.13}\\
		&\begin{cases}
			\partial_t (a \ast  \tilde{v}^n) = \tilde{w}^n  \\
			\partial_t (a \ast \tilde{w}^n) = \frac{1}{\gamma} \left[\Delta  \tilde{v}^n -  \tilde{w}^n
			+ \xi(I-\epsilon \Delta)^{-1} ( w- \tilde{w}^n)1_{t \leq \tau} + \Phi_n(\tilde{v}^n) \frac{\mathrm{d}W}{\mathrm{d}t} \right],
		\end{cases}\label{6.14}
	\end{align}
	where $\tau$  is some stopping time adapted to the filtration $(\mathcal F_t)_{t\geq 0}$, $\xi$ is the damping factor, and
	$\epsilon$ and $\xi$ will be specified subsequently.
	
	Since $\Phi_n$ is Lipschitz continuous, the strong existence and uniqueness for \eqref{6.14} can be proved. Denote the mild solutions of \eqref{6.13} and \eqref{6.14} by $Y^n(t) = (v^n(t), w^n(t))$ and $\tilde{Y}^n(t)=(\tilde{v}^n(t),\tilde{w}^n(t))$, respectively. Consider the compact set in
	$L^2([0,T];\mathcal H)$:
	\begin{align*}
		L_M:=\left\{Y: \|Y\|_{L^2([0,T];\mathcal H^{\delta})}\leq 	M, \int_0^{T-\varrho} \| Y(t+\varrho) -Y(t)\|_{\mathcal H} \, \mathrm{d}t\leq o(1) \right\},
	\end{align*}
	where $M>0$ and $o(1) \to 0$ as $\varrho \to 0^+$. Define
	\begin{align*}
		\varDelta_L^n :=\max\{d_n, \sup_{u \in L_M} \|\Phi(u)-\Phi_n(u)\|_{L^2([0,T];\mathcal L(H))} \}.
	\end{align*}
	Clearly, $\lim_{n\to \infty} \varDelta_L^n = 0$. Hence, for all sufficiently large $n$, it holds that $\varDelta_L^n \in [0,1)$. Let $A=\{\omega:Y_{[0,T]} \in L_M  \}\in \mathcal F_T$. By the a priori estimates from Proposition \ref{prop 5.1} and Proposition \ref{prop 5.2}, along with Markov's inequality, one has $\mathbb P(A) \to 0$ as $M \to \infty$. Take $M$
	large enough that $\mathbb P(A) \geq \frac{1}{2}$. Define the stopping time
	\begin{align*}
		\tau:=\inf\{t\geq 0: \|P_NY-P_N\tilde{Y}^n\|_{L^2([0,t];\mathcal H^\delta)} \geq 2\varDelta_L^n \} \wedge T.
	\end{align*}
	Set $1+\xi= (\varDelta_L^n)^{\rho-1}$ and $\epsilon=(\Delta_L^n)^{c}$.
	
	We employ the Girsanov transform to estimate $P_NY(t)-P_N\tilde{Y}^n(t)$. Let $d_{TV}$ denote the total variation distance on $\mathcal P(L^2([0,T];\mathcal H))$, and let
	$\mathbb H_{rel}(\cdot|\cdot)$ be the relative entropy between two probability measures in $\mathcal P(L^2([0,T];\mathcal H))$. By Pinsker's inequality,
	\begin{align*}
		d_{TV}(\mathrm{Law}(P_NY^n|_{[0,T]}),\mathrm{Law}(P_N\tilde{Y}^n|_{[0,T]})) \leq\sqrt{2 \mathbb H_{rel}(\mathrm{Law}(P_NY^n|_{[0,T]})|\mathrm{Law}(P_N\tilde{Y}^n|_{[0,T]}))}.
	\end{align*}
	Note that the SPDE for $P_N\tilde{Y}^n$
	is obtained from that for $P_NY^n$ by adding a drift to the Brownian paths, with the corresponding path transformation
	\begin{align*}
		(\mathrm{d}W(t))_{ t \geq 0} \to (\mathrm{d}W(t) + \xi\Phi_n(t,v_n(t))^{-1}  (I-\epsilon \Delta)^{-1 }(P_Nw(t) - P_Nw^n(t)) 1_{ t \leq \tau}\mathrm{d}t).
	\end{align*}
	Since $(\Phi_n)^{-1}$  is bounded, it follows from \eqref{6.10} that
	\begin{align}\label{6.15}
		d_{TV}(\mathrm{Law}(P_NY^n|_{[0,T]}),\mathrm{Law}(P_N\tilde{Y}^n|_{[0,T]})) \lesssim (\varDelta_L^n)^{\rho-\frac{c(1-\delta)}{2}}.
	\end{align}
	
	We now turn to the estimation of $P_NY(t) -P_N\tilde{Y}^n(t)$  on the sub-probability space $(A,\mathcal F_T|_A,\mathbb P(\cdot|A))$ via a pathwise argument. By \eqref{2.2}, on
	$A$, whenever $t \in [0,\tau]$ and $\varDelta_L^n \leq 1$,
	\begin{align}\label{6.16}
		\begin{aligned}
			&\quad \|P_N\Phi(v)-P_N\Phi_n(\tilde{v}^n)\|_{L^2([0,t];\mathcal L(H))}\\
			&\lesssim\|\Phi(P_N v)-\Phi_n(P_N\tilde{v}^n)\|_{L^2([0,t];\mathcal L(H))}\\
			&\leq \|\Phi(P_N v)-\Phi_n(P_N v)\|_{L^2([0,t];\mathcal L(H))}+ \|\Phi_n(P_N v)-\Phi_n(P_N\tilde{v}^n)\|_{L^2([0,t];\mathcal L(H))}\\
			&\leq \varDelta_L^n+C_{\Phi,T}(\varDelta_L^n)^{\kappa} \lesssim C_{\Phi,T} (\varDelta_L^n)^\kappa.
		\end{aligned}
	\end{align}
	Here, the second term in the third line is estimated as follows. If $\|v-\tilde{v}^n\|_{L^2([0,t];H)} \leq d_n$, then, since $d_n \leq \varDelta_L^n$,
	\begin{align*}
		\|\Phi_n(v) -\Phi_n(\tilde{v}^n)\|_{L^2([0,t];H)} \lesssim (d_n)^{\kappa}\leq (\varDelta_L^n)^\kappa.
	\end{align*}
	If $ \|v-\tilde{v}^n\|_{L^2([0,t];H)} \in (d_n,\varDelta_L^n]$, then the uniform H\"{o}lder continuity of $\Phi_n$(in $n$) on large scales gives
	\begin{align*}
		\|\Phi_n(v) -\Phi_n(\tilde{v}^n)\|_{L^2([0,t];H)} \lesssim \|v-\tilde{v}^n\|^\kappa_{L^2([0,t];H)} \leq  (\varDelta_L^n)^\kappa.
	\end{align*}
	
	Denote $P_NY(t) -P_N\tilde{Y}^n(t) = P_N\bar{Y}^n(t) = (P_N\bar{v}^n(t),P_N\bar{w}^n(t))$. Then, for a.a. $t \in [0,\tau]$, $P_N\bar{Y}^n(t)$ is the mild solution of the following SPDE:
	\begin{align*}
		\begin{cases}
			\partial_t (a \ast P_N\bar{v}^n) =P_N \bar{w}^n  \\
			\partial_t (a \ast P_N\bar{w}^n) = \frac{1}{\gamma} \left[\Delta P_N\bar{v}^n -(1+\xi)P_N\bar{w}^n
			+P_N\Phi_n(\bar{v}^n) \frac{\mathrm{d}W}{\mathrm{d}t} \right]\\
			\quad\quad\quad\quad\quad \ +\frac{\xi}{\gamma} \left[(I-\epsilon \Delta)^{-1} (P_N w-P_N \bar{w}^n)- (P_N w-P_N \bar{w}^n) \right].
		\end{cases}
	\end{align*}
	Therefore, if $\omega\in A$, for a.a. $t \in (0,\tau]$,
	\begin{align}\label{6.17}
		\begin{aligned}
			&\quad P_NY(t) -P_N\tilde{Y}^n(t)\\
			&= \int_0^t \mathcal R_{\gamma,\xi}(t-\zeta)J_{\gamma}\left[P_N\Phi(\mathcal I_1 Y(\zeta))-P_N\Phi_n(\mathcal I_1 \tilde{Y}^n(\zeta) ) \right] \, \mathrm{d}W(\zeta)\\
			&\quad + \xi\int_0^t \mathcal R_{\gamma,\xi}(t-\zeta)J_{\gamma}P_N \left[(I-\epsilon \Delta)^{-1}  \mathcal I_2 \bar{Y}^n(\zeta)-\mathcal I_2 \bar{Y}^n(\zeta) \right]\, \mathrm{d}\zeta.
		\end{aligned}
	\end{align}
	From the conditional expectation formula
	$\mathbb E_{\mathbb P(\cdot|A)}[\cdot] = \frac{\mathbb E[\cdot 1_A]}{\mathbb P(A)}$, $\mathbb{P}(A) \geq \frac{1}{2}$, and \eqref{6.9}, \eqref{6.16} we get that for $t \in [0,\tau]$,
	\begin{align}\label{6.18}
		\begin{aligned}
			&\quad \mathbb E_{\mathbb P(\cdot|A) }\left[ \left\|\int_0^\cdot \mathcal R_{\gamma,\xi}(\cdot-\zeta)J_{\gamma}\left[P_N \Phi(\mathcal I_1 Y(\zeta))-P_N \Phi_n(\mathcal I_1 \tilde{Y}^n(\zeta) ) \right] \, \mathrm{d}W(\zeta) \right\|^2_{L^2([0,t];\mathcal H^\delta)} \right]\\
			&\leq \frac{1}{\mathbb P(A)} \mathbb E\left[ \left\|\int_0^\cdot \mathcal R_{\gamma,\xi}(\cdot-\zeta)J_{\gamma}\left[P_N \Phi(\mathcal I_1 Y(\zeta))-P_N \Phi_n(\mathcal I_1 \tilde{Y}^n(\zeta) ) \right] \, \mathrm{d}W(\zeta) \right\|^2_{L^2([0,t];\mathcal H^\delta)} \right]\\
			&\lesssim (\varDelta_L^n)^{ 2\beta_1(1-\rho)} \|P_N \Phi(\mathcal I_1 Y)-P_N \Phi_n(\mathcal I_1 \tilde{Y}^n )\|_{L^2([0,t];\mathcal L(H))} \\
			&\lesssim \mathbb  (\varDelta_L^n)^{2\beta_1(1-\rho) +2\kappa}.
		\end{aligned}
	\end{align}
	Note that, for each fixed $N$, when $n$ is sufficiently large, $\xi$ is also chosen large enough so that $N \leq N_{\gamma,\xi}$. Consequently, applying Lemma \ref{lem 4.6}, Lemma \ref{lem 4.7}, and \eqref{6.11} yields that for
	$t \in [0,\tau]$,
	\begin{align}\label{6.19}
		\begin{aligned}
			&\quad \xi^2\mathbb E_{\mathbb P(\cdot|A) } \left[ \left\|\int_0^{\cdot} \mathcal R_{\gamma,\xi}(\cdot-\zeta)J_{\gamma}P_N \left[(I-\epsilon \Delta)^{-1} \mathcal I_2  \bar{Y}^n(\zeta)-\mathcal  I_2 \bar{Y}^n(\zeta) \right]\, \mathrm{d}\zeta \right\|^2_{L^2([0,t];\mathcal H^\delta) } \right]\\
			&\lesssim (1+\xi)^{1+\delta} \|(I-\epsilon \Delta)^{-1} \mathcal I_2 P_N \bar{Y}^n-\mathcal  I_2 P_N\bar{Y}^n\|^2_{L^2([0,t];H^{-1})} \\
			&\lesssim (\varDelta_L^n)^{-(1+\delta)(1-\rho)+c\delta  } \|\mathcal  I_2 P_N\bar{Y}^n\|^2_{L^2([0,t];H^{\delta-1})}\\
			&\lesssim
			(\varDelta_L^n)^{2-(1+\delta)(1-\rho)+ c\delta }.
		\end{aligned}
	\end{align}
	This, together with \eqref{6.18}, shows that for $t \in [0,\tau]$,
	\begin{align*}
		\mathbb E_{\mathbb P(\cdot|A) }\left[\|P_NY(t) -P_N\tilde{Y}^n(t)\|^2_{L^2([0,t];\mathcal H^\delta)}\right] \leq C (\varDelta_L^n)^{1+ 2\min\{\chi_1,\chi_2\} },
	\end{align*}
	provided $\varDelta_L^n \leq 1$. Then applying Markov's inequality yields
	\begin{align*}
		\mathbb P_A\left(\|P_NY-P_N\tilde{Y}^n\|_{L^2([0,\tau];\mathcal H^\delta)} \geq C(\varDelta_L^n)^{1+ 2\min\{\chi_1,\chi_2\} }S \right) \leq \frac{C}{S^2}.
	\end{align*}
	Set $S= C^{-1}(\varDelta_L^n)^{-\min\{ \chi_1,\chi_2\}}$. Then
	\begin{align}\label{6.20}
		\mathbb P_A\left(\|P_NY-P_N\tilde{Y}^n\|_{L^2([0,\tau];\mathcal H^\delta)} \geq (\varDelta_L^n)^{1+ \min\{\chi_1,\chi_2\} } \right) \leq C^2 (\varDelta_L^n)^{2\min\{ \chi_1,\chi_2\} }.
	\end{align}
	
	Consider the set
	\begin{align*}
		A_1:= \{\omega\in A: \|P_NY-P_N\tilde{Y}^n\|_{L^2([0,\tau];\mathcal H^\delta)} \geq (\varDelta_L^n)^{1+ \min\{\chi_1,\chi_2\} }  \}.
	\end{align*}
	Then, on $A-A_1$, it holds necessarily that $\|P_NY-P_N\tilde{Y}^n\|_{L^2([0,t];\mathcal H^\delta)} \leq  \varDelta_L^n$ for $t \in (0,\tau]$. Thus, by the continuity of the norm $\|\cdot\|_{L^2([0,t];\mathcal H^\delta)}$ in $t$, we conclude that  $\tau = T$ on $A-A_1$. By \eqref{6.20} and $\varDelta_L^n \to 0$ as $n \to \infty$, for each $\iota>0$, we have
	\begin{align}\label{6.21}
		\mathbb P_{A-A_1}\left(\|P_NY-P_N\tilde{Y}^n\|_{L^2([0,T];\mathcal H^\delta)} \geq \iota \right) \to 0 \text{ ~as~ } n \to \infty.
	\end{align}
	
	Since each approximating solution $Y^n$ has a well-defined law, it remains to prove that, for any weak solution $Y$,
	\begin{align}\label{6.22}
		\mathbb E[F(P_N Y) -F(P_N Y^n)] \to 0 \text{ ~as~ } n \to \infty.
	\end{align}
	where $F:L^2([0,T];\mathcal H^{\delta}) \to \mathbb R$ is an arbitrary bounded continuous functional. Then $P_N Y$ is unique in law.
	
	Note that \eqref{6.15} implies
	\begin{align*}
		\mathbb E\left[(F(P_N Y^n) -F(P_N \tilde{Y}^n)) \right] \to 0 \text{ ~as~ } n \to \infty.
	\end{align*}
	From  \eqref{6.20} we have $\|P_N Y- P_N \tilde{Y}^n\|_{L^2([0,T];\mathcal H^\delta)} \xrightarrow{\mathbb{P}(\cdot|A-A_1)} 0$. Also, since $\lim_{n\to \infty}\mathbb P_A(A_1) =0$, we conclude that
	\begin{align}\label{6.23}
		\limsup_{n\to \infty}\left|\mathbb E[F(P_N Y) -F(P_N Y^n)] \right| \leq 2 \sup_{u} F(u) \mathbb P(\Omega-A).
	\end{align}
	In the preceding proof, we have shown that $\mathbb P(\Omega-A)$  can be made arbitrarily small(by taking $M$ sufficiently large). Since the left side of \eqref{6.23} is independent of the choice of $A$, we have established \eqref{6.22}. This completes the proof of uniqueness in law for $P_N Y$.
	
	Finally, we establish uniqueness in law of $Y(t)$. To this end, we need the following key property.
	\begin{proposition}\label{prop 6.3}
		For any fixed $T>0$, let $Y$ be an arbitrary weak mild solution of \eqref{1.2} in the case
		$\Psi\equiv 0$. Then, for each $\delta \in [0,\frac{1}{2})$,
		\begin{align}\label{6.24}
			\lim_{N\to \infty}\sup_{Y}\mathbb E \left[ \int_0^T \|Y(t) -P_NY(t)\|^2_{\mathcal H^\delta}\, \mathrm{d}t \right] =0.
		\end{align}
		\begin{proof}
			Any weak mild solution of \eqref{1.2} with $\Psi \equiv 0$ is given by
			\begin{align*}
				Y(t) = \mathcal R_{\gamma}(t)(u_0,u_1) + \int_0^t \mathcal R_{\gamma}(t-\zeta)\Phi(\mathcal I_1 Y(\zeta)) \, \mathrm{d}W(\zeta).
			\end{align*}
			It suffices to replace \eqref{6.1} in the proof of Proposition \ref{prop 6.2}(i) by the linear growth condition on $\Phi$. Then by an analogous argument, we obtain that for arbitrary $\delta \in [0,\frac{1}{2})$ and $T>0$,
			\begin{align*}
				\sup_{Y}\mathbb E \left[\|Y\|_{L^2([0,T];\mathcal H^{\delta_1})} \right]<\infty.
			\end{align*}
			This implies the existence of $\delta_1 \in (\delta, \frac{1}{2})$ such that
			\begin{align*}
				\sup_{Y}\mathbb E \left[ \int_0^T \|Y(t) -P_NY(t)\|^2_{\mathcal H^\delta}\, \mathrm{d}t \right] &\leq \lambda_{N}^{ \frac{\delta-\delta_1}{2} } \sup_{Y}\mathbb E \left[ \int_0^T  \|Y(t) -P_NY(t)\|^2_{\mathcal H^{\delta_1}}\, \mathrm{d}t \right]\\
				&\lesssim \lambda_{N}^{ \frac{\delta-\delta_1}{2} } \sup_{Y}\mathbb E \left[ \int_0^T  \|\Theta_{\gamma}(t)\|^2_{\mathcal H^{\delta_1}}\, \mathrm{d}t \right]\to 0
			\end{align*}
			as $N \to \infty$. Therefore, \eqref{6.24} holds.
		\end{proof}
	\end{proposition}
	
	Let $Y_1(t)$ and $Y_2(t)$ be any two weak mild solutions of \eqref{1.2} with $\Psi \equiv 0$. Then
	\begin{align}\label{6.25}
		\begin{aligned}
			\left|\mathbb E[F(Y_1) - F(Y_2)] \right| &\leq \left|\mathbb E[F(Y_1) - F(P_NY_1)] \right|+\left|\mathbb E[F(P_NY_2) - F(Y_2)] \right|\\
			&\quad + \left|\mathbb E[F(P_NY_1) - F(P_NY_2)] \right| \to 0 \text{ ~as~ } N \to \infty.
		\end{aligned}
	\end{align}
	By Proposition \ref{prop 6.3} and Markov's inequality, we have
	$P_NY_i \xrightarrow{\mathbb P} Y_i$ in $L^2([0,T];\mathcal H^\delta)$ for $i=1,2$. Continuity of $F$ and the continuous mapping theorem give
	$F(P_NY_i) \xrightarrow{\mathbb P} F(Y_i)$. Since $F$ is bounded, the dominated convergence shows the first two terms on the right-hand side of
	\eqref{6.25} tend to $0$ as $N\to \infty$. The third term on the right-hand side of \eqref{6.25} vanishes by the uniqueness in law of $P_N Y_i$. Thus the proof of uniqueness in law for \eqref{1.2} is complete.
	
	\section{Further applications of the novel coupling method}\label{sec 7}

	\subsection{General measure-valued memory kernels}
	We first give concrete examples of kernels satisfying $(\mathcal{PC}_\varepsilon)$ to show the wide applicability of Theorem \ref{the 6.1}.
	
	\begin{example}\label{eg 7.1}
		Classical time-fractional case.\rm{
			This type of kernel is usually used to describe single memory genetic effects. Consider the pair $(g_{1-\alpha},g_\alpha)$ with $\alpha \in (\frac{1}{2},1)$, where $g_\alpha$ is the Riemann-Liouville kernel. Then $g_{1-\alpha} \in (\mathcal{PC}_{\varepsilon})$ with $\varepsilon \in (0,\frac{2\alpha-1}{1-\alpha})$. In this case, \eqref{1.2} becomes
			\begin{align*}
				\begin{cases}
					\gamma \partial^\alpha_t \left( \partial^\alpha_t v(t,x)\right)=\Delta v(t,x)-\partial^\alpha_t v(t,x) + \Psi(v(t,x) ) + \Phi(v(t,x)) \frac{\mathrm{d}W}{\mathrm{d}t}, \quad t>0, \ x \in  (0,1), \\
					v(t,0) =v(t,1) = 0, \quad t>0,\\
					(g_{1-\alpha} \ast v(\cdot,x))(0) =u_0(x), \  \left(g_{1-\alpha} \ast \partial^\alpha_t v(\cdot,x)\right)(0)=u_1(x), \quad x \in (0,1),
				\end{cases}
			\end{align*}
			where $\partial^\alpha_t$ is the Riemann-Liouville fractional derivative.}
	\end{example}
	
	\begin{example}\label{eg 7.2}
		The multi-term time-fractional case.
		\rm{This kernel type describes complex transport in media with multi-scale memory. Consider the pair $a(t)=\sum_{i=1}^N g_{1-\alpha_i}(t)$ with $\min_{1\leq i\leq N}\{\alpha_i\}>\frac{1}{2}$. Then $a \in (\mathcal{PC}_\varepsilon)$ with $\varepsilon \in (0, \frac{2\min_{1\leq i\leq N}\{\alpha_i\}-1}{1-\min_{1\leq i\leq N}\{\alpha_i\}})$.}
	\end{example}
	
	\begin{example}\label{eg 7.3}
		The time-fractional case with weight.
		\rm{Kernels of this form are used to characterize memory dissipation with tunable intensity in layered viscoelastic media and composite disordered materials. Let $0<\alpha<\beta<\frac{1}{2}$ and $\nu>0$. Consider the kernel $a(t)=g_\beta(t)E_{\alpha,\beta}(-\nu t^\alpha), \quad t>0$, where $E_{\alpha,\beta}$ is the  generalized Mittag-Leffler function. Then $a \in (\mathcal{PC}_\varepsilon)$ with $\varepsilon \in (0,\frac{1-2\beta}{\beta})$.}
	\end{example}
	
	\begin{example}\label{eg 7.4}
		Ultraslow diffusion.\rm{This kernel describes anomalous transport with logarithmic slow evolution in heterogeneous media. Consider $a(t)= \int_0^t g_{\alpha}(t)\, \mathrm{d}t$. Then $a \in (\mathcal{PC}_\varepsilon)$ for any $\varepsilon>0$, which means that $\varepsilon$ in \eqref{2.1} can be chosen arbitrarily large. Consequently, the admissible H\"{o}lder exponent for the diffusion coefficient in Theorem \ref{the 6.1} can be made arbitrarily close to $\frac{3}{5}$ from above.}
	\end{example}
	
	Verification of the $(\mathcal{PC}_\varepsilon)$ condition for these kernels requires their asymptotic behavior at $t=0$, which follows from the Karamata-Feller Tauberian theory (see \cite[Chapter XIII]{W. Feller}). We omit the details. More examples can be found in \cite{R. L. Schilling, F. Alegria 23}.

	\subsection{General drift and diffusion coefficients}
	
	The condition $(\mathcal{PC}_\varepsilon)$ is used to simplify the model and isolate memory-driven transport effects. In fact, all results extend directly to measure-valued kernels, with the same proof ideas and no need for modification. Specifically, consider the following problem
	\begin{align}\label{7.1}
		\begin{cases}
			\gamma \partial_t \left( \mathrm{d}a \ast \partial_t(\mathrm{d}a \ast v(\cdot,x))\right)(t)=\Delta v(t,x)-\partial_t \left(\mathrm{d}a \ast v(\cdot,x) \right)(t) \\
			\quad\quad\quad\quad+ \Psi(v(t,x) ) + \Phi(v(t,x)) \frac{\mathrm{d}W}{\mathrm{d}t}, \quad t>0, \ x \in  (0,1), \\
			v(t,0) =v(t,1) = 0, \quad t>0,\\
			(\mathrm{d}a \ast v)(0) =u_0(x), \  \left(\mathrm{d}a \ast \partial_t(\mathrm{d}a \ast v(\cdot,x))\right)(0)=u_1(x), \quad x \in (0,1),
		\end{cases}
	\end{align}
	where $a$ satisfies
	\begin{enumerate}
		\item [$(\mathcal{PC}^\ast_\varepsilon)$]
		he function $a$ admits the representation
		$a(t) = a_0 + \int_0^t a_1(\zeta) \, \mathrm{d}\zeta$ for $t >0$, where $a_0 \geq 0$, $a_1 \in L^1_{loc}(\mathbb R_+)$ is nonnegative and nonincreasing, and there exists $b \in L_{\text{loc}}^{2+\varepsilon}(\mathbb{R}_+)$ for some $\varepsilon>0$, is also nonincreasing, such that $a_0 b(t) + (a_1 \ast b)(t) = 1$ for $t>0$.	
	\end{enumerate}
	This equation includes both an instantaneous local feedback and a hereditary relaxation integral, capturing both instantaneous constitutive responses and memory-coupled signal transmission with delay. Replacing $(\mathcal{PC}_\varepsilon)$ with $(\mathcal{PC}^\ast_\varepsilon)$ in Theorem \ref{the 6.1} does not affect the validity of the conclusion. In particular, if $a_0>0$, then $b \in L^{2+\varepsilon}_{loc}(\mathbb R_+)$ for any $\varepsilon >0$. Hence the admissible H\"{o}lder exponent of the diffusion coefficient can also approach
	$\frac{3}{5}$ arbitrarily from above. 	
	
	For example, if $a$ is the Heaviside function, i.e., $\mathrm{d}a(t)=\delta_0(t)\mathrm{d}t$, then \eqref{7.1} reduces to IBVP for the damped stochastic wave equation. In \cite{Y. Han}, the author ingeniously combined phase-lifting with the GC-D method to study the following damped stochastic wave equation
	\begin{align}\label{7.2}
		\begin{cases}
			\gamma \partial_t^2   v(t,x)=\Delta v(t,x)-\partial_t v(t,x) + B(t,v(t,x))
			+ G(t,v(t,x))\frac{\mathrm{d}W}{\mathrm{d}t}, \\
			v(t,0) =v(t,1) = 0, \quad t>0,\\
			v|_{t=0} =v_0(x), \  \partial_t v|_{t=0}=v_1(x), \quad x \in (0,1).
		\end{cases}
	\end{align}
	Here $G(t,v)$ is a uniformly nondegenerate diffusion coefficient that is $\frac{3}{4}+\varsigma$-H\"{o}lder continuous with respect to $v$. He constructed weak mild solutions in law and extended the framework to SPDEs on separable Hilbert spaces with densely defined operators, which is a notable contribution. Since the nonlinearity depends only on the displacement $v(t)$, the mild solution for the displacement is decoupled from the velocity. Hence the full state
	$Y(t)=(v(t),\partial_t v(t))$ is determined by $v(t)=\mathcal I_1 Y(t)$, and it suffices to analyze only this component.
	
	Physically, however, linear viscous damping, nonlinear Forchheimer drag, and Coulomb dry friction are typical velocity-dependent nonlinearities common in real dynamical systems. This motivates a framework in which the drift and diffusion depend on both $v$ and $\partial_t v$, as formulated in the following problem
	\begin{align}\label{7.3}
		\begin{cases}
			\gamma \partial_t^2   v(t,x)=\Delta v(t,x)-\partial_t v(t,x) + \Psi(v,\partial_tv)
			+ \Phi(v,\partial_tv) \frac{\mathrm{d}W}{\mathrm{d}t}, \\
			v(t,0) =v(t,1) = 0, \quad t>0,\\
			v|_{t=0} =v_0(x), \  \partial_t v|_{t=0}=v_1(x), \quad x \in (0,1).
		\end{cases}
	\end{align}
	In this setting, the dynamics of the displacement and the velocity become coupled, necessitating a direct analysis of the full state $Y(t)=(v(t),\partial_t v(t))$.
	
	It is worth noting that both the GC-D method of \cite{Y. Han} and the present GC-V method are based on power-law decay estimates for the damping factor.
	The core difference is that the former fails to obtain power-law decay estimates for the velocity component in the moment estimates. Specifically, the former requires dealing with the core function
	\begin{align*}
		f_{\gamma,\lambda_j+\xi,}(t)= e^{ -\frac{t}{2\gamma}}\sin(  \frac{\sqrt{4\gamma(\lambda_j+\xi)-1} }{2\gamma}  ).
	\end{align*}
	As shown in \cite{Y. Han}, a power-law damping factor can be extracted from the estimates of $f_{\gamma,\lambda_j}$. For the velocity case, however, such a factor would require a $\xi^{-\alpha}$ factor from the estimates of $f'_{\gamma,\lambda_j+\xi}(t)$,  which is impossible. Indeed, for fixed $t$, $\gamma$ and $\lambda_j$, one checks that
	$\lim_{\xi \to \infty} |f'_{\gamma,\lambda_j +\xi}(t)|$ does not exist. Thus, the GC-D method fails for nonlinearities depending on both displacement and velocity. In contrast, our GC-V method extracts decay factors of type $(1+\xi)^{-\alpha}$ in moment estimates for both components, see Lemma \ref{lem 5.1} and Lemma \ref{lem 5.2}. Moreover, our proof framework for existence and uniqueness is established directly for the full state variable $Y(t)$, and can be adapted to the analysis of problem \eqref{7.3}. One only needs to set $V= L^2([0,T];\mathcal H)$ and $V'=L^2([0,T];\mathcal L(H))$ in Proposition \ref{prop 6.1} and adjust all the corresponding estimates accordingly to obtain the following conclusion.
	
	\begin{definition}\label{def 7.1}
		Let $\gamma>0$. A weak mild solution to \eqref{7.3}
		with initial pair $(v_0, v_1) \in \mathcal{H}^1$ as a collectionn $ (\Omega, \mathcal{F}, (\mathcal{F}_t), \mathbb{P}, W, Y)$, subject to the following conditions:
		\begin{itemize}
			\item[(i)] $(\Omega, \mathcal{F}, (\mathcal{F}_t), \mathbb{P})$ is some filtered probability space;
			\item[(ii)] $W$ is a cylindrical Wiener process with respect to this filtration;
			\item[(iii)] $Y$ is a progressively measurable process taking values in $L_{loc}^2(\mathbb R_+;\mathcal H^\delta)$ for some $\delta \in [0,\frac{1}{2})$;
			\item[(iv)] $\mathbb{P}$-a.s, the mild solution $Y(t)$ satisfies
			\begin{align*}
				\begin{aligned}
					Y(t) &= \mathcal S_{\gamma}(t)(v_0,v_1) + \int_0^t \mathcal S_{\gamma}(t-\zeta) J_\gamma \Psi(Y(\zeta))\, \mathrm{d}\zeta \\
					&\quad+ \int_0^t \mathcal S_{\gamma}(t-\zeta) J_\gamma \Phi( Y(\zeta)) \, \mathrm{d}W(\zeta), \text{ ~a.a.~ } t \in \mathbb R_+,
				\end{aligned}
			\end{align*}
			where $S_\gamma(t)$ denotes the operator introduced in \cite{Y. Han}, which is written as
			$\mathcal S_\mu(t)$.
		\end{itemize}
	\end{definition}
	
	\begin{corollary}\label{cor 7.1}
		Suppose that the following assumptions are satisfied.
		\begin{enumerate}
			\item[$(\mathcal M'_1)$]The map $\Psi: L_{loc}^2(\mathbb R_+;\mathcal H) \to L_{loc}^2(\mathbb R_+; H)$ is casual in the sense that, for any $T>0$ and any $Y, Z \in L_{loc}^2(\mathbb R_+;\mathcal H)$, if $Y=Z$ a.e. on $[0,T]$, then $\Psi(Y) = \Psi(Z)$ a.e. on $[0,T]$ in $H$.
			Moreover, for each $T>0$, the induced map $\Psi_T: L^2([0,T];\mathcal H) \to L^2([0,T]; H)$ is Borel measurable.
			
			\item[($\mathcal M'_2$)] For each $Y \in \mathcal  H$, $\Phi(Y) \in \mathcal L(H)$ and $\Phi(Y)$ has a right inverse $\Phi^{-1}(Y)$ on $H$ in the sense that $\Phi(Y)\Phi^{-1}(Y)u = u$ for any $u \in H$,
			with $\Phi^{-1}(Y)H \subset H$ and satisfies, for some $C>0$,
			\begin{align*}
				\sup_{Y \in \mathcal H} |\Phi^{-1}(Y)|_{\mathcal L( H)} \leq C < \infty.
			\end{align*}
			Moreover, $\Phi$ is self-adjoint, in the sense that for any $Y \in \mathcal H$, $\langle \Phi(Y)e_i,e_j\rangle =\langle e_i,\Phi(Y)e_j\rangle $ for any $i,j \in \mathbb N_+$.
			
			\item[$(\mathcal M'_3)$] $\Phi$
			admits an extension to a map
			$\Phi: L_{loc}^2(\mathbb R_+;\mathcal H)  \to L_{loc}^2(\mathbb R_+; \mathcal L(H))$, and is assumed to be
			$\kappa$-H\"{o}lder continuous in the sense that, for each $T>0$, there exists a constant $C_T$ such that
			\begin{align*}
				\|\Phi(Y) - \Phi(Z)\|_{L^2([0,T];{\mathcal L}(H)) } \leq C_T\|Y - Z\|_{L^2([0,T];\mathcal H)}^\kappa
			\end{align*}
			for any $Y, Z \in L_{loc}^2(\mathbb R_+;\mathcal H)$, where $\kappa \in (\frac{3}{5},1]$.
			Additionally, there exists a constant $C_T'>0$, such that
			\begin{align*}
				&\quad \|P_N\Phi(Y) - P_N\Phi(Z)\|_{L^2([0,T];{\mathcal L}(H)) }\\
				&\leq C_T'  \|\Phi(P_NY) - \Phi(P_NZ)\|_{L^2([0,T];{\mathcal L}(H))}, \quad  Y,Z \in L_{loc}^2(\mathbb R_+;\mathcal H),
			\end{align*}
			holds uniformly for all sufficiently large $N\in \mathbb N_+$.
			
			\item[$(\mathcal M'_4)$]For each $T>0$, there exist some constant $C_{\Psi,T}>0$ and $C_{\Phi,T}>0$ such that
			\begin{align*}
				&\|\Psi(Y)\|_{L^2([0,T];H)} \leq C_{\Psi,T}(1 + \|Y\|_{L^2([0,T];\mathcal H)}),\\
				&\|\Phi(Y)\|_{L^2([0,T],\mathcal L(H) )} \leq C_{\Phi,T}(1 + \|Y\|_{L^2([0,T];\mathcal H)}).
			\end{align*}
		\end{enumerate}
		Then, for any initial datum $(v_0, v_1) \in \mathcal H^1$, \eqref{7.3} admits a weak mild solution in the sense of Definition \ref{def 7.1}, where $\delta$ can be chosen arbitrarily close to $\frac{1}{2}$ from below, and the solution is unique in law.
	\end{corollary}
	
	\begin{remark}\label{rem 7.1}
		Since the solution operator $\mathcal S_\gamma(t)$ has no singularity at $t=0$, the coupling analysis method developed in this paper can also be applied to study the weak existence and uniqueness in law of the mild solution to \eqref{7.3} as $C([0,T];\mathcal H^\delta)$-valued processes.
	\end{remark}
	
	Below we provide an example of a diffusion coefficient satisfying all the assumptions above.
	\begin{example}\label{eg 7.5}
		Multi-scale spectral multiplicative diffusion coefficient. \rm{For each $(v,w)\in \mathcal H$, define $\Phi(v,w) \in \mathcal L(H)$ by
			\begin{align*}
				\Phi(v,w)u = \sum_{i,j=1}^{N_0} c_{i,j}(\langle v,e_i \rangle, \langle w,e_i \rangle )\langle u,e_j \rangle e_j  +  \sum_{ j > N_0} f_j(\langle v,e_j \rangle, \lambda_j^{-\frac{1}{2}} \langle w,e_j \rangle)\langle u,e_j \rangle e_j,
			\end{align*}
			where $N_0 \in \mathbb N_+$ is some fixed constant. The following assumptions are satisfied.
			\begin{itemize}
				\item [(i)] The low-frequency matrix  $\{c_{ij}\}$ is self-adjoint and either uniformly positive definite or uniformly negative definite;
				\item [(ii)] For some $C_0>0$ independent of $j$, the functions $f_j: \mathbb R \times \mathbb R \to \mathbb R \setminus (-C_0, C_0)$ are uniformly H\"{o}lder continuous in $j$ and
				satisfy a uniform linear growth condition, i.e., there exists a constant $C>0$, independent of $j$, such that
				\begin{align*}
					|f(y_1,z_1)-f(y_2,z_2)| \leq C (|y_1-y_2|^\kappa + |z_1-z_2|^\kappa), \ |f(y,z)|\leq C(1+|y|+|z|),
				\end{align*}
				where $\kappa \in (\frac{3}{5},1]$.
		\end{itemize}}
	\end{example}
	
	\subsection{Open problem}
	The IBVP for the classical stochastic damped wave equation is included in model \eqref{7.1}. For such problems, the Smoluchowski-Kramers approximation yields the standard asymptotic scenario: when the inertia parameter $\gamma \to 0$, the inertial effects become negligible, and the second-order dynamics reduce to a first-order diffusion equation in the asymptotic limit. For example, Salins \cite{M. Salins} developed the Smoluchowski-Kramers approximation for stochastic damped wave equations with multiplicative noise in any dimension. Under standard noise regularity, he proved $L^p$-uniform convergence of solutions to the corresponding heat equation in the small-mass limit, and established well-posedness of mild solutions. In \cite{Y. Han}, the author also proved that, as the mass density tends to zero, the weak mild solutions of the stochastic damped wave equation converge in distribution in the space of continuous paths to the solutions of the corresponding heat equation. 
	
	However, in the presence of more general temporal memory effects, it remains an open question whether the Smoluchowski-Kramers approximation analysis can be extended to time-nonlocal telegraph equations. Specifically, for any $\varepsilon_1, \varepsilon_2 \in BV_{loc}(\mathbb R_+)$ satisfying $\mathrm{d}\varepsilon_1 \ll \mathrm{d}\varepsilon_2$, does the voltage evolution depart from the wave propagation regime and become governed by charge relaxation and polarization diffusion? Furthermore, as
	$\gamma \to 0$, will the solutions of \eqref{1.2} converge, in some sense, to the solution of the following nonlocal diffusion problem?
	\begin{align*}
		\begin{cases}
			\partial_t \left(a \ast v(\cdot,x) \right)(t) - \Delta v(t,x)= \Psi(v(t,x) ) + \Phi(v(t,x)) \frac{\mathrm{d}W}{\mathrm{d}t}, \quad t>0, \ x\in(0,1), \\
			v(t,0) =v(t,1) = 0, \quad t>0, \\
			(a \ast v(\cdot,x))(0)=u_0(x), \quad x \in (0,1).
		\end{cases}
	\end{align*}
	If such a convergence mode exists, does the conclusion hold for all kernels satisfying $(\mathcal{PC}_\varepsilon)$, or does it exhibit a critical phenomenon with respect to the integrability exponent $\varepsilon$, or does the convergence behavior reveal an entirely new structure depending on the range of
	$\varepsilon$ that has not yet been explored? In subsequent work, we will carry out a systematic investigation to elucidate the essential mechanisms underlying nonlocal telegraph equations and to consolidate the theoretical foundation of this field.
	
	\appendix
	\section*{Appendix}
	\section{Critical function power bound against damping factor}\label{app A}
	\subsection{Proof of Proposition \ref{prop 4.1}}
	\begin{proof}
		Clearly, for $(1+\xi)^2 -4\gamma\lambda_j >0$, we have
		\begin{align}\label{A.1}
			\frac{1+\xi-\sqrt{(1+\xi)^2-4\gamma\lambda_j}}{2\gamma} =  \frac{2\lambda_j }{1+\xi+\sqrt{(1+\xi)^2-4\gamma\lambda_j}}>\frac{\lambda_j}{1+\xi}
		\end{align}
		Below, we discuss five cases according to the comparison between $(1+\xi)^3$ and $4\gamma\lambda_j$. In particular, the first two cases are analyzed under the condition $\alpha \in [-1,1]$, while the remaining three cases are treated under the assumption $\alpha \in [0,1]$.
		
		If $(1+\xi)^2 >4\gamma\lambda_j > \frac{1}{2}(1+\xi)^2$, by the mean value theorem and \eqref{A.1}, we have $|g_{\gamma,\lambda_j,\xi}(\theta)|\leq \theta e^{ -\frac{\lambda_j}{1+\xi}\theta }$. Note that $\sup\limits_{\theta \in \mathbb R_+} \theta e^{-\frac{\lambda_j}{2(1+\xi)}\theta}\leq\frac{2(1+\xi)}{e\lambda_j}$.Then we obtain
		\begin{align}\label{A.2}
			|g_{\gamma,\lambda_j,\xi}(\theta)| \leq  \frac{2(1+\xi)}{e\lambda_j} e^{-\frac{\lambda_j}{2(1+\xi)}\theta}= \frac{2}{e(1+\xi)^\alpha}\frac{1}{\lambda_j^{\frac{1-\alpha}{2}  } } \left[\frac{(1+\xi)^2}{\lambda_j}\right]^{\frac{1+\alpha}{2} } e^{-\frac{\lambda_j}{2(1+\xi)}\theta}.
		\end{align}
		Note that $4\gamma\lambda_j>\frac{1}{2}(1+\xi)^2$ implies $\frac{(1+\xi)^2}{\lambda_j}< 8\gamma$. This along with \eqref{A.2} gives
		\begin{align}\label{A.3}
			|g_{\gamma,\lambda_j,\xi}(\theta)| \leq 2\frac{8^{ \frac{1+\alpha}{2} }  }{e(1+\xi)^\alpha}\frac{\gamma^{\frac{1+\alpha}{2} } }{\lambda_j^{\frac{1-\alpha}{2}  } } e^{-\frac{\lambda_j}{2(1+\xi)}\theta}\leq 6  (1+\xi)^{-\alpha}\gamma^{\frac{1+\alpha}{2} } \lambda_j^{-\frac{1-\alpha}{2}  } e^{-\frac{\lambda_j}{2(1+\xi)}\theta}.
		\end{align}
		
		If $ \frac{1}{2}(1+\xi)^2 \geq 4\gamma\lambda_j$, we have
		\begin{align*}
			&((1+\xi)^2-4\gamma\lambda_j)^{\frac{\alpha}{2}} \geq {2}^{ -\frac{\alpha}{2} } (1+\xi)^\alpha \text{ ~for~ } \alpha \in [0,1],\\
			&((1+\xi)^2-4\gamma\lambda_j)^{\frac{\alpha}{2}} \geq (1+\xi)^\alpha \text{ ~for~ } \alpha \in [-1,0).
		\end{align*}
		Therefore,
		\begin{align}\label{A.4}
			\begin{aligned}
				\sqrt{(1+\xi)^2 -4\gamma \lambda_j} &=((1+\xi)^2-4\gamma\lambda_j)^{ \frac{\alpha}{2} } ((1+\xi)^2-4\gamma\lambda_j)^{\frac{1-\alpha}{2}} \\
				&\geq \min\{ 1, 2^{-\frac{\alpha}{2}} \}(1+\xi)^\alpha (4\gamma\lambda_j)^{\frac{ 1-\alpha}{2}}\\
				&\geq \frac{1}{\sqrt{2} }(1+\xi)^\alpha \gamma^{ \frac{1-\alpha}{2} } \lambda_j^{ \frac{1-\alpha}{2}}.
			\end{aligned}
		\end{align}
		Also, clearly, in this case
		the following holds:
		\begin{align*}
			|g_{\gamma,\lambda_j,\xi}(\theta)| \leq \frac{2\gamma}{\sqrt{(1+\xi)^2 - 4\gamma\lambda_j}}e^{-\frac{\lambda_j}{2(1+\xi)}\theta}.
		\end{align*}
		Combining the above with \eqref{A.4} yields \eqref{A.3}.
		
		When $(1+\xi)^2= 4\gamma\lambda_j$, note that $\sup\limits_{\theta \in \mathbb R_+} \theta e^{-\frac{1+\xi}{4\gamma}\theta}\leq\frac{4\gamma}{e(1+\xi)}$. Consequently,
		\begin{align*}
			|g_{\gamma,\lambda_j,\xi}(\theta)|&\leq \frac{4\gamma}{e(1+\xi)}e^{-\frac{1+\xi}{4\gamma}\theta}= \frac{2}{2^{ 1-\alpha}e}(1+\xi)^{-\alpha} \gamma^{ \frac{1+\alpha}{4} } \lambda_j^{-\frac{1-\alpha}{2}}e^{-\frac{1+\xi}{4\gamma}\theta}\\
			&\leq 2(1+\xi)^{-\alpha} \gamma^{ \frac{1+\alpha}{2} } \lambda_j^{-\frac{1-\alpha}{2}}e^{-\frac{1+\xi}{4\gamma}\theta}.
		\end{align*}
		
		In the case where $2(1+\xi)^2>4\gamma\lambda_j>(1+\xi)^2$, it is clear that $\frac{\gamma}{1+\xi}<\frac{1+\xi}{2\lambda_j}$. Combining this with
		$|\sin x| \leq |x|$ and $\sup\limits_{\theta \in \mathbb R_+} \theta e^{-\frac{1+\xi}{4\gamma}\theta}\leq\frac{4\gamma}{e(1+\xi)}$ yields
		\begin{align*}
			|g_{\gamma,\lambda_j,\xi}(\theta)|\leq \theta e^{-\frac{1+\xi}{2\gamma}} \leq \frac{4\gamma}{e(1+\xi)}e^{-\frac{1+\xi}{4\gamma}\theta}\leq \frac{2(1+\xi)}{e\lambda_j}e^{-\frac{1+\xi}{4\gamma}\theta}.
		\end{align*}
		A similar derivation as in \eqref{A.2} applied to the above gives
		\begin{align}\label{A.5}
			|g_{\gamma,\lambda_j,\xi}(\theta)| \leq \frac{2}{e(1+\xi)^\alpha}\frac{1}{\lambda_j^{\frac{1-\alpha}{2}  } } \left[\frac{(1+\xi)^2}{\lambda_j}\right]^{\frac{1+\alpha}{2} } e^{-\frac{\lambda_j}{2(1+\xi)}\theta}.
		\end{align}
		Observe that from $4\gamma\lambda_j>(1+\xi)^2$ it follows that $\frac{(1+\xi)^2}{\lambda_j}<4\gamma$. This, combined with \eqref{A.5}, yields
		\begin{align}\label{A.6}
			|g_{\gamma,\lambda_j,\xi}(\theta)| \leq 3  (1+\xi)^{-\alpha}\gamma^{\frac{1+\alpha}{2} } \lambda_j^{-\frac{1-\alpha}{2}  } e^{-\frac{(1+\xi)}{4\gamma}\theta}.
		\end{align}
		
		Finally, consider the case $4\gamma\lambda_j\geq 2(1+\xi)^2$. Note that
		\begin{align*}
			&(4\gamma\lambda_j-(1+\xi)^2)^{\frac{\alpha}{2}} \geq  (1+\xi)^\alpha \text{ ~for~ } \alpha \in [0,1].
		\end{align*}
		Thus,
		\begin{align}\label{A.7}
			\begin{aligned}
				\sqrt{4\gamma \lambda_j- (1+\xi)^2} &=(4\gamma\lambda_j-(1+\xi)^2)^{ \frac{\alpha}{2} } (4\gamma\lambda_j-(1+\xi)^2)^{\frac{1-\alpha}{2}} \\
				&\geq \min\{ 1,2^{ \frac{\alpha}{2} } \}(1+\xi)^\alpha (2\gamma\lambda_j)^{\frac{ 1-\alpha}{2}}\\
				&\geq (1+\xi)^\alpha \gamma^{ \frac{1-\alpha}{2} } \lambda_j^{ \frac{1-\alpha}{2}}.
			\end{aligned}
		\end{align}
		Also, since $|\sin x| \leq 1$, in this case we have
		\begin{align*}
			|g_{\gamma,\lambda_j,\xi}|\leq \frac{2\gamma}{\sqrt{4\gamma\lambda_j-(1+\xi)^2} }e^{-\frac{1+\xi}{4\gamma}\theta }
		\end{align*}
		This, together with \eqref{A.7}, gives \eqref{A.6}.
	\end{proof}
	
	\subsection{Proof of Proposition \ref{prop 4.2}}
	\begin{proof}
		The proof is divided into three cases.
		
		If $\frac{1}{2}(1+\xi)^2 \geq 4\gamma\lambda_j$, then
		\begin{align*}
			\frac{\gamma}{\sqrt{(1+\xi)^2-4\gamma\lambda_j } } \leq \sqrt{2}\frac{\gamma}{1+\xi}
		\end{align*}
		and
		\begin{align*}
			&\quad |Z_1e^{-Z_1\theta} - Z_2e^{-Z_2\theta}|\\
			&\leq \frac{2\lambda_j}{1+\xi + \sqrt{(1+\xi)^2- 4\gamma\lambda_j} }e^{ -\frac{\lambda_j}{1+\xi}\theta }  + \frac{1+\xi+\sqrt{(1+\xi)^2 -4\gamma\lambda_j } }{2\gamma} e^{ -\frac{1+\xi}{2\gamma}\theta }\\
			&\leq \frac{2\sqrt{2} }{1+\sqrt{2}} \frac{\lambda_j}{1+\xi}e^{ -\frac{\lambda_j}{1+\xi}\theta } + \frac{1+\xi}{\gamma} e^{ -\frac{1+\xi}{2\gamma}\theta } .
		\end{align*}
		These two inequalities imply
		\begin{align*}
			|g'_{\gamma,\lambda_j,\xi}(\theta)| &\leq \frac{1}{2+2\sqrt{2}} \frac{8\gamma\lambda_j}{(1+\xi)^2} e^{ -\frac{\lambda_j}{1+\xi}\theta }+ \sqrt{2} e^{ -\frac{1+\xi}{2\gamma}\theta }
			\\
			&\leq 2 (1+\xi)^{-\alpha} \gamma^{ \frac{\alpha}{2} } \lambda_j^{ \frac{\alpha}{2}}e^{ -\frac{\lambda_j}{1+\xi}\theta } +\sqrt{2} e^{ -\frac{1+\xi}{4\gamma}\theta }, \quad \theta>0.
		\end{align*}
		
		In the case where $(1+\xi)^2 > 4\gamma\lambda_j>\frac{1}{2}(1+\xi)^2$, by the mean value theorem, there exists some $Z\in(Z_1, Z_2)$  such that
		\begin{align*}
			\frac{|Z_1e^{-Z_1\theta} - Z_2e^{-Z_2\theta}|}{|Z_1 -Z_2|} = |e^{-Z\theta }( 1-Z\theta )|\leq e^{-\frac{Z_1\theta }{2}}|e^{-\frac{Z\theta}{2} }( 1-Z\theta )|\leq e^{-\frac{\lambda_j}{2(1+\xi)}\theta} .
		\end{align*}
		Observe that from
		$4\gamma\lambda_j>\frac{1}{2}(1+\xi)^2$ it follows that $\lambda_j^{ -\frac{\alpha}{2} } \leq(8\gamma)^{\frac{\alpha}{2} }(1+\xi)^{ -\alpha }$. Consequently,
		\begin{align}\label{A.8}
			|g'_{\gamma,\lambda_j,\xi}(\theta)|\leq  \lambda_j^{-\frac{\alpha}{2}}  \lambda_j^{ \frac{\alpha}{2} } e^{-\frac{\lambda_j}{2(1+\xi)}\theta} \leq 2\sqrt{2} (1+\xi)^{ -\alpha } \gamma^{ \frac{\alpha}{2} } \lambda_j^{ \frac{\alpha}{2} }e^{-\frac{\lambda_j}{2(1+\xi)}\theta}.
		\end{align}
		
		Finally, one readily verifies that in the case $(1+\xi)^2 \leq 4\gamma\lambda_j$,
		\begin{align*}
			|g'_{\gamma,\lambda_j,\xi}(\theta)|\leq e^{-\frac{1+\xi}{4\gamma}\theta} \sup_{\theta \in \mathbb R_+} e^{-\frac{1+\xi}{4\gamma}\theta} (1+\frac{1+\xi}{2\gamma}\theta) \leq \sqrt{2} e^{-\frac{1+\xi}{4\gamma}\theta}.
		\end{align*}
		Therefore, by an argument analogous to that leading to \eqref{A.8}, we obtain
		\begin{align*}
			|g'_{\gamma,\lambda_j,\xi}(\theta)|\leq 2\sqrt{2} (1+\xi)^{ -\alpha } \gamma^{ \frac{\alpha}{2} } \lambda_j^{ \frac{\alpha}{2} } e^{-\frac{1+\xi}{4\gamma}\theta}.
		\end{align*}
	\end{proof}
	
	\subsection{Proof of Lemma \ref{lem 4.2}}
	\begin{proof}
		First prove (i). For $(1+\xi)^2 -4\gamma\lambda_j >0$, \eqref{4.1} and Proposition \ref{prop 4.1} give
		\begin{align}\label{A.9}
			\begin{aligned}
				|H_{\gamma,\lambda_j,\xi}(t;0,\gamma^{-1},0)(t)| &\leq  \gamma^{-1}\int_0^\infty |g_{\gamma,\lambda_j,\xi}(\theta)| \eta(t,\mathrm{d}\theta)\\
				&\leq 6(1+\xi)^{-\alpha_1} \gamma^{-\frac{1-\alpha_1}{2}}\lambda_j^{- \frac{1-\alpha_1}{2}}\int_0^\infty e^{ - \frac{\lambda_j}{2(1+\xi)}\theta } \eta(t,\mathrm{d}\theta), \quad\\
				&= 6(1+\xi)^{-\alpha_1} \gamma^{-\frac{1-\alpha_1}{2}}\lambda_j^{- \frac{1-\alpha_1}{2}} r_{ \frac{\lambda_j}{2(1+\xi)} }(t), \quad  t>0.
			\end{aligned}
		\end{align}
		When $\beta_1 \in (0, \frac{\varepsilon}{2(2+\varepsilon)})$,
		applying Proposition \ref{prop 2.1}(iv), H\"{o}lder's inequality, and Young's inequality successively yields
		\begin{align*}
			\begin{aligned}
				&\quad \int_0^T r_{\frac{\lambda_j}{2(1+\xi)} }(t)^2\mathrm{d}t\\
				&\leq \int_0^T b(t)^2 [1+\frac{\lambda_j}{2(1+\xi)}(1\ast b)(t)]^{-2} \, \mathrm{d}t\\
				&\leq \|b\|^2_{L^{2+\varepsilon}([0,T]) } \left(\int_0^T [1+\frac{\lambda_j}{2(1+\xi)}(1\ast b)(t) ]^{- \frac{2(2+\varepsilon)}{\varepsilon} }   \, \mathrm{d}t \right)^{ \frac{\varepsilon}{2+\varepsilon}}\\
				&\leq C_{\beta_1} \|b\|^2_{L^{2+\varepsilon}([0,T])} \left(\int_0^T  [\frac{\lambda_j}{2(1+\xi)}(1\ast b)(t)]^{- \frac{2\beta_1(2+\varepsilon)}{\varepsilon} }  \, \mathrm{d}t \right)^{ \frac{\varepsilon}{2+\varepsilon}}\\
				&\leq C_{\beta_1}  (1+\xi)^{2\beta_1 } \lambda_j^{-2\beta_1} b(T)^{-\frac{\varepsilon}{2+\varepsilon} } \|b\|^2_{L^{2+\varepsilon}([0,T])} \left(\int_0^T b(t)[(1\ast b)(t)]^{- \frac{2\beta_1(2+\varepsilon)}{\varepsilon} }  \, \mathrm{d}t \right)^{ \frac{\varepsilon}{2+\varepsilon}}\\
				&\leq C_{\varepsilon,\beta_1}(1+\xi)^{2\beta_1} \lambda_j^{-2\beta_1} b(T)^{-\frac{\varepsilon}{2+\varepsilon}  } \|b\|^2_{L^{2+\varepsilon}([0,T])} [(1\ast b)(T)]^{\frac{\varepsilon}{2+\varepsilon}-2\beta_1}.
			\end{aligned}
		\end{align*}
		If $\beta_1=0$, it follows directly that $\int_0^T r_{ \frac{\lambda_j}{2(1+\xi) } }(t)^2\, \mathrm{d}t \leq \|b\|^2_{L^2([0,T])}$. Thus, for $\beta_1 \in [0, \frac{\varepsilon}{2(2+\varepsilon)})$,
		there is a positive constant $C_{\varepsilon,\beta_1,T}$ depending only on $\varepsilon$, $\beta_1$ and $T$, such that
		\begin{align}\label{A.10}
			\int_0^T r_{\frac{\lambda_j}{2(1+\xi)} }(t)^2\mathrm{d}t \leq C_{\varepsilon,\beta_1,T}(1+\xi)^{2\beta_1} \lambda_j^{-2\beta_1}.
		\end{align}
		This together with \eqref{A.9} yields \eqref{4.4}.
		
		We now prove \eqref{4.5}. Similarly, from \eqref{4.2} and Proposition \ref{prop 4.2}, we have
		\begin{align}\label{A.11}
			\begin{aligned}
				&\quad|(a\ast H_{\gamma,\lambda_j,\xi}(\cdot;0,\gamma^{-1},0))'(t)|\\
				&\leq 2\sqrt{2}(1+\xi)^{-\alpha_1} \gamma^{ \frac{\alpha_1}{2}-1 } \lambda_j^{ \frac{\alpha_1}{2} } \int_0^\infty e^{ -\frac{\lambda_j}{2(1+\xi)}\theta }\, \eta(t,\mathrm{d}\theta) +\sqrt{2}\gamma^{-1} \int_0^\infty e^{-\frac{1+\xi}{4\gamma} \theta} \eta(t,\mathrm{d}\theta)
				\\
				&\leq 2\sqrt{2}(1+\xi)^{-\alpha_1} \gamma^{ \frac{\alpha_1}{2}-1 } \lambda_j^{ \frac{\alpha_1}{2} } r_{ \frac{\lambda_j}{2(1+\xi)} }(t) + \sqrt{2} \gamma^{-1} r_{ \frac{1+\xi}{4\gamma} }(t).
			\end{aligned}
		\end{align}
		By the same argument as in the derivation of \eqref{A.10}, we obtain
		\begin{align*}
			\begin{aligned}
				\int_0^T r_{\frac{1+\xi}{4\gamma} }(t)^2\mathrm{d}t
				&\leq C_{\varepsilon,\beta_1,T}(1+\xi)^{-2\beta_1} \gamma^{2\beta_1}.
			\end{aligned}
		\end{align*}	
		Combining the above with \eqref{A.10} and \eqref{A.11} yields \eqref{4.5}.
		
		For assertion (ii), when $(1+\xi)^2 -4\gamma\lambda_j \leq 0$, \eqref{4.1} and Proposition \ref{prop 4.1} give
		\begin{align}\label{A.12}
			\begin{aligned}
				|H_{\gamma,\lambda_j,\xi}(t;0,\gamma^{-1},0)(t)|
				&\leq 3(1+\xi)^{-\alpha_2} \gamma^{-\frac{1-\alpha_2}{2}}\lambda_j^{- \frac{1-\alpha_2}{2}}\int_0^\infty e^{-\frac{1+\xi}{4\gamma}\theta } \eta(t,\mathrm{d}\theta), \quad\\
				&= 3(1+\xi)^{-\alpha_2} \gamma^{-\frac{1-\alpha_2}{2}}\lambda_j^{- \frac{1-\alpha_2}{2}} r_{ \frac{1+\xi}{4\gamma} }(t), \quad  t>0.
			\end{aligned}
		\end{align}	
		Similarly to the derivation of \eqref{A.10}, for $\beta_2 \in [0, \frac{\varepsilon}{2(2+\varepsilon)})$, we have
		\begin{align}\label{A.13}
			\begin{aligned}
				\int_0^T r_{\frac{1+\xi}{4\gamma} }(t)^2\mathrm{d}t
				&\leq C_{\varepsilon,\beta_2,T}(1+\xi)^{-2\beta_2} \gamma^{2\beta_2}.
			\end{aligned}
		\end{align}	
		The above estimate, together with \eqref{A.12}, gives \eqref{4.7}.
		
		Finally, we turn to the proof of \eqref{4.8}. In a similar manner, it follows from \eqref{4.2} and Proposition \ref{prop 4.2} that
		\begin{align*}
			\begin{aligned}
				|(a\ast H_{\gamma,\lambda_j,\xi}(\cdot;0,\gamma^{-1},0))'(t)|
				&\leq 2\sqrt{2}(1+\xi)^{ -\alpha_2 }   \gamma^{\frac{\alpha_2}{2}-1} \lambda_j^{ \frac{\alpha_2}{2} } \int_0^\infty   e^{-\frac{1+\xi}{4\gamma}\theta} \eta(t,\mathrm{d}\theta)\\
				&= 2\sqrt{2}(1+\xi)^{ -\alpha_2 }   \gamma^{\frac{\alpha_2}{2}-1} \lambda_j^{ \frac{\alpha_2}{2} }    r_{\frac{1+\xi}{4\gamma}}(t).
			\end{aligned}
		\end{align*}	
		This together with \eqref{A.13} yields \eqref{4.8}.
	\end{proof}
	
	\subsection{Proof of Lemma \ref{lem 4.3}}
	\begin{proof}
		We first prove \eqref{4.9} and \eqref{4.11}. In the case where $(1+\xi)^2-4\gamma\lambda_j>0$, , it follows from \eqref{4.1} and Proposition \ref{prop 4.1}(take $\alpha =1$) that
		\begin{align}\label{A.14}
			\begin{aligned}
				|H_{\gamma,\lambda_j,\xi}(t;0,\gamma^{-1},0)(t)|
				&\leq 6(1+\xi)^{-1} \int_0^\infty e^{ - \frac{\lambda_j}{2(1+\xi)}\theta } \eta(t,\mathrm{d}\theta), \quad\\
				&= 6(1+\xi)^{-1} r_{ \frac{\lambda_j}{2(1+\xi)} }(t), \quad  t>0.
			\end{aligned}
		\end{align}
		Note that $\delta\in [0, \frac{\varepsilon}{2+\varepsilon})$ implies $\frac{1+\delta}{2}\frac{2+\varepsilon}{1+\varepsilon} \in (0,1)$. Applying Proposition \ref{prop 2.1}(iv), H\"{o}lder's inequality, and Young's successively yields
		\begin{align*}
			\begin{aligned}
				\int_0^T r_{\frac{\lambda_j}{2(1+\xi)} }(t)\mathrm{d}t
				&\leq \int_0^{T} b(t)[1+ \frac{\lambda_j}{2(1+\xi)} (1\ast b)(t)]^{-1}\, \mathrm{d}t\\
				&\leq \|b\|_{L^{2+\varepsilon}([0,T]) } \left(\int_0^T [1+\frac{\lambda_j}{2(1+\xi)}(1\ast b)(t) ]^{- \frac{2+\varepsilon}{1+\varepsilon} }\, \mathrm{d}t \right)^{ \frac{1+\varepsilon}{2+\varepsilon}}\\
				&\leq C_\delta \|b\|_{L^{2+\varepsilon}([0,T])} \left(\int_0^T [\frac{\lambda_j}{2(1+\xi)}(1\ast b)(t)]^{- \frac{(1+\delta)(2+\varepsilon)}{2(1+\varepsilon)} } \, \mathrm{d}t \right)^{ \frac{1+\varepsilon}{2+\varepsilon}}\\
				&\leq C_{\varepsilon,\delta}(1+\xi)^{\frac{1+\delta}{2} } \lambda_j^{-\frac{1+\delta}{2} } b(T)^{-\frac{1+\varepsilon}{2+\varepsilon} } \|b\|_{L^{2+\varepsilon}([0,T])} [(1\ast b)(T)]^{\frac{1+\varepsilon}{2+\varepsilon}-\frac{1+\delta}{2} }.
			\end{aligned}
		\end{align*}
		Thus, there is a positive constant $C_{\varepsilon, \delta, T}$ depending only on
		$\varepsilon$, $\delta$ and $T$, such that
		\begin{align}\label{A.15}
			\int_0^T r_{\frac{\lambda_j}{2(1+\xi)} }(t)\mathrm{d}t \leq C_{\varepsilon,\delta,T}(1+\xi)^{\frac{1+\delta}{2} } \lambda_j^{-\frac{1+\delta}{2}}.
		\end{align}
		The above estimate, together with \eqref{A.14}, gives \eqref{4.9}.
		
		Next, we consider the case $(1+\xi)^2 -4\gamma\lambda_j \leq 0$. Similarly, \eqref{4.1} and Proposition \ref{prop 4.1}(take $\alpha=0 $) give
		\begin{align}\label{A.16}
			\begin{aligned}
				|H_{\gamma,\lambda_j,\xi}(t;0,\gamma^{-1},0)|
				\leq 3\gamma^{-\frac{1}{2} } \lambda_j^{-\frac{1}{2}} r_{ \frac{1+\xi}{4\gamma} }(t), \quad  t>0.
			\end{aligned}
		\end{align}
		Similarly to the derivation of \eqref{A.15}, one can show that
		\begin{align}\label{A.17}
			\int_0^T r_{\frac{1+\xi}{4\gamma} }(t)\mathrm{d}t \leq C_{\varepsilon,\delta,T}(1+\xi)^{-\frac{1+\delta}{2} } \gamma^{\frac{1+\delta}{2}}.
		\end{align}
		This together with \eqref{A.16} yields \eqref{4.11}.
		
		We next establish \eqref{4.10} and \eqref{4.12}. Similarly, we first consider the case $(1+\xi)^2-4\gamma\lambda_j>0$. It follows from \eqref{4.2} and Proposition \ref{prop 4.2} (take $\alpha= 1$) that
		\begin{align}\label{A.18}
			\begin{aligned}
				&\quad|(a\ast H_{\gamma,\lambda_j,\xi}(\cdot;0,\gamma^{-1},0))'(t)|\\
				&\leq 2\sqrt{2}(1+\xi)^{-1} \gamma^{-\frac{1}{2} } \lambda_j^{ \frac{1}{2} } r_{ \frac{\lambda_j}{2(1+\xi)} }(t) + \sqrt{2} \gamma^{-1} r_{ \frac{1+\xi}{4\gamma} }(t).
			\end{aligned}
		\end{align}
		Note that $(1+\xi)^2-4\gamma\lambda_j>0$ shows that $\frac{\gamma}{1+\xi} < \frac{1+\xi}{4\lambda_j}$. Combining this with \eqref{A.15}, \eqref{A.17} and \eqref{A.18} yields
		\begin{align*}
			&\quad \int_0^T |(a \ast H_ {\gamma,\lambda_j,\xi}(\cdot; 0, \gamma^{-1},0))'(t)|\,\mathrm{d}t\\
			&\lesssim (1+\xi)^{-\frac{1-\delta}{2}}\gamma^{ -\frac{1}{2} }  \lambda_j^{-\frac{\delta}{2}}
			+ (1+\xi)^{-\frac{1}{2} } \gamma^{-\frac{1}{2}} \frac{\gamma^{ \frac{\delta}{2} }}{(1+\xi)^{ \frac{\delta}{2}  }  }\\
			&\lesssim (1+\xi)^{-\frac{1-\delta}{2}}\gamma^{ -\frac{1}{2} }  \lambda_j^{-\frac{\delta}{2}},
		\end{align*}
		where the constant in the estimate depends solely on $\varepsilon$, $\delta$ and $T$. Thus \eqref{4.10} is established.
		
		Finally, we consider the case
		$(1+\xi)^2 -4\gamma\lambda_j\leq 0$.
		From \eqref{4.2} and Proposition \ref{prop 4.2}(take $\alpha =0 $) we obtain
		\begin{align}\label{A.19}
			\begin{aligned}
				|(a\ast H_{\gamma,\lambda_j,\xi}(\cdot;0,\gamma^{-1},0))'(t)|
				&\leq 2  \gamma^{-1}    r_{\frac{1+\xi}{4\gamma}}(t).
			\end{aligned}
		\end{align}	
		Combining this with \eqref{A.17} yields \eqref{4.12}.
	\end{proof}
	
	\subsection{Proof of Lemma \ref{lem 4.4}}
	\begin{proof}
		When $(1+\xi)^2-4\gamma\lambda_j\neq0$, from Lemma \ref{lem 3.1} and Proposition \ref{prop 2.1}(i) we have
		\begin{align}\label{A.20}
			\begin{aligned}
				&\quad H_{\gamma,\lambda_j,\xi}(t;h_0,h_1,0)\\
				&=\frac{h_0}{2}(r_{Z_1}(t) +r_{Z_2}(t)) + \frac{h_0}{2\sqrt{(1+\xi)^2-4\gamma\lambda_j} }\int_0^\infty (e^{-Z_1\theta} -e^{-Z_2\theta} )\eta(t,\mathrm{d}\theta) \\
				&\quad +\frac{\gamma h_1}{\sqrt{(1+\xi)^2-4\gamma\lambda_j}}\int_0^\infty (e^{-Z_1\theta} -e^{-Z_2\theta} )\eta(t,\mathrm{d}\theta)\\
				&=\frac{h_0}{2}(r_{Z_1}(t)+r_{Z_2}(t)) + (\frac{1+\xi}{2\gamma}h_0 + h_1)\int_0^\infty g_{\gamma,\lambda_j,\xi}(\theta) \eta(t,\mathrm{d}\theta) \\
				&:=I_1(t)h_0 +I_2(t;h_0,h_1).
			\end{aligned}
		\end{align}
		Also, \eqref{3.14} and Proposition \ref{prop 2.1}(i) give
		\begin{align}\label{A.21}
			\begin{aligned}
				&\quad (a \ast H_{\gamma,\lambda_j,\xi}(\cdot;h_0,h_1,0))'(t)\\
				&=-\gamma^{-1}\lambda_j h_0 \int_0^\infty g_{\gamma,\lambda_j,\xi}(\theta) \eta(t,\mathrm{d}\theta) + h_1\int_0^\infty g'_{\gamma,\lambda_j,\xi}(\theta)\eta(t,\mathrm{d}\theta) \\
				&:=I_3(t)h_0+ I_4(t)h_1.
			\end{aligned}
		\end{align}
		
		We now prove \eqref{4.13}. We first consider the case $(1+\xi)^2-4\gamma\lambda_j \neq 0$. Proposition \ref{prop 2.1}(iv) yields $|I_1(t)| \leq b(t)$. Also, Proposition \ref{prop 4.1} shows that
		\begin{align*}
			|I_2(t;h_0,h_1)|&\leq  (\frac{1+\xi}{2\gamma}|h_0| + |h_1|)\int_0^\infty |g_{\gamma,\lambda_j,\xi}(\theta)| \eta(t,\mathrm{d}\theta)\\
			&\leq (\frac{3}{2}|h_0| +\frac{3\gamma}{1+\xi} |h_1|)b(t).
		\end{align*}
		From \eqref{A.21} we immediately obtain
		$|H_{\gamma,\lambda_j,\xi}(t;h_0,h_1,0)|\leq \frac{5}{2}b(t)|h_0| +\frac{3\gamma}{1+\xi} b(t)|h_1|$. When $1-4\gamma\lambda_j=0$, it follows from Lemma \ref{lem 3.1}(ii) that
		\begin{align}\label{A.22}
			H_{\gamma,\lambda_j,\xi}(t;h_0,h_1,0)=\left[r_{\frac{1+\xi}{2\gamma}}(t) +\frac{1+\xi}{2\gamma}(r_{\frac{1+\xi}{2\gamma}}\ast r_{\frac{1+\xi}{2\gamma}})(t) \right]h_0 +(r_{\frac{1+\xi}{2\gamma}}\ast r_{\frac{1+\xi}{2\gamma}})(t)h_1.
		\end{align}
		By $0<r_{\frac{1+\xi}{2\gamma}}(t) \leq b(t)$ (from Proposition \ref{2.1}(iv)), we obtain
		\begin{align}\label{A.23}
			0<r_{\frac{1+\xi}{2\gamma}}(t) +\frac{1+\xi}{2\gamma}(r_{\frac{1+\xi}{2\gamma}}\ast r_{\frac{1+\xi}{2\gamma}})(t)\leq r_{\frac{1+\xi}{2\gamma}}(t) +\frac{1+\xi}{2\gamma}(b\ast r_{\frac{1+\xi}{2\gamma}})(t)= b(t),
		\end{align}
		which implies $|H_{\gamma,\lambda_j,\xi}(t;h_0,h_1,0)|\leq b(t)|h_0| + \frac{2\gamma}{1+\xi} b(t)|h_1|$.
		
		We now establish \eqref{4.14}. In the case where $(1+\xi)^2 -4\gamma\lambda_j \neq 0$,  it follows from Proposition \ref{prop 4.2}($\alpha =0$) that
		\begin{align}\label{A.24}
			\begin{aligned}
				|I_4(t)| &\leq \int_0^\infty |g'_{\gamma,\lambda_j,\xi}(\theta)|\eta(t,\mathrm{d}\theta)\leq  2\sqrt{2}\int_0^\infty \eta(t,\mathrm{d}\theta)  +\int_0^\infty \sqrt{2}e^{-\frac{1+\xi}{2\gamma} \theta} \eta(r,\mathrm{d}\theta)\\
				&\leq 2(\sqrt{2}+1)b(t).
			\end{aligned}
		\end{align}
		
		If $(1+\xi)^2-4\gamma\lambda_j> 0$,  it follows from Proposition \ref{prop 4.1}($\alpha =1$) that
		\begin{align*}
			|I_3(t)| \leq \gamma^{-1}\lambda_j \int_0^\infty |g_{\gamma,\lambda_j,\xi}(\theta)| \eta(t,\mathrm{d}\theta)\leq \frac{3\lambda_j}{1+\xi} \int_0^\infty \eta(t,\mathrm{d}\theta)\leq \frac{3\lambda_j}{1+\xi}b(t).
		\end{align*}
		This together with \eqref{A.21} and \eqref{A.24} yields
		\begin{align*}
			|(a \ast H_{\gamma,\lambda_j,\xi}(\cdot;h_0,h_1,0))'(t)|\leq\frac{3\lambda_j}{1+\xi}b(t)|h_0| +2(\sqrt{2}+1)b(t)|h_1|.
		\end{align*}
		
		If $(1+\xi)^2-4\gamma\lambda_j= 0$, from \eqref{3.15} we obtain
		\begin{align*}
			&\quad (a \ast H_{\gamma,\lambda_j,\xi}(\cdot;h_0,h_1,0))'(t)\\
			&=-\frac{(1+\xi)^2}{4\gamma^2}(r_{\frac{1+\xi}{2\gamma}}\ast r_{\frac{1+\xi}{2\gamma}})(t)h_0
			+ \left[ r_{\frac{1+\xi}{2\gamma} }(t)-\frac{1+\xi}{2\gamma}(r_{\frac{1+\xi}{2\gamma}}\ast r_{\frac{1+\xi}{2\gamma}})(t) \right]h_1.
		\end{align*}
		Combining the above with \eqref{A.23} shows that
		\begin{align*}
			|(a \ast H_{\gamma,\lambda_j,\xi}(\cdot;0,h_1,0))'(t) |\leq \frac{2\lambda_j}{1+\xi} b(t)|h_0| + b(t) |h_1|.
		\end{align*}
		
		If $(1+\xi)^2 -4\gamma\lambda_j<0$,
		from Proposition \ref{prop 4.1}($\alpha =0$) we have
		\begin{align*}
			|I_3(t)| &\leq \gamma^{-1}\lambda_j \int_0^\infty |g_{\gamma,\lambda_j,\xi}(\theta)| \eta(t,\mathrm{d}\theta)\leq 3 \gamma^{-\frac{1}{2} } \lambda_j^{\frac{1}{2}} \int_0^\infty \eta(t,\mathrm{d}\theta)\leq 3 \gamma^{-\frac{1}{2} } \lambda_j^{\frac{1}{2}}b(t).
		\end{align*}
		The above estimate together with \eqref{A.21} and \eqref{A.24} gives
		\begin{align*}
			|(a \ast H_{\gamma,\lambda_j,\xi}(\cdot;h_0,h_1,0))'(t)|\leq 3 \sqrt{\frac{\lambda_j}{\gamma} } b(t)|h_0| +2(\sqrt{2}+1)b(t)|h_1|.
		\end{align*}
		The proof is completed.
	\end{proof}
	
	\section{Space compactness and time regularity}\label{app B}
	\subsection{Proof of Proposition \ref{prop 5.1}}
	\begin{proof}
		Set $\beta_1 = \frac{\varepsilon}{4(2+\varepsilon)}$ and $\alpha_1=2\beta_1$. Clearly, $\alpha_1$ and $\beta_1$ satisfy \eqref{4.3}. Thus, by Lemmas \ref{lem 5.1}(i) and \ref{lem 5.2}(i), respectively, for any $\delta \in [0,\frac{1}{2})$ and $\xi \geq 0$,
		\begin{align*}
			\begin{aligned}
				&\mathbb E \left[ \int_0^T\|P_{N_{\gamma,\xi} } \mathcal I_1 \Theta_{\gamma,\xi}(t)\|^2_{H^\delta } \, \mathrm{d}t \right] \leq C_{\gamma,\varepsilon, \delta, T}
				\mathbb E \left[ \int_0^T \|\Upsilon(t)\|^2_{\mathcal L(H)}\, \mathrm{d}t \right],\\
				&\mathbb E \left[ \int_0^T\|P_{N_{\gamma,\xi} } \mathcal I_2 \Theta_{\gamma,\xi}(t)\|^2_{H^{\delta-1} } \, \mathrm{d}t \right]\leq C_{\gamma, \varepsilon,\delta, T}
				\mathbb E \left[ \int_0^T \|\Upsilon(t)\|^2_{\mathcal L(H)}\, \mathrm{d}t \right].
			\end{aligned}
		\end{align*}
		We next take $\alpha_2=0$ and  $\beta_2 = \frac{\varepsilon}{4(2+\varepsilon)}$ . One readily verifies that $\alpha_2$ and $\beta_2$ satisfy \eqref{4.6}.  Applying Lemma \ref{lem 5.1}(ii) and Lemma \ref{lem 5.2}(ii), respectively, we have that for any $\delta \in [0,\frac{1}{2})$ and $\xi \geq 0$,
		\begin{align*}
			&\mathbb E \left[ \int_0^T\|(I-P_{N_{\gamma,\xi} })\mathcal I_1 \Theta_{\gamma,\xi}(t)\|^2_{H^\delta } \, \mathrm{d}t \right] \leq C_{\gamma,\varepsilon, \delta,T}
			\mathbb E \left[ \int_0^T \|\Upsilon(t)\|^2_{\mathcal L(H)}\, \mathrm{d}t \right],\\
			&\mathbb E \left[ \int_0^T\|(I-P_{N_{\gamma,\xi} })\mathcal I_2 \Theta_{\gamma,\xi}(t)\|^2_{H^{\delta-1} } \, \mathrm{d}t \right] \leq C_{\gamma, \varepsilon, \delta, T}
			\mathbb E \left[ \int_0^T \|\Upsilon(t)\|^2_{\mathcal L(H)}\, \mathrm{d}t \right].
		\end{align*}
		The above four inequalities, together with the linear growth condition for $\Phi$ in $(\mathcal M_4)$,  immediately give \eqref{5.15}.
	\end{proof}
	
	\subsection{Proof of Proposition \ref{prop 5.2}}
	\begin{proof}
		We first estimate the displacement component. By Fubini's theorem and It\^o isometry, we obtain
		\begin{align}\label{B.1}
			\begin{aligned}
				&\quad \mathbb E \left[ \int_0^{T-\varrho}\|\mathcal I_1 \Theta_{\gamma,\xi}(t+\varrho)-\mathcal I_1\Theta_{\gamma,\xi}(t)\|^2_{H}\, \mathrm{d}t \right] \\
				&\lesssim \int_0^{T-\varrho}\mathbb E \left[ \left\| \int_0^{t}   \left(\mathcal I_1 \mathcal R_{\gamma,\xi}(t+\varrho-\zeta)-\mathcal I_1 \mathcal R_{\gamma,\xi}(t-\zeta) \right) J_\gamma \Upsilon(\zeta)\, \mathrm{d}W(\zeta) \right \|^2_{H} \right] \, \mathrm{d}t  \\
				&\quad+  \int_0^{T-\varrho}  \mathbb E \left[ \left\|\int_{t}^{t+\varrho} \mathcal I_1 \mathcal R_{\gamma,\xi}(t+\rho-\zeta)  J_\gamma \Upsilon(\zeta)\, \mathrm{d}W(\zeta)    \right \|^2_{H}\right]\, \mathrm{d}t  \\
				&=\int_0^{T-\varrho} \mathbb E \left[ \int_0^{t} \left\|(\mathcal I_1 \mathcal R_{\gamma,\xi}(t+\varrho-\zeta) - \mathcal I_1 \mathcal R_{\gamma,\xi}(t-\zeta)) J_\gamma \Upsilon(\zeta) \right\|^2_{\mathcal L_2(H,H)}  \, \mathrm{d}\zeta \right]\,\mathrm{d}t\\
				&\quad +\int_0^{T-\varrho} \mathbb E \left[  \int_{t}^{t+\varrho} \left\| \mathcal I_1 \mathcal R_{\gamma,\xi}(t+\varrho-\zeta)  J_\gamma \Upsilon(\zeta)\,  \right \|^2_{\mathcal L^2(H,H)} \mathrm{d}\zeta  \right]\, \mathrm{d}t  \\
				&=\int_0^{T-\varrho}\mathbb E \left[ \sum_{j=1}^{\infty } \int_0^{t} \left \|\left(\mathcal I_1 \mathcal R_{\gamma,\xi}(t+\varrho-\zeta) - \mathcal I_1 \mathcal R_{\gamma,\xi}(t-\zeta) \right) J_\gamma \Upsilon(\zeta)e_j \right \|^2_{H}  \, \mathrm{d}\zeta \right]\, \mathrm{d}t\\
				&\quad+ \int_0^{T-\varrho} \mathbb E \left[ \sum_{j=1}^{\infty } \int_{t}^{t+\varrho} \left \|\mathcal I_1 \mathcal R_{\gamma,\xi}(t+\varrho-\zeta)J_\gamma \Upsilon(\zeta)e_j \right\|^2_{H}  \, \mathrm{d}\zeta \right] \,\mathrm{d}t\\
				&:=\mathbb E\left[\int_0^{T-\varrho}  \mathcal G_{5,\gamma,\xi}(t)  \, \mathrm{d}t\right] +\mathbb E\left[\int_0^{T-\varrho}  \mathcal G_{6,\gamma,\xi}(t)  \, \mathrm{d}t\right] .
			\end{aligned}
		\end{align}
		Similarly to the derivation of \eqref{5.4}, we get
		\begin{align*}
			&\quad \mathcal G_{5,\gamma,\xi}(t)\\
			&=\sum_{i=1}^{ \infty } \sum_{j=1}^\infty \int_0^{t}  |H_{\gamma,\lambda_i,\xi}(t+\varrho-\zeta;0,\gamma^{-1},0)-H_{\gamma,\lambda_i,\xi}(t-\zeta;0,\gamma^{-1},0)|^2 \left \langle \Upsilon(\zeta)e_j, e_i \right \rangle^2\, \mathrm{d}\zeta,\\
			&\mathcal G_{6,\gamma,\xi}(t) = \sum_{i=1}^{ \infty } \sum_{j=1}^\infty \int_t^{t+\varrho} |H_{\gamma,\lambda_i,\xi}(t+\varrho-\zeta;0,\gamma^{-1},0)|^2 \left \langle \Upsilon(\zeta)e_j, e_i \right \rangle^2  \, \mathrm{d}\zeta.
		\end{align*}
		Combining each of the above two estimates with \eqref{5.5}, respectively, yields
		\begin{align*}
			&\quad \mathcal G_{5,\gamma,\xi}(t)\\
			&\leq \int_0^{t} \|\Upsilon(\zeta)\|^2_{\mathcal L(H)}  \sum_{i=1}^{ \infty } |H_{\gamma,\lambda_i,\xi}(t+\varrho-\zeta;0,\gamma^{-1},0)-H_{\gamma,\lambda_i,\xi}(t-\zeta;0,\gamma^{-1},0)|^2\, \mathrm{d}\zeta,\\
			&G_{6,\gamma,\xi}(t)
			\leq \int_t^{t+\varrho} \|\Upsilon(\zeta)\|^2_{\mathcal L(H)} \sum_{i=1}^{ \infty } |H_{\gamma,\lambda_i,\xi}(t+\varrho-\zeta;0,\gamma^{-1},0)|^2\, \mathrm{d}\zeta.
		\end{align*}
		Moreover, by an argument similar to the derivation of \eqref{5.6}, we obtain
		\begin{align}\label{B.2}
			\begin{aligned}
				\int_0^{T-\varrho} \mathcal G_{5,\gamma,\xi}(t)\, \mathrm{d}t
				&\leq \int_0^{T-\varrho} \|\Upsilon(t)\|^2_{\mathcal L(H)}\, \mathrm{d}t \\
				&\quad \times \sum_{i=1}^{ \infty } \int_{0}^{T}  |H_{\gamma,\lambda_i,\xi}(t+\varrho;0,\gamma^{-1},0)-H_{\gamma,\lambda_i,\xi}(t;0,\gamma^{-1},0)|^2\, \mathrm{d}t,
			\end{aligned}
		\end{align}
		and
		\begin{align}\label{B.3}
			\begin{aligned}
				&\quad \int_0^{T-\varrho} \mathcal G_{6,\gamma,\xi}(t)\, \mathrm{d}t\\
				&\leq \int_0^{\varrho} \|\Upsilon(\zeta)\|^2_{\mathcal L(H)}  \int_{0}^{\zeta}  \sum_{i=1}^{ \infty }  |H_{\gamma,\lambda_i,\xi}(t+\varrho-\zeta;0,\gamma^{-1},0)|^2\, \mathrm{d}t \mathrm{d}\zeta\\
				&\quad + \int_{\varrho}^{T-\varrho} \|\Upsilon(\zeta)\|^2_{\mathcal L(H)} \int_{\zeta-\varrho}^{\zeta}  \sum_{i=1}^{ \infty }  |H_{\gamma,\lambda_i,\xi}(t+\varrho-\zeta;0,\gamma^{-1},0)|^2\, \mathrm{d}t\mathrm{d}\zeta\\
				&\quad + \int_{T-\varrho}^T \|\Upsilon(\zeta)\|^2_{\mathcal L(H)}  \int_{\zeta-\varrho}^{T-\varrho}  \sum_{i=1}^{ \infty }  |H_{\gamma,\lambda_i,\xi}(t+\varrho-\zeta;0,\gamma^{-1},0)|^2\, \mathrm{d}t\mathrm{d}\zeta\\
				&\leq \int_0^{T} \|\Upsilon(t)\|^2_{\mathcal L(H)}  \mathrm{d}t \times \sum_{i=1}^{ \infty } \int_{0}^{\varrho}   |H_{\gamma,\lambda_i,\xi}(t;0,\gamma^{-1},0)|^2\, \mathrm{d}t.
			\end{aligned}
		\end{align}
		
		We begin with the treatment of \eqref{B.2}. Choose the parameters
		$\beta_1=\beta_2 \in (0,\frac{\varepsilon}{2(2+\varepsilon)})$, $\alpha_1=2\beta_1$ and $\alpha_2=0$. By \eqref{4.4} and \eqref{4.7}, it follows that
		$H_{\gamma,\lambda_i,\xi}(\cdot;0,\gamma^{-1},0) \in L^2_{loc}(\mathbb R_+)$ for each $\xi\geq 0$ and each $\lambda_i$. Consequently, a classical density argument yields
		\begin{align*}
			\int_{0}^{T}|H_{\gamma,\lambda_i,\xi}(t+\varrho;0,\gamma^{-1},0)-H_{\gamma,\lambda_i,\xi}(t;0,\gamma^{-1},0)|^2\, \mathrm{d}t \to 0 \text{ ~as~ } \varrho \to 0^+
		\end{align*}
		for each $\xi\geq 0$ and each $\lambda_i$. Moreover, from \eqref{4.4} and \eqref{4.7}, we also have
		\begin{align*}
			&\int_0^T |H_{\gamma,\lambda_i,\xi}(t;0,\gamma^{-1},0)|^2\, \mathrm{d}t
			\leq C_{\gamma,T}  \lambda_i^{-1},\\
			&\int_0^{T} |H_{\gamma,\lambda_i,\xi}(t+\varrho;0,\gamma^{-1},0)|^2\, \mathrm{d}t\leq \int_0^{T+1} |H_{\gamma,\lambda_i,\xi}(t;0,\gamma^{-1},0)|^2\, \mathrm{d}t\leq C_{\gamma,T}\lambda_i^{-1}
		\end{align*}
		for each $\varrho\in(0,1)$, each $\xi \geq 0$ and each $\lambda_i$. This implies
		\begin{align*}
			\int_{0}^{T}|H_{\gamma,\lambda_i,\xi}(t+\varrho;0,\gamma^{-1},0)-H_{\gamma,\lambda_i,\xi}(t;0,\gamma^{-1},0)|^2\, \mathrm{d}t
		\end{align*}
		can be bounded by $\lambda_i^{-1}$.
		Then, applying the dominated convergence theorem yields
		\begin{align}\label{B.4}
			\sum_{i=1}^{ \infty }  \int_{0}^{T}|H_{\gamma,\lambda_i,\xi}(t+\varrho;0,\gamma^{-1},0)-H_{\gamma,\lambda_i,\xi}(t;0,\gamma^{-1},0)|^2\, \mathrm{d}t \to 0 \text{ ~as~ } \varrho \to 0^+.
		\end{align}
		
		Now turn to \eqref{B.3}. Take $\beta_1=\beta_2 \in (0,\frac{\varepsilon}{2(2+\varepsilon)})$, $\alpha_1=2\beta_1$, $\alpha_2=0$. By \eqref{4.4} and \eqref{4.7},
		\begin{align*}
			\lim_{\varrho \to 0^+} \int_{0}^{\varrho}   |H_{\gamma,\lambda_i,\xi}(t;0,\gamma^{-1},0)|^2\, \mathrm{d}t \to 0 \text{ ~as~ } \varrho \to 0^+.
		\end{align*}
		Applying the dominated convergence theorem again yields
		\begin{align}\label{B.5}
			\sum_{i=1}^{ \infty } \int_{0}^{\varrho}   |H_{\gamma,\lambda_i,\xi}(t;0,\gamma^{-1},0)|^2\, \mathrm{d}t \to 0 \text{ ~as~ } \varrho \to 0^+
		\end{align}
		
		From \eqref{B.1}, \eqref{B.2}, \eqref{B.3}, \eqref{B.4}, \eqref{B.5}, together with the linear growth condition on $\Phi$ in $(\mathcal M_4)$, it follows that for any $\xi \geq 0$,
		\begin{align}\label{B.6}
			\mathbb E\left[ \int_0^{T-\varrho} \|\mathcal I_1 \Theta_{\gamma,\xi}(t+\varrho) -\mathcal I_1\Theta_{\gamma,\xi}(t)\|_{H} \, \mathrm{d}t \right] \leq C_{\gamma,T} o(1) (1+\mathbb E\left[ \|\mathcal I_1 Y\|_{L^2([0,T];H)}^2 \right]).
		\end{align}
		
		Next, we estimate the velocity component. Similarly, by Fubini's theorem and It\^o isometry, we have
		\begin{align*}
			\begin{aligned}
				\mathbb E \left[ \int_0^{T-\varrho}\|\mathcal I_2 \Theta_{\gamma,\xi}(t+\varrho)-\mathcal I_1\Theta_{\gamma,\xi}(t)\|^2_{H}\, \mathrm{d}t \right]
				\leq \mathbb E\left[\int_0^{T-\varrho}  \mathcal G_{7,\gamma,\xi}(t)  \, \mathrm{d}t\right] +\mathbb E\left[\int_0^{T-\varrho}  \mathcal G_{8,\gamma,\xi}(t)  \, \mathrm{d}t\right],
			\end{aligned}
		\end{align*}
		where
		\begin{align*}
			\begin{aligned}
				&\quad \int_0^{T-\varrho} \mathcal G_{7,\gamma,\xi}(t)\, \mathrm{d}t\\
				&\lesssim \int_0^{T-\varrho} \|\Upsilon(t)\|^2_{\mathcal L(H)}\, \mathrm{d}t \\
				&\quad \times \sum_{i=1}^{ \infty } \lambda_i^{-1} \int_{0}^{T}  |(a\ast H_{\gamma,\lambda_i,\xi}(\cdot;0,\gamma^{-1},0))'(t+\varrho)-(a\ast H_{\gamma,\lambda_i,\xi}(\cdot;0,\gamma^{-1},0))'(t)|^2\, \mathrm{d}t,
			\end{aligned}
		\end{align*}
		and
		\begin{align*}
			\int_0^{T-\varrho} \mathcal G_{8,\gamma,\xi}(t)\, \mathrm{d}t
			\lesssim \int_0^{T} \|\Upsilon(t)\|^2_{\mathcal L(H)}  \mathrm{d}t \times \sum_{i=1}^{ \infty } \lambda_i^{-1}\int_{0}^{\varrho}   |(a \ast H_{\gamma,\lambda_i,\xi}(\cdot;0,\gamma^{-1},0))'(t)|^2\, \mathrm{d}t.
		\end{align*}
		By the same argument as in the derivation of \eqref{B.6}, it suffices to show that for any $\xi \geq 0$,
		\begin{align}\label{B.7}
			\mathbb E\left[ \int_0^{T-\varrho} \|\mathcal I_2 \Theta_{\gamma,\xi}(t+\varrho) -\mathcal I_2\Theta_{\gamma,\xi}(t)\|_{H^{-1}} \, \mathrm{d}t \right] \leq C_{\gamma,T} o(1) \mathbb E\left[ \int_0^T\|\mathcal I_1 Y(t)\|_{H}^2\, \mathrm{d}t \right].
		\end{align}
		The derivation of \eqref{B.7} differs from that of \eqref{B.6} only in that \eqref{4.5} and \eqref{4.8} are employed in place of the corresponding estimates. We therefore omit the details. Finally, \eqref{5.14} follows immediately from \eqref{B.6} and \eqref{B.7}.
	\end{proof}
	
	\section*{Statements and Declarations}
	\noindent{\bf Competing interests}\\
	The authors declare that there is no conflict of interests regarding the publication of this paper.\\
	\noindent{\bf Funding}\\
	This work was supported by National Natural Science Foundation of China (12471127).\\
	\noindent{\bf Availability of data and materials}\\
	Not applicable.

\end{document}